\documentclass[reqno]{amsart}

\usepackage[utf8]{inputenc}
\usepackage{amsmath,bm}
\usepackage{amssymb}
\usepackage{hyperref}
\usepackage{verbatim}
\usepackage{dsfont}
\usepackage{enumerate}
\usepackage[all]{xy}
\usepackage[mathscr]{eucal}
\usepackage[usenames]{color}
\usepackage{amsthm}
\usepackage{stmaryrd}
\usepackage{amsfonts}
\usepackage{latexsym}
\usepackage{amscd}
\usepackage{graphicx}
\usepackage{setspace}
\usepackage{wasysym}
\usepackage[usenames,dvipsnames]{xcolor}
\usepackage{tikz}
\usepackage{tikz-cd}
\usetikzlibrary{decorations.pathmorphing}
\usepackage{todonotes}

\usepackage[capitalize]{cleveref}
\crefname{theorem}{Theorem}{Theorems}
\crefname{fact}{Fact}{Facts}
\crefname{note}{Note}{Notes}
\crefname{lemma}{Lemma}{Lemmas}
\crefname{alg}{Algorithm}{Algorithms}
\crefname{remark}{Remark}{Remarks}
\crefname{example}{Example}{Examples}
\crefname{prop}{Proposition}{Propositions}
\crefname{conj}{Conjecture}{Conjectures}
\crefname{cor}{Corollary}{Corollaries}
\crefname{defn}{Definition}{Definitions}
\crefname{equation}{\!\!}{\!\!} %Remove spacing around phantom equation name

\makeatletter

\newcommand{\bd}[1]{\mathbf{#1}}

\newtheorem{theorem}{Theorem}[section]

\newtheorem{lemma}[theorem]{Lemma}
\newtheorem{proposition}[theorem]{Proposition}
\newtheorem{corollary}[theorem]{Corollary} 
\theoremstyle{definition}  
\newtheorem{definition}[theorem]{Definition}

\newtheorem{remark}[theorem]{Remark}

\@addtoreset{equation}{section}

\newcommand{\END}{\mathcal{E}nd}
\newcommand{\End}{\operatorname{End}}
 
\newcommand{\Hom}{\operatorname{Hom}}

\newcommand{\gl}{\mathfrak{gl}}

\newcommand{\Z}{\mathbb{Z}}

\newcommand{\C}{\mathbb{C}}
\newcommand{\K}{\mathbb{K}}

\newcommand{\cat}{\mathcal}
\newcommand{\ep}{\varepsilon}
\newcommand{\arup}[1]{\stackrel{#1}{\longrightarrow}}
\newcommand{\larup}[1]{\stackrel{#1}{\longleftarrow}}

\newcommand{\up}{\uparrow}
\newcommand{\down}{\downarrow}
\newcommand{\ob}[1]{\mathsf{#1}}
\newcommand{\wrd}{\langle\up,\down\rangle}

\newcommand{\gr}{\operatorname{gr}}

\newcommand{\unit}{\mathds{1}}

\newcommand{\0}{\bar{0}}
\renewcommand{\1}{\bar{1}}

\newcommand{\ibar}{{\bar{i}}}
\newcommand{\jbar}{{\bar{j}}}

\newcommand{\OBC}{\mathcal{OBC}}
\newcommand{\AOBC}{\mathcal{AOBC}}

\newcommand{\OB}{\mathcal{OB}}
\newcommand{\AOB}{\mathcal{AOB}}

\renewcommand{\k}{\Bbbk}
\newcommand{\ev}{\operatorname{ev}}
\newcommand{\coev}{\operatorname{coev}}

\newcommand{\fg}{\mathfrak{g}}
\newcommand{\fq}{\mathfrak{q}}
\newcommand{\fh}{\mathfrak{h}}
\newcommand{\fn}{\mathfrak{n}}
\newcommand{\fb}{\mathfrak{b}}
\newcommand{\Id}{\operatorname{Id}}

\newcommand{\chr}{\operatorname{char}}

\newcommand{\scat}{\mathfrak{scat}}
\newcommand{\svec}{\mathfrak{svec}}
\newcommand{\usvec}{{\underline{\mathfrak{svec}}}}
\newcommand{\gsvec}{\mathfrak{gsvec}}
\newcommand{\ugsvec}{{\underline{\mathfrak{gsvec}}}}
\newcommand{\fsvec}{\mathfrak{fsvec}}
\newcommand{\ufsvec}{{\underline{\mathfrak{fsvec}}}}

\newcommand{\gsmod}{\fg\operatorname{-smod}}

\newcommand{\hw}{\hat{u}}

\newcommand{\clr}{rgb:black,1;blue,4;red,1}

\newcommand{\wdot}{ node[circle, draw, fill=white, thick, inner sep=0pt, minimum width=4pt]{}}
\newcommand{\bdot}{ node[circle, draw, fill=\clr, thick, inner sep=0pt, minimum width=4pt]{}}

\newcommand{\fulldot}{
    \begin{tikzpicture}[color=\clr]
    \draw (0,0) \bdot;
    \end{tikzpicture}
}
\newcommand{\emptydot}{
    \begin{tikzpicture}[color=\clr]
    \draw (0,0) \wdot;
    \end{tikzpicture}
}
\newcommand{\undot}[1]{\operatorname{undot}({#1})}
\newcommand{\p}[1]{|{#1}|}

\newcommand{\lcap}{
\begin{tikzpicture}[baseline = 3pt, scale=0.5, color=\clr]
        \draw[-,thick] (1,0) to[out=up, in=right] (0.53,0.5) to[out=left, in=right] (0.47,0.5);
        \draw[->,thick] (0.49,0.5) to[out=left,in=up] (0,0);
\end{tikzpicture}
}
\newcommand{\lcup}{
\begin{tikzpicture}[baseline = 6pt, scale=0.5, color=\clr]
        \draw[-,thick] (1,1) to[out=down, in=right] (0.53,0.5) to[out=left, in=right] (0.47,0.5);
        \draw[->,thick] (0.49,0.5) to[out=left,in=down] (0,1);
\end{tikzpicture}
}
\newcommand{\rcap}{
\begin{tikzpicture}[baseline = 3pt, scale=0.5, color=\clr]
        \draw[<-,thick] (1,0) to[out=up, in=right] (0.53,0.5) to[out=left, in=right] (0.47,0.5);
        \draw[-,thick] (0.49,0.5) to[out=left,in=up] (0,0);
\end{tikzpicture}
}
\newcommand{\rcup}{
\begin{tikzpicture}[baseline = 6pt, scale=0.5, color=\clr]
        \draw[<-,thick] (1,1) to[out=down, in=right] (0.53,0.5) to[out=left, in=right] (0.47,0.5);
        \draw[-,thick] (0.49,0.5) to[out=left,in=down] (0,1);
\end{tikzpicture}
}
\newcommand{\swap}{
\begin{tikzpicture}[baseline = 3pt, scale=0.5, color=\clr]
        \draw[->,thick] (0,0) to[out=up, in=down] (1,1);
        \draw[->,thick] (1,0) to[out=up, in=down] (0,1);
\end{tikzpicture}
}
\newcommand{\rswap}{
\begin{tikzpicture}[baseline = 3pt, scale=0.5, color=\clr]
        \draw[->,thick] (0,0) to[out=up, in=down] (1,1);
        \draw[<-,thick] (1,0) to[out=up, in=down] (0,1);
\end{tikzpicture}
}
\newcommand{\lswap}{
\begin{tikzpicture}[baseline = 3pt, scale=0.5, color=\clr]
        \draw[<-,thick] (0,0) to[out=up, in=down] (1,1);
        \draw[->,thick] (1,0) to[out=up, in=down] (0,1);
\end{tikzpicture}
}
\newcommand{\dswap}{
\begin{tikzpicture}[baseline = 3pt, scale=0.5, color=\clr]
        \draw[<-,thick] (0,0) to[out=up, in=down] (1,1);
        \draw[<-,thick] (1,0) to[out=up, in=down] (0,1);
\end{tikzpicture}
}
\newcommand{\xdot}{
\begin{tikzpicture}[baseline = 3pt, scale=0.5, color=\clr]
        \draw[->,thick] (0,0) to[out=up, in=down] (0,1);
        \draw (0,0.4)\bdot;
\end{tikzpicture}
}
\newcommand{\downxdotk}{
\begin{tikzpicture}[baseline = 3pt, scale=0.5, color=\clr]
        \draw[->,thick] (0,0) to[out=up, in=down] (0,1);
        \draw[<-,thick] (-0.5,0) to[out=up, in=down] (-0.5,1);
        \draw (0,0.4)\bdot;
        \draw (0.5,0.4) node{$k$};
\end{tikzpicture}
}
\newcommand{\cldot}{
\begin{tikzpicture}[baseline = 3pt, scale=0.5, color=\clr]
        \draw[->,thick] (0,0) to[out=up, in=down] (0,1);
        \draw (0,0.4)\wdot;
\end{tikzpicture}
}
\newcommand{\downxdot}{
\begin{tikzpicture}[baseline = 3pt, scale=0.5, color=\clr]
        \draw[<-,thick] (0,0) to[out=up, in=down] (0,1);
        \draw (0,0.6)\bdot;
\end{tikzpicture}
}

\allowdisplaybreaks

\newcommand{\spinOB}{{\mathcal{OB}_{\ep}}}
\newcommand{\spinOBC}{{\mathcal{OBC}_{\ep}}}
\newcommand{\spinAOBC}{{\mathcal{AOBC}_{\ep}}}
\newcommand{\spinAOB}{{\mathcal{AOB}_{\ep}}}
\newcommand{\Uksmod}{U(\fq )\operatorname{-smod}}
\newcommand{\EndUksmod}{\END(U(\fq )\operatorname{-smod})}

\newcommand{\sgn}{\operatorname{sgn}}

\newcommand{\cldotdown}{
\begin{tikzpicture}[baseline = 3pt, scale=0.5, color=\clr]
        \draw[<-,thick] (0,0) to[out=up, in=down] (0,1);
        \draw (0,0.6)\wdot;
\end{tikzpicture}}

\newcommand{\clrtwo}{rgb:black,1;blue,1;red,4}

\newcommand{\clrgreen}{rgb:black,2;green,3}
\newcommand{\rdot}{ node[circle, draw, fill=\clrtwo, thick, inner sep=0pt, minimum width=4pt]{}}

\newcommand{\fullreddot}{
    \begin{tikzpicture}[color=\clrtwo]
    \draw (0,0) \rdot;
    \end{tikzpicture}
}
\newcommand{\emptyreddot}{
    \begin{tikzpicture}[color=\clrtwo]
    \draw (0,0) \wdot;
    \end{tikzpicture}
}
\newcommand{\lcapred}{
\begin{tikzpicture}[baseline = 3pt, scale=0.5, color=\clrtwo]
        \draw[-,thick] (1,0) to[out=up, in=right] (0.53,0.5) to[out=left, in=right] (0.47,0.5);
        \draw[->,thick] (0.49,0.5) to[out=left,in=up] (0,0);
\end{tikzpicture}
}
\newcommand{\lcupred}{
\begin{tikzpicture}[baseline = 6pt, scale=0.5, color=\clrtwo]
        \draw[-,thick] (1,1) to[out=down, in=right] (0.53,0.5) to[out=left, in=right] (0.47,0.5);
        \draw[->,thick] (0.49,0.5) to[out=left,in=down] (0,1);
\end{tikzpicture}
}
\newcommand{\rcapred}{
\begin{tikzpicture}[baseline = 3pt, scale=0.5, color=\clrtwo]
        \draw[<-,thick] (1,0) to[out=up, in=right] (0.53,0.5) to[out=left, in=right] (0.47,0.5);
        \draw[-,thick] (0.49,0.5) to[out=left,in=up] (0,0);
\end{tikzpicture}
}
\newcommand{\rcupred}{
\begin{tikzpicture}[baseline = 6pt, scale=0.5, color=\clrtwo]
        \draw[<-,thick] (1,1) to[out=down, in=right] (0.53,0.5) to[out=left, in=right] (0.47,0.5);
        \draw[-,thick] (0.49,0.5) to[out=left,in=down] (0,1);
\end{tikzpicture}
}
\newcommand{\swapred}{
\begin{tikzpicture}[baseline = 3pt, scale=0.5, color=\clrtwo]
        \draw[-,thick] (0,0) to[out=up, in=down] (1,0.85);
        \draw[-,thick] (1,0) to[out=up, in=down] (0,0.85);
        \draw[->,thick] (1,0.8) to[out=up, in=down] (1,1);
        \draw[->,thick] (0,0.8) to[out=up, in=down] (0,1);
\end{tikzpicture}
}
\newcommand{\rswapred}{
\begin{tikzpicture}[baseline = 3pt, scale=0.5, color=\clrtwo]
        \draw[->,thick] (0,0) to[out=up, in=down] (1,1);
        \draw[->,thick] (0,1) to[out=down, in=up] (1,0);
\end{tikzpicture}
}
\newcommand{\lswapred}{
\begin{tikzpicture}[baseline = 3pt, scale=0.5, color=\clrtwo]
        \draw[<-,thick] (0,0) to[out=up, in=down] (1,1);
        \draw[->,thick] (1,0) to[out=up, in=down] (0,1);
\end{tikzpicture}
}
\newcommand{\dswapred}{
\begin{tikzpicture}[baseline = 3pt, scale=0.5, color=\clrtwo]
        \draw[<-,thick] (0,0) to[out=up, in=down] (1,1);
        \draw[<-,thick] (1,0) to[out=up, in=down] (0,1);
\end{tikzpicture}
}
\newcommand{\xdotred}{
\begin{tikzpicture}[baseline = 3pt, scale=0.5, color=\clrtwo]
        \draw[->,thick] (0,0) to[out=up, in=down] (0,1);
        \draw (0,0.4)\rdot;
\end{tikzpicture}
}
\newcommand{\downxdotred}{
\begin{tikzpicture}[baseline = 3pt, scale=0.5, color=\clrtwo]
        \draw[<-,thick] (0,0) to[out=up, in=down] (0,1);
        \draw (0,0.6)\rdot;
\end{tikzpicture}
}
\newcommand{\downcldotred}{
\begin{tikzpicture}[baseline = 3pt, scale=0.5, color=\clrtwo]
        \draw[<-,thick] (0,0) to[out=up, in=down] (0,1);
        \draw (0,0.6)\wdot;
\end{tikzpicture}
}

\newcommand{\cldotred}{
\begin{tikzpicture}[baseline = 3pt, scale=0.5, color=\clrtwo]
        \draw[->,thick] (0,0) to[out=up, in=down] (0,1);
        \draw (0,0.4)\wdot;
\end{tikzpicture}
}
\newcommand{\capgreen}{
\begin{tikzpicture}[baseline = 3pt, scale=0.5, color=\clrgreen]
        \draw[-,thick] (1,0) to[out=up, in=right] (0.53,0.5) to[out=left, in=right] (0.47,0.5);
        \draw[-,thick] (0.49,0.5) to[out=left,in=up] (0,0);
\end{tikzpicture}
}
\newcommand{\cupgreen}{
\begin{tikzpicture}[baseline = 6pt, scale=0.5, color=\clrgreen]
        \draw[-,thick] (1,1) to[out=down, in=right] (0.53,0.5) to[out=left, in=right] (0.47,0.5);
        \draw[-,thick] (0.49,0.5) to[out=left,in=down] (0,1);
\end{tikzpicture}
}
\newcommand{\swapgreen}{
\begin{tikzpicture}[baseline = 3pt, scale=0.5, color=\clrgreen]
        \draw[-,thick] (0,0) to[out=up, in=down] (1,1);
        \draw[-,thick] (1,0) to[out=up, in=down] (0,1);
\end{tikzpicture}
}

\definecolor{darkblue}{HTML}{111199}
\definecolor{darkgreen}{HTML}{336633}
\definecolor{darkred}{HTML}{993333}
\definecolor{darkpurple}{HTML}{995599}

\def\AS{\operatorname{ASerg}}
\def\Serg{\operatorname{Serg}}
\def\Sym{\operatorname{Sym}}

\begin{document}

\title[A basis theorem for $\AOBC$]{A basis theorem for the degenerate affine oriented Brauer-Clifford supercategory}

\author[Brundan]{Jonathan Brundan}
\address{J.B.: Department of Mathematics, University of Oregon}
\thanks{Research of the first author was partially supported by NSF grant DMS-1700905}\
\email{brundan@uoregon.edu}

\author[Comes]{Jonathan Comes}
\address{J.C.: Department of Mathematics \& Physical Sciences, The College of Idaho}
\email{jonnycomes@gmail.com}

\author[Kujawa]{Jonathan R. Kujawa}
\address{J.K.: Department of Mathematics, University of Oklahoma}
\thanks{Research of the third author was partially supported by NSA grant H98230-11-1-0127}\
\email{kujawa@math.ou.edu}

\date{\today}

%\subjclass[2010]{17B10, 18D10}

\begin{abstract}
We introduce the oriented Brauer-Clifford and degenerate affine oriented Brauer-Clifford supercategories.  These are diagrammatically defined monoidal supercategories which provide combinatorial models for certain natural monoidal supercategories of supermodules and endosuperfunctors, respectively, for the Lie superalgebras of type Q. Our main results are basis theorems for these diagram supercategories.  We also discuss connections and applications to the representation theory of the Lie superalgebra of type Q.
\end{abstract}

\maketitle

%\tableofcontents

\section{Introduction} 

\subsection{Overview}  

Let $\k$ denote a fixed ground field\footnote{More generally, one could take $\k$ to be a commutative ring.  See \cref{SS:PositiveCharacteristic}.} of characteristic not two.  In this paper we study certain monoidal supercategories; that is, categories in which morphisms form $\Z_{2}$-graded $\k$-vector spaces, the category has a tensor product, and compositions and tensor products of morphisms are related by a graded version of the interchange law (see \cref{section: monoidal supercategories} for more details). While enriched monoidal categories have been the object of study for some time, it is only recently that monoidal supercategories have taken on a newfound importance thanks to the role they play in higher representation theory.  To name a few examples, they appear explicitly or implicitly in the categorification of Heisenberg algebras \cite{RS}, ``odd'' categorifications of Kac-Moody (super)algebras (e.g.\ \cite{EL, KKO, KKO2}), the definition of super Kac-Moody categories \cite{BE2}, and in various Schur-Weyl dualities in the $\Z_{2}$-graded setting (e.g.\ \cite{KT}).

In this paper we introduce two monoidal supercategories.  They are the \emph{oriented Brauer-Clifford supercategory} $\OBC$ and the \emph{degenerate affine oriented Brauer-Clifford supercategory} $\AOBC$.  They are defined by generators and relations. For both monoidal supercategories the generating objects are $\up$ and $\down$. Hence, objects in both $\OBC$ and $\AOBC$ can be viewed as finite words in $\up$ and $\down$ (we write $\unit$ for the unit object). 
For $\OBC$ the generating morphisms are the three even morphisms $\lcup:\unit\to\up\down$, $\lcap: \down\up\to\unit$, $\swap: \up\up\to\up\up$, and one odd morphism $\cldot:\up\to\up$.  For $\AOBC$ the generating morphisms are the those of $\OBC$ along with an even morphism $\xdot:\up\to\up$. These generating morphisms are subject to an explicit list of local relations (see \cref{OBC defn,AOBC defn} for details). In \cref{section: AOBC} we explain how more complicated diagrams can be interpreted as morphisms in $\OBC$ and $\AOBC$. For example, here are two diagrams which correspond to morphisms in $\Hom_\AOBC(\down^2\up^2,\down^2\up^3\down)$:
\begin{equation}\label{two AOBC diagrams}
    \begin{tikzpicture}[baseline = 25pt, scale=0.35, color=\clr]
        \draw[<-,thick] (2,0) to[out=up,in=down] (0,5);
        \draw[->,thick] (6,0) to[out=up,in=down] (6,5);
        \draw[<-,thick] (0,0) to[out=up,in=left] (2,1.5) to[out=right,in=up] (4,0);
        \draw[->,thick] (2,5) to[out=down,in=left] (3,4) to[out=right,in=down] (4,5);
        \draw[->,thick] (2,3.1) to (2,3) to[out=down,in=right] (0.5,2) 
                        to[out=left,in=down] (-1,3)
                        to[out=up,in=left] (0.5,4) to[out=right,in=up] (2,3);
        \draw[->,thick] (10,5) to[out=down,in=up] (8,3) to[out=down,in=down] (10,3) 
                        to[out=up,in=down] (8,5);
        \draw (1.9,0.8) \bdot;
        \draw (3.7,0.8) \wdot;
        \draw (6,0.8) \bdot;
        \draw (6,1.53) \wdot;
        \draw (6,2.26) \bdot;
        \draw (-1,3) \bdot;
        \draw (0.65,3) \wdot;
        \draw (6,3) \wdot;
        \draw (0.05,4.5) \bdot;
        \draw (2.22,4.35) \wdot;
        \draw (3.78,4.35) \bdot;
        \draw (6,4.35) \bdot;
    \end{tikzpicture}
    \qquad\qquad\qquad
    \begin{tikzpicture}[baseline = 25pt, scale=0.35, color=\clr]
        \draw[<-,thick] (2,0) to[out=up,in=down] (0,5);
        \draw[->,thick] (6,0) to[out=up,in=down] (6,5);
        \draw[<-,thick] (0,0) to[out=up,in=left] (2,1.5) to[out=right,in=up] (4,0);
        \draw[->,thick] (2,5) to[out=down,in=left] (3,4) to[out=right,in=down] (4,5);
        \draw[->,thick] (9.1,3) to (9,3) to[out=left,in=up] (8,2) to[out=down,in=left] (9,1)
                        to[out=right,in=down] (10,2) to[out=up,in=right] (9,3);
        \draw[->,thick] (10,5) to[out=down,in=right] (9,4) to[out=left,in=down] (8,5);
        \draw (0.05,4.5) \bdot;
        \draw (0.65,3) \bdot;
        \draw (2.22,4.35) \bdot;
        \draw (3.78,4.35) \wdot;
        \draw (10,2) \bdot;
        \draw (6,5/4) \bdot;
        \draw (6,10/4) \bdot;
        \draw (6,15/4) \bdot;
        \draw (0.3,0.8) \wdot;
        \draw (1.9,0.8) \wdot;
    \end{tikzpicture}
\end{equation}
In particular, it will be obvious that the Hom-spaces in $\OBC$ and $\AOBC$ are spanned by the set of all such diagrams they contain.  As is usually the case, the difficulty is in identifying a subset of these diagrams which form a basis. The main results of this paper are contained in \cref{OBC basis theorem,AOBC basis theorem} in which we provide a diagrammatic basis for the morphism spaces of these supercategories, and \cref{Cyclotomic basis conjecture} in which we provide bases for the cyclotomic quotients of $\AOBC$.

\subsection{Motivation and applications}  

Let us describe the motivation for these supercategories and some consequences of the aforementioned basis theorems.  Let 
\[
\fq = \fq (n) = \left\{ \left(\begin{matrix} A & B \\
                                    B & A
\end{matrix} \right) \mathrel{\Big|} A, B \text{ are $n \times n$ matrices with entries in $\k$} \right\}.
\]  Put a $\Z_{2}$-grading on $\fq = \fq_{\0}\oplus \fq_{\1}$ by setting $\fq_{\0}$ (resp.\ $\fq_{\1}$) to be the subspace of matrices with $B=0$ (resp.\ $A=0$).  Then $\fq$ is the Lie superalgebra of type $Q$, where the Lie bracket given by the graded version of the commutator bracket. See \cref{SS:Liesuperalgebras} for details. 

The representations in type $Q$ do not have a classical analogue. Despite the important early work done by Penkov-Serganova and others to obtain character formulas and other information (see \cite{PS,BruQ} and references therein), the representation theory in type Q remain mysterious.  For example, only very recently the structure of category $\mathcal{O}$ for $\fq$ became clear thanks to the work of Chen \cite{Chen}, Cheng-Kwon-Wang \cite{CKW}, and Brundan-Davidson \cite{Brundan-Davidson,BD2}.

Since the enveloping superalgebra of $\fq$, $U(\fq )$, is a Hopf superalgebra, one can consider the tensor product of $\fq$-supermodules and the duals of finite-dimensional $\fq$-supermodules. Let $V$ denote the natural supermodule for $\fq$; that is, column vectors of height $2n$ with the action of $\fq$ given by matrix multiplication. Using the Hopf structure we can then form tensor products of $V$ and its dual, $V^{*}$. For brevity, let us write $V^{\up}=V$ and $V^{\down}=V^{*}$ and, more generally, given a word $\ob{a}$ in $\up$ and $\down$, let $V^{\ob{a}}$ denote the tensor product of the corresponding supermodules (e.g.\ $V^{\up\up\down}= V\otimes V \otimes V^{*}$).   The full subsupercategory of all $\fq$-supermodules obtained in this way is a natural object of study.   

Moreover, the translation superfunctors given by tensoring with $V$'s and $V^{*}$'s are a key tool in much of the progress made in the study of type Q representations.  That is, an important role is played by the full subsupercategory of endosuperfunctors of the form $V^{\ob{a}} \otimes -$ as $\ob{a}$ ranges over all finite words in $\up$ and $\down$.  Given the importance of these endosuperfunctors, it is of interest to understand this supercategory.

By design $\OBC$ and $\AOBC$, respectively, are combinatorial models for these two supercategories.  Specifically, there is a monoidal superfunctor 
\[
\Phi : \OBC \to \fq\text{-supermodules}
\] given on objects by $\Phi(\ob{a}) = V^{\ob{a}}$.  When $\k$ is characteristic zero, this superfunctor is full (see \cref{T:MixedSchurWeylDuality}).  That is, there is a surjective superalgebra homomorphism 
\begin{equation}\label{E:isomtheorem}
\End_{\OBC}\left(\ob{a} \right) \to \End_{\fq(n)}\left(V^{\ob{a}} \right)
\end{equation}
which, moreover, is an isomorphism whenever the length of $\ob{a}$ is less than or equal to $n$ (see \cref{R:JungKangResult}).  

It follows from our basis theorem that $\End_{\OBC}\left(\up^r\right)$ is isomorphic to the (finite) Sergeev superalgebra introduced in \cite{Ser} (see \cref{C:OBCandSergeev}).  For arbitrary $\ob a$,  $\End_{\OBC}(\ob{a})$ is isomorphic to the walled Brauer-Clifford superalgebra introduced by Jung-Kang \cite{JK} (see \cref{C:EndOBCisomorphictoWalledBrauerCliffordAlgebra}). The fact that \cref{E:isomtheorem} is an isomorphism whenever the length of $\ob{a}$ is less than or equal to $n$ recovers \cite[Theorem 3.5]{JK}.  We should point out the definitions given in \cite{JK} are global in nature. For example, it is not a priori clear their intricate rule defines an associative product.  In contrast, our diagrammatic description for these superalgebras involves only local relations and leads to significant simplifications.

Analogously, there is a monoidal superfunctor from $\AOBC$ to the supercategory of endosuperfunctors of $\fq$-supermodules, $\Psi: \AOBC \to \EndUksmod$, given on objects by $\Psi(\ob{a}) = V^{\ob{a}} \otimes -$. When $\k$ is characteristic zero this superfunctor is faithful ``asymptotically'' in the sense that given any nonzero morphism in $\AOBC$, its image under $\Psi$ is nonzero as long as $n$ is sufficiently large.  Indeed, this observation is key to proving the basis theorems.  We reduce to showing that the induced map 
\[
\End_{\AOBC}(\up^r) \to \End_{\EndUksmod}(V^{\otimes r} \otimes -)
\]
is injective for $n$ sufficiently large (e.g.\  $n \geq r$).  This in turn is proven by introducing a certain $\fq (n)$-supermodule $M$ (which we call the generic Verma supermodule) and proving that the induced map of superalgebras
\[
\End_{\AOBC}(\up^r) \to \End_{\fq(n)}(V^{\otimes r} \otimes M)
\]
is injective.  

As an application, in \cref{SS:bubblesandcenter} we use $\Psi$ to compute a family of central elements in $U(\fq)$.  By the basis theorem $\End_{\AOBC}(\unit)$ is known to be a polynomial ring in $\Delta_1,\Delta_3,\Delta_5,\ldots$ where $\Delta_{k}$ is defined by \cref{Delta sub k}. In \cref{SS:bubblesandcenter} we explicitly compute the central element corresponding to $\Psi(\Delta_{k})$ for each $k$ and show that they are essentially the central elements first introduced by Sergeev \cite{SergeevCenter} after the application of the antipode of $U(\fq)$ (see \cref{P:SergeevCenter}).

\subsection{Cyclotomic quotients}\label{SS:Cyclotomic quotients}

Fix nonnegative integers $a,b$ and $m_i\in\k$ for each $1\leq i\leq a$.  Using this data fix the polynomial of degree $\ell:=2a+b$ given by $f(t)=t^b\prod_{1\leq i\leq a}(t^2-m_i)$.   The \emph{cyclotomic quotient} $\OBC^f$ is the supercategory defined as the quotient of $\AOBC$ by the left tensor ideal generated by $f(\xdot)$. Note that $\OBC^f$  does not obviously inherit the structure of a monoidal supercategory from $\AOBC$, but it is a left module supercategory over $\AOBC$. In \cref{S:CyclotomicQuotients} we describe a basis for the Hom-spaces of $\OBC^f$ (see \cref{Cyclotomic basis conjecture}).  That this forms a basis was conjectured in a preprint version of this article (written by the second two authors). The first author provided a proof soon after the preprint appeared on the ArXiv; see \cref{S:cyclobasis}. Subsequently, Gao-Rui-Son-Su posted an independent proof of \cref{Cyclotomic basis conjecture} in \cite{GRSS2}; their proof is in the spirit of our proof of \cref{AOBC basis theorem}.

\subsection{Connection to the superalgebras of Gao-Rui-Song-Su}\label{SS:GRSSalgebras}  Shortly after the authors released the first version of this paper, Gao-Rui-Song-Su posted a preprint to the ArXiv in which they describe their independent work on affine walled Brauer-Clifford superalgebras and their cyclotomic quotients \cite{GRSS}.  In the same spirit as this paper, they define these superalgebras by generators and relations and provide basis theorems.  The key difference is that we choose to work with the \emph{supercategories}, whereas they work with \emph{superalgebras}.  In \cref{SS:connection to GRSS} we explain the connection between these two approaches and show their superalgebras are specializations of endomorphism superalgebras in our supercategories. To do so, we construct explicit superalgebra maps from the Gao-Rui-Song-Su superalgebras to the endomorphism superalgebras of our supercategories. Using our bases theorems, one can check that the images under those superalgebra maps of the so-called regular monomials in the superalgebras of Gao-Rui-Song-Su are bases of the appropriate endomorphism superalgebras. Whence, our basis theorems imply the linear independence of the corresponding basis theorems of Gao-Rui-Song-Su (\cite[Theorem 5.15 and Theorem 6.10]{GRSS}). 
These arguments cannot merely be reversed in order to obtain our basis theorems from those in \cite{GRSS}. The main obstacle comes from the fact that our endomorphism superalgebras do not come equipped with nice descriptions via generators and relations. Indeed, without already having a basis theorem for the supercategory, it is difficult to extract a full system of generators and relations for the endomorphism superalgebras from the defining generators and relations for the monoidal supercategory. Hence, providing a well-defined superalgebra map from our endomorphism superalgebras to the superalgebras of Gao-Rui-Song-Su is not an easy task. In particular, we are unable to conclude \cref{T:integralAOBC} from \cite[Theorem 5.14]{GRSS} nor \cref{Cyclotomic basis conjecture} from \cite[Theorem 6.10]{GRSS}.

\subsection{Future directions}\label{SS:futuredirections}

There are a number of interesting questions yet to be considered.  For example, in his PhD thesis, Reynolds showed that the locally finite-dimensional representations of a certain specialization of the oriented Brauer category provide a categorification of the tensor product of a highest weight representation and lowest weight representation for a Kac-Moody Lie algebra of type A \cite{ARthesis}.  See \cite{BSkein} for the quantum version of this story.  We expect similar results to hold for $\OBC$ where the categorification is of a representation for a Kac-Moody algebra of type B.  Another natural question is to extend the results of this paper from the classical to the quantum setting.  There is a quantized enveloping superalgebra, $U_{q}(\fq )$, which is a Hopf superalgebra and one can ask for quantum analogues of $\OBC$ and $\AOBC$.  The quantum walled Brauer-Clifford superalgebras were already introduced in \cite{BGHKW}.  In a third direction, there should be representations of $\AOBC$ and its cyclotomic quotients related to the representations of finite $W$-superalgebras of type Q and to an expected higher level mixed Schur-Weyl-Sergeev duality (see \cite[Section 4.6]{BCNR} where this is explained for type A).  So far as the authors are aware, this theory has yet to be developed. 

\subsection{Acknowledgements}\label{SS:acknowledgements}  The authors are pleased to thank Nicholas Davidson and Catharina Stroppel for helpful conversations. Part of this project was completed while the second author enjoyed a visit to the Max Planck Institute in Bonn. He would like to thank the institute for providing an excellent working environment. 

\section{Monoidal supercategories}\label{section: monoidal supercategories}
In this section we give a brief introduction to monoidal supercategories following \cite[\S1]{BE}. We refer the reader to \emph{loc.~cit.}~for more details and further references. 

\subsection{Superspaces}\label{SS:superspaces} Let $\k$ be a fixed ground field of characteristic not two.  A \emph{superspace} $V = V_{\0} \oplus V_{\1}$ is a $\Z_{2}$-graded $\k$-vector space.  As we will also have $\Z$-gradings we reserve the word degree for later use and instead refer to the \emph{parity} of an element.  That is, elements of $V_{\0}$ (resp.\ $V_{\1}$) are said to have parity $\0$ or to be \emph{even} (resp.\ parity $\1$ or \emph{odd}). Given a homogeneous element $v \in V$ we write $\p{v} \in \Z_{2}$ for the parity of the element. Given two superspaces $V$ and $W$, the set of all linear maps $\Hom_{\k}(V,W)$ is naturally $\Z_{2}$-graded by declaring that $f:V \to W$ has parity $\ep \in \Z_{2}$ if $f(V_{\ep'}) \subseteq V_{\ep+\ep'}$ for all $\ep'\in  \Z_{2}$. 
Let $\svec$ and $\usvec$ denote the categories of all superspaces with $\Hom_\svec(V,W)=\Hom_\k(V,W)$ and $\Hom_\usvec(V,W)=\Hom_\k(V,W)_{\bar{0}}$.

Given superspaces $V$ and $W$, the tensor product $V \otimes W$ as vector spaces is also naturally a superspace with $\Z_{2}$-grading given by declaring $\p{v\otimes w}= \p{v}+\p{w}$ for all homogeneous $v \in V$ and $w \in W$.  
The tensor product of linear maps between superspaces is defined via $(f\otimes g)(v\otimes w)=(-1)^{\p{g}\p{v}}f(v)\otimes g(w)$. This gives $\usvec$ (but \emph{not} $\svec$) the structure of a monoidal category with $\unit=\k$ (viewed as superspace concentrated in even parity). The graded flip map $v\otimes w\mapsto(-1)^{\p{v}\p{w}}w\otimes v$ gives $\usvec$ the structure of a symmetric monoidal category. Here and elsewhere we write the formula only for homogeneous elements with the general case given by extending linearly.

\subsection{Monoidal supercategories}\label{SS: monoidal supercats}
By a \emph{supercategory} we mean a category enriched in $\usvec$. Similarly, a \emph{superfunctor} is a functor enriched in $\usvec$. Given two superfunctors $F,G:\cat{A}\to\cat{B}$, a \emph{supernatural transformation} $\eta:F\to G$ consists of $\eta_{\ob a,\ep}\in\Hom_\cat{B}(F\ob a,G\ob a)_\ep$ for each object $\ob a\in \cat{A}$ and $\ep\in\Z_2$ such that $\eta_{\ob b,\ep}\circ F f=(-1)^{\ep\p{f}}G f\circ\eta_{\ob a,\ep}$ for every $f\in\Hom_\cat{A}(\ob a,\ob b)$. We will write $\eta_{\ob a}=\eta_{\ob a,\0}+\eta_{\ob a,\1}\in\Hom_\cat{B}(F\ob a,G\ob a)$. 
The space of all supernatural transformations from $F$ to $G$ is given the structure of a superspace by declaring a supernatural transformation $\eta$ to be \emph{even} (resp.~\emph{odd}) if $\eta_{\ob a,\bar{1}}=0$ (resp.~$\eta_{\ob a,\bar{0}}=0$) for all objects $\ob a$.

Given two supercategories $\cat{A}$ and $\cat{B}$, there is a supercategory $\cat{A}\boxtimes\cat{B}$ whose objects are pairs $(\ob a,\ob b)$ of objects $\ob a\in\cat{A}$ and $\ob b\in\cat{B}$ and whose morphisms are given by the tensor product of superspaces $\Hom_{\cat{A}\boxtimes\cat{B}}((\ob a,\ob b),(\ob a',\ob b'))=\Hom_\cat{A}(\ob a,\ob a')\otimes\Hom_\cat{B}(\ob b,\ob b')$ with composition defined using the symmetric braiding on $\usvec$: $(f\otimes g)\circ(h\otimes k)=(-1)^{\p{g}\p{h}}(f\circ h)\otimes(g\circ k)$. This can be used to give the category $\scat$ of all supercategories the structure of a monoidal category. 

By a \emph{monoidal supercategory} we mean a supercategory $\cat{A}$ equipped with a superfunctor $-\otimes-:\cat{A}\boxtimes\cat{A}\to\cat{A}$, a unit object $\unit$, and even supernatural isomorphisms $(-\otimes-)\otimes-\arup{\sim}-\otimes(-\otimes-)$ and $\unit\otimes-\arup{\sim}-\larup{\sim}-\otimes\unit$ called \emph{coherence maps} satisfying certain axioms analogous to the ones for a monoidal category. A monoidal supercategory is called \emph{strict} if its coherence maps are identities.  
A \emph{monoidal superfunctor} between two monoidal supercategories $\cat{A}$ and $\cat{B}$ is a superfunctor $F:\cat{A}\to\cat{B}$ equipped with an even supernatural isomorphism $(F-)\otimes (F-)\arup{\sim}F(-\otimes-)$ and an even isomorphism $\unit_\cat{B}\arup{\sim}F\unit_\cat{A}$ satisfying axioms analogous to the ones for a monoidal functor. 

A \emph{braided monoidal supercategory} is a monoidal supercategory $\cat{A}$ equipped with a $\usvec$-enriched version of a braiding. More precisely, let $T:\cat{A}\boxtimes\cat{A}\to\cat{A}$ denote the superfunctor defined on objects by $(\ob a,\ob b)\mapsto \ob b\otimes\ob a$ and on morphisms by $f\otimes g\mapsto (-1)^{\p{f}\p{g}}g\otimes f$. A braiding on a $\cat{A}$ is a supernatural isomorphism $\gamma:-\otimes-\to T$ satisfying the usual hexagon axioms. A \emph{symmetric monoidal supercategory} is a braided monoidal supercategory $\cat{A}$ with $\gamma_{\ob a,\ob b}^{-1}=\gamma_{\ob b,\ob a}$ for all objects $\ob a,\ob b\in\cat{A}$.

Given a monoidal supercategory $\cat{A}$ and an object $\ob a\in\cat{A}$, by a \emph{(left) dual} to $\ob a$ we mean an object $\ob a^*$ equipped with homogeneous \emph{evaluation} and \emph{coevaluation} morphisms $\ev_{\ob{a}}: \ob a^*\otimes \ob a\to\unit$ and $\coev_{\ob{a}}: \unit\to\ob a\otimes\ob a^*$, respectively, in which $\ev_{\ob{a}}$ and $\coev_{\ob{a}}$ have the same parity and satisfy the super version of the usual adjunction axioms. For example, given a finite-dimensional superspace $V$ with homogeneous basis $\{v_i \mid i\in I\}$, then $V^*=\Hom_\k(V,\k)$ with evaluation and coevaluation given by $f\otimes v\mapsto f(v)$ and $1\mapsto\sum_{i\in I}v_i\otimes v^*_i$ respectively, where $v^*_i\in V^*$ is defined by $v_i^*(v_j)=\delta_{i,j}$.  A monoidal supercategory in which every object has a (left) dual is called \emph{(left) rigid}.

The following examples will be relevant for what follows.

\begin{itemize} 
   \item [(i)]  Any $\k$-linear monoidal category can be viewed as a monoidal supercategory in which all $\Hom$-spaces are concentrated in parity $\0$.  If the category is braided, symmetric braided, or rigid, then it still is as a supercategory.

    \item[(ii)] The tensor product and braiding defined in \cref{SS:superspaces} give $\svec$ the structure of a symmetric monoidal supercategory with $\unit=\k$ (viewed as a superspace concentrated in parity $\0$).  The symmetric braiding  $\gamma_{V,W}:V \otimes W \to W \otimes V$ is given by the graded flip map.  The full subsupercategory of finite-dimensional superspaces is rigid.

    \item[(iii)] Given a Lie superalgebra $\fg = \fg_{\0}\oplus \fg_{\1 }$ over a field $\k $ of characteristic not two, let $\gsmod$ denote the supercategory of all $\fg$-supermodules.  That is, superspaces $M = M_{\0}\oplus M_{\1}$ with an action by $\fg$ which respects the grading in the sense that $\fg_{\ep}.M_{\ep'} \subseteq M_{\ep+\ep'}$.  The tensor product $M \otimes M'$ has action given by $x.(m\otimes m') = (x.m) \otimes m' + (-1)^{\p{x}\p{m}}m \otimes (x.m')$ for all homogeneous $x \in \fg$, $m\in M$, and $m'\in M'$ and the graded flip map provides a symmetric braiding.  The unit object $\unit$ is the ground field $\k$ with trivial $\fg$-action.  In this way $\gsmod$ is a symmetric monoidal supercategory. The full subsupercategory of finite-dimensional $\fg$-supermodules is rigid with the action given on $M^{*}$ by $(x.f)(m) = -(-1)^{\p{x}\p{f}}f(x.m)$.

    \item[(iv)] Given a supercategory $\cat{A}$ let $\END(\cat{A})$ denote the supercategory whose objects are all superfunctors $\cat{A}\to\cat{A}$ with supernatural transformations as morphisms. We give $\END(\cat{A})$ the structure of a monoidal supercategory with $\unit=\Id$ as follows. The tensor product of two superfunctors $F,G:\cat{A}\to\cat{A}$ is defined by composition $F\otimes G=F\circ G$. Given supernatural transformations $\eta:F\to G$ and $\theta:H\to K$ we define $\eta\otimes \theta:F\circ H\to G\circ K$ by setting $(\eta\otimes \theta)_{\ob a,\ep}=\sum\limits_{\ep_1+\ep_2=\ep}\eta_{K\ob a,\ep_1} \circ F\theta_{\ob a,\ep_2}$ for each object $\ob a\in\cat{A}$ and $\ep\in\Z_2$. The coherence maps are the obvious ones.  
\end{itemize}

When working with monoidal supercategories it will sometimes be convenient to use the following notation. Given objects $\ob a$ and $\ob b$ in a monoidal supercategory, we write $\ob a\ob b:=\ob a\otimes\ob b$. We will also write ${\ob a}^r:=\underbrace{\ob a\otimes\cdots\otimes\ob a}_{r\text{ times}}$.

\subsection{String calculus}\label{string calculus}
There is a well-defined string calculus for strict monoidal supercategories discussed in \cite[\S1.2]{BE}. A morphism $f:\ob a\to \ob b$ is drawn as
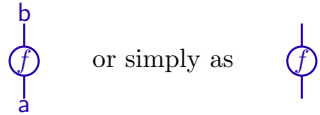
\begin{equation*}
    \begin{tikzpicture}[baseline = 12pt,scale=0.5,color=\clr,inner sep=0pt, minimum width=11pt]
        \draw[-,thick] (0,0) to (0,2);
        \draw (0,1) node[circle,draw,thick,fill=white]{$f$};
        \draw (0,-0.2) node{$\ob a$};
        \draw (0, 2.3) node{$\ob b$};
    \end{tikzpicture}
    \qquad\text{or simply as}\qquad
    \begin{tikzpicture}[baseline = 12pt,scale=0.5,color=\clr,inner sep=0pt, minimum width=11pt]
        \draw[-,thick] (0,0) to (0,2);
        \draw (0,1) node[circle,draw,thick,fill=white]{$f$};
    \end{tikzpicture}
\end{equation*}
when the objects are left implicit.  Notice that the convention used in this paper is to read diagrams from bottom to top. 
The products of morphisms $f\otimes g$ and $f\circ g$ are given by horizontal and vertical stacking respectively:
\begin{equation*}
    \begin{tikzpicture}[baseline = 19pt,scale=0.5,color=\clr,inner sep=0pt, minimum width=11pt]
        \draw[-,thick] (0,0) to (0,3);
        \draw[-,thick] (2,0) to (2,3);
        \draw (1,1.5) node[color=black]{$\otimes$};
        \draw (0,1.5) node[circle,draw,thick,fill=white]{$f$};
        \draw (2,1.5) node[circle,draw,thick,fill=white]{$g$};
        % \draw (0,-0.2) node{$\ob c$};
        % \draw (0, 3.3) node{$\ob d$};
        % \draw (2,-0.2) node{$\ob a$};
        % \draw (2, 3.3) node{$\ob b$};
    \end{tikzpicture}
    ~=~
    \begin{tikzpicture}[baseline = 19pt,scale=0.5,color=\clr,inner sep=0pt, minimum width=11pt]
        \draw[-,thick] (0,0) to (0,3);
        \draw[-,thick] (2,0) to (2,3);
        \draw (0,1.5) node[circle,draw,thick,fill=white]{$f$};
        \draw (2,1.5) node[circle,draw,thick,fill=white]{$g$};
        % \draw (0,-0.2) node{$\ob c$};
        % \draw (0, 3.3) node{$\ob d$};
        % \draw (2,-0.2) node{$\ob a$};
        % \draw (2, 3.3) node{$\ob b$};
    \end{tikzpicture}
    ~,\qquad
    \begin{tikzpicture}[baseline = 19pt,scale=0.5,color=\clr,inner sep=0pt, minimum width=11pt]
        \draw[-,thick] (0,0) to (0,3);
        \draw[-,thick] (2,0) to (2,3);
        \draw (1,1.5) node[color=black]{$\circ$};
        \draw (0,1.5) node[circle,draw,thick,fill=white]{$f$};
        \draw (2,1.5) node[circle,draw,thick,fill=white]{$g$};
        % \draw (0, 3.3) node{$\ob c$};
        % \draw (0,-0.2) node{$\ob b$};
        % \draw (2, 3.3) node{$\ob b$};
        % \draw (2,-0.2) node{$\ob a$};
    \end{tikzpicture}
    ~=~
    \begin{tikzpicture}[baseline = 19pt,scale=0.5,color=\clr,inner sep=0pt, minimum width=11pt]
        \draw[-,thick] (0,0) to (0,3);
        \draw (0,2.2) node[circle,draw,thick,fill=white]{$f$};
        \draw (0,0.8) node[circle,draw,thick,fill=white]{$g$};
        % \draw (0,-0.2) node{$\ob c$};
        % \draw (0.3,1.5) node{$\ob b$};
        % \draw (0, 3.3) node{$\ob a$};
    \end{tikzpicture}
    ~.
\end{equation*}
Pictures involving multiple products should be interpreted by \emph{first composing horizontally, then composing vertically}. For example, 
\begin{equation*}
    \begin{tikzpicture}[baseline = 19pt,scale=0.5,color=\clr,inner sep=0pt, minimum width=11pt]
        \draw[-,thick] (0,0) to (0,3);
        \draw[-,thick] (2,0) to (2,3);
        \draw (0,2.2) node[circle,draw,thick,fill=white]{$f$};
        \draw (2,2.2) node[circle,draw,thick,fill=white]{$g$};
        \draw (0,0.8) node[circle,draw,thick,fill=white]{$h$};
        \draw (2,0.8) node[circle,draw,thick,fill=white]{$k$};
    \end{tikzpicture}
\end{equation*}
should be interpreted as $(f\otimes g)\circ(h\otimes k)$. In general, this is \emph{not} the same as $(f\circ h)\otimes(g\circ k)$ because of the \emph{super-interchange law}:
\[
    (f\otimes g)\circ(h\otimes k)=(-1)^{\p{g}\p{h}}(f\circ h)\otimes(g\circ k).
\]
In terms of string calculus, the super-interchange law implies 
\begin{equation}\label{super-interchange}
    \begin{tikzpicture}[baseline = 19pt,scale=0.5,color=\clr,inner sep=0pt, minimum width=11pt]
        \draw[-,thick] (0,0) to (0,3);
        \draw[-,thick] (2,0) to (2,3);
        \draw (0,2) node[circle,draw,thick,fill=white]{$f$};
        \draw (2,1) node[circle,draw,thick,fill=white]{$g$};
    \end{tikzpicture}
    ~=~
    \begin{tikzpicture}[baseline = 19pt,scale=0.5,color=\clr,inner sep=0pt, minimum width=11pt]
        \draw[-,thick] (0,0) to (0,3);
        \draw[-,thick] (2,0) to (2,3);
        \draw (0,1.5) node[circle,draw,thick,fill=white]{$f$};
        \draw (2,1.5) node[circle,draw,thick,fill=white]{$g$};
    \end{tikzpicture}
    ~=(-1)^{\p{f}\p{g}}~
    \begin{tikzpicture}[baseline = 19pt,scale=0.5,color=\clr,inner sep=0pt, minimum width=11pt]
        \draw[-,thick] (0,0) to (0,3);
        \draw[-,thick] (2,0) to (2,3);
        \draw (0,1) node[circle,draw,thick,fill=white]{$f$};
        \draw (2,2) node[circle,draw,thick,fill=white]{$g$};
    \end{tikzpicture}
    ~.
\end{equation}

\subsection{Graded and filtered superspaces}\label{subsection: graded and filtered superspaces}

By a \emph{graded superspace} we mean a $\k$-vector space $V$ equipped with a decomposition $V=\bigoplus_{(i,\ep)\in \Z\times\Z_2}V_{i,\ep}$. We write $V_i=V_{i,\bar{0}}\oplus V_{i,\bar{1}}$ for the elements of $V$ that are \emph{homogeneous of degree $i$}. Given two graded superspaces $V$ and $W$, we write $\Hom_\k(V,W)_{i,\bar{0}}$ (resp.~$\Hom_\k(V,W)_{i,\bar{1}}$) for the space of all $\k$-linear maps that map $V_{j,\ep}$ to $W_{j+i,\ep}$ (resp.~$W_{j+i,\ep+\bar{1}}$) for each $(j,\ep)\in\Z\times\Z_2$. We let $\gsvec$ and $\ugsvec$ denote the supercategories of all graded superspaces with 
\begin{align*}
    \Hom_\gsvec(V,W) &=\bigoplus_{(i,\ep)\in\Z\times\Z_2}\Hom_\k(V,W)_{i,\ep}\\
    \text{and } \Hom_\ugsvec(V,W) &=\Hom_\k(V,W)_{0,\bar{0}}.
\end{align*}
There is a natural way to give $\gsvec$ (resp.~$\ugsvec$) the structure of a monoidal supercategory (resp.~monoidal category) with 
\[(V\otimes W)_{i,\ep}=\bigoplus_{(j,\eta)\in\Z\times\Z_2}V_{j,\eta}\otimes W_{i-j,\ep+\eta}.\]

By a \emph{filtered superspace} we mean a superspace $V=V_{\bar{0}}\oplus V_{\bar{1}}$ equipped with a filtration $\cdots\subseteq V_{\leq i,\ep}\subseteq V_{\leq i+1,\ep}\subseteq\cdots$ such that $\bigcap_{(i,\ep)\in\Z\times\Z_2}V_{\leq i,\ep}=0$ and $\bigcup_{(i,\ep)\in\Z\times\Z_2}V_{\leq i,\ep}=V_\ep$ for each $\ep\in\Z_2$. We write $V_{\leq i}=V_{\leq i,\bar{0}}\oplus V_{\leq i,\bar{1}}$ for the elements of $V$ that are \emph{filtered degree $i$}. Given two filtered superspaces $V$ and $W$, we write $\Hom_\k(V,W)_{\leq i,\bar{0}}$ (resp.~$\Hom_\k(V,W)_{\leq i,\bar{1}}$) for the space of all $\k$-linear maps that map $V_{\leq j,\ep}$ to $W_{\leq j+i,\ep}$ (resp.~$W_{\leq j+i,\ep+\bar{1}}$) for each $(j,\ep)\in\Z\times\Z_2$. We let $\fsvec$ and $\ufsvec$ denote the supercategories of all filtered superspaces with 
\begin{align*}
    \Hom_\fsvec(V,W) &=\bigoplus_{(i,\ep)\in\Z\times\Z_2}\Hom_\k(V,W)_{\leq i,\ep}\\
    \text{and } \Hom_\ufsvec(V,W) &=\Hom_\k(V,W)_{\leq 0,\bar{0}}.
\end{align*}
There is a natural way to give $\fsvec$ (resp.~$\ufsvec$) the structure of a monoidal supercategory (resp.~monoidal category) with 
\[(V\otimes W)_{\leq i,\ep}=\bigoplus_{(j,\eta)\in\Z\times\Z_2}V_{\leq j,\eta}\otimes W_{\leq i-j,\ep+\eta}.\]
Every graded superspace $V$ can be viewed as a filtered superspace by setting $V_{\leq i,\ep}=\bigoplus_{j\leq i}V_{j,\ep}$ for all $(i,\ep)\in\Z\times\Z_2$. On the other hand, given a filtered superspace $V$, we write $\gr V$ for the \emph{associated graded superspace} with $(\gr V)_{i,\ep}:=V_{\leq i,\ep}/V_{\leq i-1,\ep}$ for each $(i,\ep)\in\Z\times\Z_2$. Given filtered superspaces $V$ and $W$, a map $f\in\Hom_\k(V,W)_{\leq i,\ep}$ induces a map $\gr_{i,\ep}f\in\Hom_\k(\gr V,\gr W)_{i,\ep}$ in an obvious way.

\subsection{Graded and filtered supercategories}\label{subsection: graded and filtered super-cats}
By a \emph{graded} (resp.~\emph{filtered}) \emph{supercategory} we mean a category enriched in $\ugsvec$ (resp.~$\ufsvec$). Similarly, a \emph{graded} (resp.~\emph{filtered}) \emph{superfunctor} is a functor enriched in $\ugsvec$ (resp.~$\ufsvec$). By a \emph{graded monoidal supercategory} we mean a monoidal supercategory that is graded in such a way that $f\otimes g$ is homogeneous of degree $i+j$ whenever $f$ and $g$ are homogeneous of degree $i$ and $j$ respectively. Similarly, a \emph{filtered monoidal supercategory} is a monoidal supercategory that is filtered in such a way that $f\otimes g$ has filtered degree $i+j$ whenever $f$ and $g$ have filtered degree $i$ and $j$ respectively. 
 
Given a filtered supercategory $\cat{A}$, the \emph{associated graded supercategory} $\gr\cat{A}$ is the supercategory with the same objects as $\cat{A}$ and with $\Hom_{\gr\cat{A}}(\ob a,\ob b):=\gr\Hom_\cat{A}(\ob a,\ob b)$. The composition in $\gr\cat{A}$ is induced from the composition in $\cat{A}$. Similarly, given a filtered superfunctor $F:\cat{A}\to\cat{B}$ we write $\gr F:\gr\cat{A}\to\gr\cat{B}$ for the associated graded superfunctor defined in the obvious way. 

For example, $\gsvec$ and $\fsvec$ are graded and filtered monoidal supercategories respectively. Note that $\gr\fsvec$ and $\gsvec$ are not the same, but there is a faithful superfunctor 
\begin{equation*}\label{gr fsvec to gsvec}
    \Gamma:\gr\fsvec\to\gsvec
\end{equation*}
which maps a filtered superspace to its associated graded superspace and maps $f+\Hom_\k(V,W)_{\leq i-1,\ep}\in \Hom_\k(V,W)_{\leq i,\ep}/\Hom_\k(V,W)_{\leq i-1,\ep}$
to $\gr_{i,\ep}f$.

\section{The degenerate affine oriented Brauer-Clifford supercategory}\label{section: AOBC}
In this section we define the monoidal supercategories $\OBC$ and $\AOBC$. First, however, we recall the definition of the supercategory $\OB$ from \cite{BCNR}. Our definitions will make use of the string calculus for strict monoidal supercategories (see \cref{string calculus}). In particular, each of the supercategories mentioned above admit a diagrammatic description. 
The objects in each of these supercategories are $\otimes$-generated by two objects denoted $\up$ and $\down$. Hence, the set of all objects can be identified with the set $\wrd$ of all finite words in the letters $\up$ and $\down$. The string diagrams for these supercategories will be made with oriented strings with an upward (resp.~downward) string corresponding to the object $\up$ (resp.~$\down$). For example, a diagram of the form 
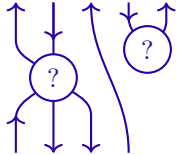
\begin{equation*}
    \begin{tikzpicture}[baseline = 10pt, scale=0.5, color=\clr]
        \draw[->,thick] (0,-1) to (0,0);
        \draw[-,thick] (0,-0.1) to[out=up, in=down] (1,1.25);
        \draw[<-,thick] (1,-1) to[out=up, in=down] (1,1.25);
        \draw[<-,thick] (2,-1) to (2,0) to[out=up, in=down] (1,1.25);
        \draw[->,thick] (3,-1) to[out=up, in=down] (2,3);
        \draw[<-,thick] (0,3) to (0,2) to[out=down, in=up] (1,0.75);
        \draw[-,thick] (1,3) to[out=down, in=up] (1,0.75);
        \draw[->,thick] (1,3) to[out=down, in=up] (1,2);
        \draw[->,thick] (3,3) to (3,2.5);
        \draw[-,thick] (3,2.6) to[out=down, in=up] (3.5,1.5);
        \draw[<-,thick] (4,3) to (4,2.5) to[out=down, in=up] (3.5,1.5);
        \draw (1,1) node[circle,draw,thick,color=\clr,fill=white]{?};
        \draw (3.5,1.75) node[circle,draw,thick,color=\clr,fill=white]{?};
    \end{tikzpicture}
\end{equation*}
corresponds to a morphism $\up\down\down\up\to\up\down\up\down\up$.  
We will describe classes of diagrams which give bases for the Hom-spaces of $\OBC$ and $\AOBC$. In this section we will show these diagrams indeed span the appropriate Hom-spaces. Proof that the diagrams are linearly independent will be given in \cref{section: proof of AOBC basis theorem}.

\subsection{The oriented Brauer category} \label{subsection: OB}
In \cite{BCNR} the oriented Brauer category is defined diagrammatically and then a presentation is given in terms of generators and relations \cite[Theorem 1.1]{BCNR}. We take the latter description as our definition and view it as a supercategory concentrated in parity $\0$.  

\begin{definition}\label{OB defn}
    The \emph{oriented Brauer supercategory} $\OB$ is the $\k$-linear strict monoidal supercategory generated by two objects $\up, \down$ and three even morphisms $\lcup:\unit\to\up\down,$ $\lcap: \down\up\to\unit, \swap: \up\up\to\up\up$ subject to the following relations:
    \begin{equation}\label{OB relations 1 (symmetric group)}
        \begin{tikzpicture}[baseline = 10pt, scale=0.5, color=\clr]
            \draw[-,thick] (0,0) to[out=up, in=down] (1,1);
            \draw[->,thick] (1,1) to[out=up, in=down] (0,2);
            \draw[-,thick] (1,0) to[out=up, in=down] (0,1);
            \draw[->,thick] (0,1) to[out=up, in=down] (1,2);
        \end{tikzpicture} 
        ~=~
        \begin{tikzpicture}[baseline = 10pt, scale=0.5, color=\clr]
            \draw[-,thick] (0,0) to (0,1);
            \draw[->,thick] (0,1) to (0,2);
            \draw[-,thick] (1,0) to (1,1);
            \draw[->,thick] (1,1) to (1,2);
        \end{tikzpicture}
        ,\qquad
        \begin{tikzpicture}[baseline = 10pt, scale=0.5, color=\clr]
            \draw[->,thick] (0,0) to[out=up, in=down] (2,2);
            \draw[->,thick] (2,0) to[out=up, in=down] (0,2);
            \draw[->,thick] (1,0) to[out=up, in=down] (0,1) to[out=up, in=down] (1,2);
        \end{tikzpicture}
        ~=~ 
        \begin{tikzpicture}[baseline = 10pt, scale=0.5, color=\clr]
            \draw[->,thick] (0,0) to[out=up, in=down] (2,2);
            \draw[->,thick] (2,0) to[out=up, in=down] (0,2);
            \draw[->,thick] (1,0) to[out=up, in=down] (2,1) to[out=up, in=down] (1,2);
        \end{tikzpicture},
    \end{equation}
    \begin{equation}\label{OB relations 2 (zigzags and invertibility)}
        \begin{tikzpicture}[baseline = 10pt, scale=0.5, color=\clr]
            \draw[-,thick] (2,0) to[out=up, in=down] (2,1) to[out=up, in=right] (1.5,1.5) to[out=left,in=up] (1,1);
            \draw[->,thick] (1,1) to[out=down,in=right] (0.5,0.5) to[out=left,in=down] (0,1) to[out=up,in=down] (0,2);
        \end{tikzpicture} 
        ~=~
        \begin{tikzpicture}[baseline = 10pt, scale=0.5, color=\clr]
            \draw[-,thick] (0,0) to (0,1);
            \draw[->,thick] (0,1) to (0,2);
        \end{tikzpicture} 
        ,\qquad
        \begin{tikzpicture}[baseline = 10pt, scale=0.5, color=\clr]
            \draw[-,thick] (2,2) to[out=down, in=up] (2,1) to[out=down, in=right] (1.5,0.5) to[out=left,in=down] (1,1);
            \draw[->,thick] (1,1) to[out=up,in=right] (0.5,1.5) to[out=left,in=up] (0,1) to[out=down,in=up] (0,0);
        \end{tikzpicture} 
        ~=~
        \begin{tikzpicture}[baseline = 10pt, scale=0.5, color=\clr]
            \draw[-,thick] (0,2) to (0,1);
            \draw[->,thick] (0,1) to (0,0);
        \end{tikzpicture}
        ,\qquad
        \begin{tikzpicture}[baseline = 10pt, scale=0.5, color=\clr]
            \draw[-,thick] (2,2) to[out=down, in=up] (2,1) to[out=down, in=right] (1.5,0.5) to[out=left,in=down] (1,1);
            \draw[->,thick] (1,1) to[out=up,in=right] (0.5,1.5) to[out=left,in=up] (0,1) to[out=down,in=up] (0,0);
            \draw[->,thick] (0.7,0) to[out=up,in=down] (1.3,2);
        \end{tikzpicture}
        ~\text{ is invertible}.
    \end{equation}
\end{definition}
Note that the last relation is the assertion that there is
another distinguished generator 
% $\rswap:\up\down\to\down\up$ that is a two-sided inverse to 
\begin{equation}\label{left cross}
    \rswap:\up\down\to\down\up\text{ which is a two-sided inverse to }
    \begin{tikzpicture}[baseline = 5pt, scale=0.5, color=\clr]
        \draw[<-,thick] (0,0) to[out=up, in=down] (1,1);
        \draw[->,thick] (1,0) to[out=up, in=down] (0,1);
    \end{tikzpicture}
    ~:=~
    \begin{tikzpicture}[baseline = 10pt, scale=0.5, color=\clr]
        \draw[-,thick] (2,2) to[out=down, in=up] (2,1) to[out=down, in=right] (1.5,0.5) to[out=left,in=down] (1,1);
        \draw[->,thick] (1,1) to[out=up,in=right] (0.5,1.5) to[out=left,in=up] (0,1) to[out=down,in=up] (0,0);
        \draw[->,thick] (0.7,0) to[out=up,in=down] (1.3,2);
    \end{tikzpicture}
    ~.
\end{equation}

We define rightward cups/caps and downward crossings in $\OB$ as follows:
\begin{equation}\label{right cuaps - down cross}
    \begin{tikzpicture}[baseline = 5pt, scale=0.5, color=\clr]
        \draw[->,thick] (0,1) to[out=down,in=left] (0.5,0.35) to[out=right,in=down] (1,1);
    \end{tikzpicture}
    ~:=~
    \begin{tikzpicture}[baseline = 5pt, scale=0.5, color=\clr]
        \draw[->,thick] (0,1) to[out=down,in=up] (1,0) to[out=down,in=right] (0.5,-0.5) to[out=left,in=down] (0,0) to[out=up,in=down] (1,1);
    \end{tikzpicture}
    ~,\qquad
    \begin{tikzpicture}[baseline = 5pt, scale=0.5, color=\clr]
        \draw[->,thick] (0,0) to[out=up,in=left] (0.5,0.65) to[out=right,in=up] (1,0);
    \end{tikzpicture}
    ~:=~
    \begin{tikzpicture}[baseline = 5pt, scale=0.5, color=\clr]
        \draw[->,thick] (0,0) to[out=up,in=down] (1,1) to[out=up,in=right] (0.5,1.5) to[out=left,in=up] (0,1) to[out=down,in=up] (1,0);
    \end{tikzpicture}
    ~,\qquad
    \begin{tikzpicture}[baseline = 5pt, scale=0.5, color=\clr]
        \draw[->,thick] (0,1) to[out=down,in=up] (1,0);
        \draw[->,thick] (1,1) to[out=down,in=up] (0,0);
    \end{tikzpicture}
    ~:=~
    \begin{tikzpicture}[baseline = 10pt, scale=0.5, color=\clr]
        \draw[-,thick] (2,2) to[out=down, in=up] (2,1) to[out=down, in=right] (1.5,0.5) to[out=left,in=down] (1,1);
        \draw[->,thick] (1,1) to[out=up,in=right] (0.5,1.5) to[out=left,in=up] (0,1) to[out=down,in=up] (0,0);
        \draw[<-,thick] (0.7,0) to[out=up,in=down] (1.3,2);
    \end{tikzpicture}
    ~.
\end{equation}
An \emph{oriented Brauer diagram with bubbles} of type $\ob a\to\ob b$ is any string diagram obtained by stacking (vertically and horizontally) the defining generators of $\OB$ along with the diagrams \cref{left cross,right cuaps - down cross} in such a way that the result can be interpreted as a morphism in $\Hom_\OB(\ob a,\ob b)$. For example, here are two oriented Brauer diagrams with bubbles of type $\down\up^4\down^2\to\up^3\down^2$:
\begin{equation}\label{OB diagram examples}
    \begin{tikzpicture}[baseline = 10pt, scale=0.5, color=\clr]
        \draw[-<,thick] (0,1.5) to[out=left,in=up] (-0.5,1) to[out=down, in=left] (0,0.5) to[out=right, in=down] (0.5,1) to[out=up,in=right] (0,1.5) to (-0.15,1.5);
        \draw[->,thick] (4.4,1.3) to[out=right,in=up] (4.9,0.8) to[out=down,in=right] (4.4,0.3) to[out=left,in=right] (3.4,1.3) to[out=left,in=up] (2.9,0.8) to[out=down,in=left] (3.4,0.3) to[out=right,in=left] (4.4,1.3) to (4.5,1.3);
        \draw[->,thick] (2,0) to[out=up,in=down] (1,1) to[out=up,in=down] (1,2);
        \draw[->,thick] (4,0) to[out=up,in=down] (2,2);
        \draw[->,thick] (3,0) to[out=up,in=right] (2,0.75) to[out=left,in=up] (1,0);
        \draw[->,thick] (4,2) to[out=down,in=right] (3.5,1.5) to[out=left,in=down] (3,2);
        \draw[->,thick] (5,0) to[out=up,in=left] (6,1) to[out=right,in=up] (7,0);
        \draw[->,thick] (5,2) to[out=down,in=up] (6,0);
    \end{tikzpicture}
    \qquad\qquad
    \begin{tikzpicture}[baseline = 10pt, scale=0.5, color=\clr]
        \draw[->,thick] (5,2) to[out=down,in=right] (3.5,0.25) to[out=left,in=right] (1.5,1.5) to[out=left,in=up]  (1,0);
        \draw[->,thick] (6,2) to[out=down,in=up] (7,0);
        \draw[->,thick] (2,0) to[out=up,in=down] (4,2);
        \draw[->,thick] (3,0) to[out=up,in=down] (2,1) to[out=up,in=down] (3,2);
        \draw[->,thick] (4,0) to[out=up,in=down] (2,2);
        \draw[->,thick] (5,0) to (5,0.1) to[out=up,in=down] (6,0.75) to[out=up,in=up] (5,0.75) to[out=down,in=up] (6,0.1) to (6,0);
    \end{tikzpicture}
\end{equation}
The term \emph{bubble} refers to any component of such a diagram without an endpoint. In the examples above, the left diagram has two bubbles whereas the right has none. An \emph{oriented Brauer diagram} refers to an oriented Brauer diagram with bubbles that has no bubbles. We say that two oriented Brauer diagrams are \emph{equivalent} if they are of the same type and one diagram can be obtained from the other by continuously deforming its strands, possibly moving them through other strands and crossings, but keeping endpoints fixed. 
Moreover, we say two oriented Brauer diagrams with bubbles are \emph{equivalent} if they have the same number of bubbles and their underlying oriented Brauer diagrams (without bubbles) are equivalent. 
For example, the left (resp.\ right) diagram in \cref{OB diagram examples} is equivalent to the following digram on the left (resp.\ right):
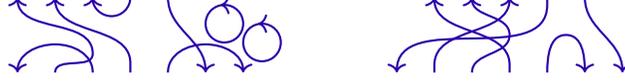
\begin{equation*}   
    \begin{tikzpicture}[baseline = 10pt, scale=0.5, color=\clr]
        \draw[->,thick] (7.5,1.3) to[out=left,in=up] (7,0.8) to[out=down, in=left] (7.5,0.3) to[out=right, in=down] (8,0.8) to[out=up,in=right] (7.5,1.3) to (7.4,1.3);
        \draw[->,thick] (6.5,1.8) to[out=left,in=up] (6,1.3) to[out=down, in=left] (6.5,0.8) to[out=right, in=down] (7,1.3) to[out=up,in=right] (6.5,1.8) to (6.4,1.8);
        \draw[->,thick] (2,0) to[out=up,in=down] (3,0.5) to[out=up,in=down] (1,2);
        \draw[->,thick] (4,0) to (4,0.25) to[out=up,in=down] (2,2);
        \draw[->,thick] (3,0) to[out=up,in=right] (2,0.75) to[out=left,in=up] (1,0);
        \draw[->,thick] (4,2) to[out=down,in=right] (3.5,1.5) to[out=left,in=down] (3,2);
        \draw[->,thick] (5,0) to[out=up,in=left] (6,0.75) to[out=right,in=up] (7,0);
        \draw[->,thick] (5,2) to[out=down,in=up] (6,0);
    \end{tikzpicture}
    \qquad\qquad
    \begin{tikzpicture}[baseline = 10pt, scale=0.5, color=\clr]
        \draw[->,thick] (5,2) to (5,1.8) to[out=down,in=up]  (1,0);
        \draw[->,thick] (6,2) to[out=down,in=up] (7,0);
        \draw[->,thick] (2,0) to (2,0.2) to[out=up,in=down] (4,2);
        \draw[->,thick] (3,0) to[out=up,in=down] (4,1) to[out=up,in=down] (3,2);
        \draw[->,thick] (4,0) to (4,0.2) to[out=up,in=down] (2,2);
        \draw[->,thick] (5,0) to[out=up,in=left] (5.5,1) to[out=right,in=up]  (6,0);
    \end{tikzpicture}
\end{equation*}

Of course, any morphism in $\OB$ can be realized as a $\k$-linear combination of oriented Brauer diagrams with bubbles. It follows from \cite[Theorem 1.1]{BCNR} that two oriented Brauer diagrams with bubbles represent the same morphism in $\OB$ if and only if they are equivalent. Moreover, the set of all equivalence classes of oriented Brauer diagrams with bubbles of type $\ob a\to \ob b$ is a basis for $\Hom_\OB(\ob a, \ob b)$. 

As explained in \cite[\S1]{BCNR}, $\OB$ is a rigid symmetric monoidal supercategory. Briefly, the symmetric braiding on $\OB$ is the obvious one given on generating objects by the crossings $\swap, \lswap, \rswap, \dswap$. Moreover, the generating objects are dual to one another with evaluation and coevaluation maps given by $\lcap, \rcap$ and $\lcup, \rcup$ respectively.

% \subsection{Transpositions in \texorpdfstring{$\OB$}{OB}}\label{SS: transpositions}
% We will make use of the following notation in subsequent sections. 
% Suppose $\ob a = \ob a_\ell\cdots\ob a_1$ with each $\ob a_k\in\{\up,\down\}$. Whenever $1\leq i,j\leq \ell$ we let $(i,j)=(i,j)^{\ob a}$ denote the scaled oriented Brauer diagram of type $\ob a\to\ob a$ which is either the crossing (if $\ob a_i=\ob a_j$) or minus the cup-cap pair (if $\ob a_i\not=\ob a_j$) joining $\ob a_i$ and $\ob a_j$. For example, if $\ob a=\down\up\up\down\up$ then 
% \begin{equation*}
%     (4,1)=~
%     \begin{tikzpicture}[baseline = 12pt, scale=0.5, color=\clr]
%         \draw[->,thick] (1,2) to (1,0);
%         \draw[<-,thick] (3,2) to (3,0);
%         \draw[->,thick] (4,2) to (4,0);
%         \draw[<-,thick] (5,2) to[out=down,in=up] (2,0);
%         \draw[<-,thick] (2,2) to[out=down,in=up] (5,0);
%     \end{tikzpicture}
%     \quad\text{and}\quad
%     (5,3)=-~
%     \begin{tikzpicture}[baseline = 12pt, scale=0.5, color=\clr]
%         \draw[->,thick] (4,2) to (4,0);
%         \draw[<-,thick] (5,2) to (5,0);
%         \draw[<-,thick] (2,2) to (2,0);
%         \draw[->,thick] (3,0) to[out=up,in=right] (2,0.75) to[out=left,in=up] (1,0);
%         \draw[<-,thick] (3,2) to[out=down,in=right] (2,1.25) to[out=left,in=down] (1,2);
%     \end{tikzpicture}
%     ~.
% \end{equation*} 

\subsection{The oriented Brauer-Clifford supercategory} \label{subsection: OBC}
Adjoining ``Clifford generators'' to $\OB$ results in the following:

\begin{definition}\label{OBC defn}
    The \emph{oriented Brauer-Clifford supercategory} $\OBC$ is the $\k$-linear strict monoidal supercategory generated by two objects $\up, \down$; three even morphisms $\lcup:\unit\to\up\down, \lcap: \down\up\to\unit, \swap: \up\up\to\up\up$; and one odd morphism $\cldot:\up\to\up$ subject to \cref{OB relations 1 (symmetric group),OB relations 2 (zigzags and invertibility)}, and the following relations:
    \begin{equation}\label{OBC relations}
        \begin{tikzpicture}[baseline = 7.5pt, scale=0.5, color=\clr]
            \draw[->,thick] (0,0) to[out=up, in=down] (0,1.5);
            \draw (0,1) \wdot;
            \draw (0,0.5) \wdot;
        \end{tikzpicture}
        ~=~
        \begin{tikzpicture}[baseline = 7.5pt, scale=0.5, color=\clr]
            \draw[->,thick] (0,0) to[out=up, in=down] (0,1.5);
        \end{tikzpicture}
        ,\qquad
        \begin{tikzpicture}[baseline = 7.5pt, scale=0.5, color=\clr]
            \draw[->,thick] (0,0) to[out=up, in=down] (1,1.5);
            \draw[->,thick] (1,0) to[out=up, in=down] (0,1.5);
            \draw (0.2,0.5) \wdot;
        \end{tikzpicture}
        ~=~
        \begin{tikzpicture}[baseline = 7.5pt, scale=0.5, color=\clr]
            \draw[->,thick] (0,0) to[out=up, in=down] (1,1.5);
            \draw[->,thick] (1,0) to[out=up, in=down] (0,1.5);
            \draw (0.8,1) \wdot;
        \end{tikzpicture}
        ,\qquad
        \reflectbox{\begin{tikzpicture}[baseline = -2pt, scale=0.5, color=\clr]
            \draw[->,thick] (0,0) to[out=up, in=left] (0.5,0.5);
            \draw[-,thick] (0.4,0.5) to (0.5,0.5) to[out=right,in=up] (1,0) to[out=down, in=right] (0.5,-0.5) to[out=left,in=down] (0,0);
            \draw (0,0) \wdot;
        \end{tikzpicture}}
        ~=0.
    \end{equation}
\end{definition}
Note that the last relation in \cref{OBC relations} makes use of the rightward cup defined by \cref{right cuaps - down cross}.

We define the downward analogue of $\cldot$ as follows:
\begin{equation}\label{down white dot}
    \begin{tikzpicture}[baseline = 12pt, scale=0.5, color=\clr]
        \draw[->,thick] (0,2) to (0,0);
        \draw (0,1) \wdot;
    \end{tikzpicture}
    ~:=~
    \begin{tikzpicture}[baseline = 12pt, scale=0.5, color=\clr]
        \draw[->,thick] 
            (1,2) to (1,1)
            to[out=down,in=right] (0.75,0.25)
            to[out=left,in=down] (0.5,1)
            to[out=up,in=right] (0.25,1.75)
            to[out=left,in=up] (0,1) to (0,0);
        \draw (0.5,1) \wdot;
    \end{tikzpicture}
\end{equation}

There is an obvious monoidal superfunctor $\OB\to\OBC$ which allows us to view equivalence classes of  oriented Brauer diagrams with bubbles as morphisms in $\OBC$. We define an \emph{oriented Brauer-Clifford diagram} (resp.\ \emph{with bubbles}) to be an oriented Brauer diagram (resp.~with bubbles) with finitely many $\emptydot$'s on its segments, where \emph{segment} refers to a connected component of the diagram obtained when all crossings and local extrema are deleted. For example, here are two oriented Brauer-Clifford diagrams of type $\down\up\up\up\down\down\to\up\down\up\down\up\down$:
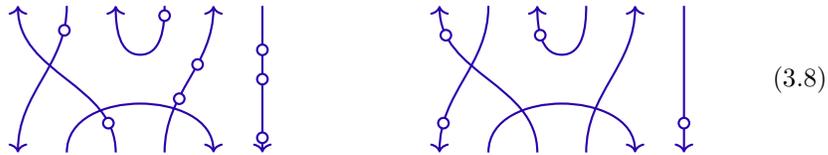
\begin{equation}\label{two OBC diagrams}

\end{equation}
With \cref{down white dot} in mind, we can interpret any oriented Brauer-Clifford diagram with bubbles as a morphism in $\OBC$.

Since $\cldot$ is odd, we must keep the super-interchange law in mind when sliding $\emptydot$'s along strands.  For example, by \cref{super-interchange} we have 
\begin{equation*}
    %
    ~.
\end{equation*}
However, the following result shows that $\emptydot$'s are allowed to freely slide along cups/caps and pass through crossings.

\begin{proposition}\label{prop: OBC slides}
    The following relations hold in $\OBC$:
    \begin{align}
        \label{left cap/cup c-slide}
            %
            ~.
    \end{align}
\end{proposition}

\begin{proof}
    The left of \cref{left cap/cup c-slide} is verified below. The right of \cref{left cap/cup c-slide} is similar.
    \begin{equation*}
        %
        ~\stackrel{\eqref{down white dot}}{=}~
        %
        ~\stackrel{\eqref{OB relations 2 (zigzags and invertibility)}}{=}~
        %
        ~.
    \end{equation*}
    
    Next, we prove \cref{up cross c-slide}--\cref{down cross c-slide}. In \cref{up cross c-slide,LR cross c-slide 1,LR cross c-slide 2,down cross c-slide} the right equality is obtained from the left by composing on top and bottom with $\swap, \rswap, \rswap$, and $\dswap$, respectively. The left of \cref{up cross c-slide} is one of the defining relations of $\OBC$. 
    The left of \cref{LR cross c-slide 1} is verified below:
    \begin{equation*}
        %
        ~.
    \end{equation*}
    Similarly, the left of \cref{LR cross c-slide 2} (resp.\ \cref{down cross c-slide}) follows from \cref{left cross} (resp.\ \cref{right cuaps - down cross}) and the left of \cref{up cross c-slide} (resp.\ \cref{LR cross c-slide 1}).

    Finally, we verify the left of \cref{right cap/cup c-slide} below. 
    \begin{equation*}
        %
        ~.
    \end{equation*}
    The right of \cref{right cap/cup c-slide} is similar.
\end{proof} 

Next, we will flesh out more of the diagrammatic nature of $\OBC$. However, before doing so we pause to point out that \cref{up cross c-slide}-\cref{down cross c-slide} imply the symmetric braiding on $\OB$ extends to a supernatural isomorphism for $\OBC$. In particular, $\OBC$ inherits the structure of a rigid symmetric monoidal supercategory from $\OB$. 

We have a downward analogue of the left relation in \cref{OBC relations}:
\begin{equation}\label{reduce two down Cliffords}
    %
    ~\stackrel{\cref{down white dot}}{=}
    %
 
    ~\stackrel{\cref{OB relations 2 (zigzags and invertibility)}}{=}-
    %
    ~\stackrel{\cref{OB relations 2 (zigzags and invertibility)}}{=}-
    %
    ~.
\end{equation}
In particular, it follows that every oriented Brauer-Clifford diagram with bubbles can be reduced, up to sign, to a diagram with the same underlying oriented Brauer diagram and at most one $\emptydot$ on each strand. For example, the two diagrams in \cref{two OBC diagrams} represent the same morphism in $\OBC$. 

Additionally, we have an even analogue of the right relation in \cref{OBC relations}. Indeed, 
\begin{equation*}
    %
    ~.
\end{equation*}
Since $\chr\k\not=2$, the computation above implies 
\begin{equation}\label{dim equals 0}
    \reflectbox{%
}
        ~=0.
\end{equation}
It follows that any oriented Brauer-Clifford diagram with at least one bubble reduces to zero in $\OBC$.  Hence, the Hom-spaces in $\OBC$ are spanned by oriented Brauer-Clifford diagrams (without bubbles) that have at most one $\emptydot$ on each strand. Next, we refine this spanning set:

We say that an oriented Brauer-Clifford diagram is \emph{normally ordered} if 
\begin{itemize}\label{N.O. OBC defn}
    \item[$\diamond$] it has at most one $\emptydot$ on each strand; all $\emptydot$'s are on outward-pointing boundary segments (i.e.\ segments which intersect the boundary at a point that is directed out of the picture); 
    \item[$\diamond$] all $\emptydot$'s which occur on upwardly oriented segments are positioned at the same height; similarly, all $\emptydot$'s which occur on downwardly oriented segments are positioned at the same height.
\end{itemize}
For example, in \cref{two OBC diagrams} the diagram on the right is normally ordered and the one on the left is not.  We say that two normally ordered oriented Brauer-Clifford diagrams are \emph{equivalent} if their underlying oriented Brauer diagrams are equivalent (see \cref{subsection: OB}) and their corresponding strands have the same number of $\emptydot$'s on them. Note that two equivalent normally ordered oriented Brauer-Clifford diagrams correspond to the same morphism in $\OBC$. Moreover, it follows from the discussion above that the Hom-spaces in $\OBC$ are spanned by normally ordered oriented Brauer-Clifford diagrams. In fact, we have:

\begin{theorem}\label{OBC basis theorem}
    For any $\ob a, \ob b\in\wrd$ the superspace $\Hom_\OBC(\ob a,\ob b)$ has basis given by equivalence classes of normally ordered oriented Brauer-Clifford diagrams of type $\ob a\to\ob b$.
\end{theorem}

It is possible to give a straightforward combinatorial proof of  \cref{OBC basis theorem}. However, we omit such a proof since \cref{OBC basis theorem} is a consequence of our basis theorem for $\AOBC$ as explained in \cref{Consequences of main result}.  

Meanwhile, let us point the following consequence of \cref{OBC basis theorem}.  For $r \geq 1$, let $C_{r}$ denote the Clifford algebra generated by $c_{1}, \dotsc , c_{r}$, subject to the relations $c_{i}^{2}=1$ and $c_{i}c_{j}=-c_{j}c_{i}$ for $i \neq j$.  We view $C_{r}$ as a superalgebra by declaring the generators to be odd.  Let $\k\Sigma_{r}$ denote the group algebra of the symmetric group on $r$ letters viewed as a superalgebra concentrated in parity $\0$.  Then the (finite) \emph{Sergeev superalgebra} is
\begin{equation}\label{E:SergeevAlgebraDef}
\operatorname{Ser}_{r} = C_{r} \otimes \k \Sigma_{r}
\end{equation}
as a superspace with $C_{r}\otimes 1 \cong C_{r}$ and $1 \otimes \k \Sigma_{r} \cong \k \Sigma_{r}$ as subsuperalgebras and with mixed relation $wc_{i}= c_{w(i)}w$ for $i=1,\dotsc ,r$ and all $w \in \Sigma_{r}$.  

\begin{corollary}\label{C:OBC End is isomorphic to Sergeev Algebra}\label{C:OBCandSergeev}  For $r \geq 1$, 
\[
\End_{\OBC}\left(\up^{r} \right) \cong \operatorname{Ser}_{r}.
\]

\end{corollary}

\begin{proof} The Sergeev superalgebra is generated by $s_{1}, \dotsc , s_{r-1}, c_{1}, \dotsc , c_{r}$ subject to $s_{i}^{2}=1$, $s_{i}s_{j}=s_{j}s_{i}$ if $|i-j|>1$, $s_{i}s_{i+1}s_{i}=s_{i+1}s_{i}s_{i+1}$, $c_{i}^{2}=1$, $c_{i}c_{j}=-c_{j}c_{i}$ if $i\neq j$, and $s_{i}c_{i}=c_{i+1}s_{i}$ for all admissible $1 \leq i, j \leq r$.  From this it follows that there is a well defined superalgebra map 
\[
\varphi: \operatorname{Ser}_{r} \to \End_{\OBC}\left(\up^{r} \right)
\] given by 
\begin{align*}
\varphi\left(s_{i} \right) &= 
        \begin{tikzpicture}[baseline = 7.5pt, scale=0.5, color=\clr]
            \draw[->,thick] (0,0) to[out=up, in=down] (0,1.5);
        \end{tikzpicture}^{r-i-1}
          \begin{tikzpicture}[baseline = 7.5pt, scale=0.5, color=\clr]
            \draw[->,thick] (0,0) to[out=up, in=down] (1,1.5);
            \draw[->,thick] (1,0) to[out=up, in=down] (0,1.5);
        \end{tikzpicture} \;\;
        \begin{tikzpicture}[baseline = 7.5pt, scale=0.5, color=\clr]
            \draw[->,thick] (0,0) to[out=up, in=down] (0,1.5);
        \end{tikzpicture}^{i-1}, \\
\varphi\left(c_{i} \right) &=  \begin{tikzpicture}[baseline = 7.5pt, scale=0.5, color=\clr]
            \draw[->,thick] (0,0) to[out=up, in=down] (0,1.5);
        \end{tikzpicture}^{r-i}
          \begin{tikzpicture}[baseline = 7.5pt, scale=0.5, color=\clr]
            \draw[->,thick] (0,0) to[out=up, in=down] (0,1.5);
            \draw (0,0.7) \wdot;
        \end{tikzpicture} \;\;
        \begin{tikzpicture}[baseline = 7.5pt, scale=0.5, color=\clr]
            \draw[->,thick] (0,0) to[out=up, in=down] (0,1.5);
        \end{tikzpicture}^{i-1}.
\end{align*}  Using \cref{OBC basis theorem} it is straightforward to see that $\varphi$ is an isomorphism.  In particular, a direct count shows that there are $2^{r}(r!)$ normally ordered oriented Brauer diagrams of type $\up^{r}\to \up^{r}$ which matches the dimension of $\operatorname{Ser}_{r}$.
\end{proof}

More generally, given nonnegative integers $r, s$, let $BC_{r,s}$ denote the \emph{walled Brauer-Clifford superalgebra}. This superalgebra was introduced in \cite{JK} where it is denoted $\vec{B}_{r,s}$ and called the \emph{walled Brauer superalgebra}. Our notation and terminology is chosen so as to be consistent with that of \cite{BGHKW}. In the next result we assume $\k$ contains $\sqrt{-1}$.

\begin{corollary}\label{C:EndOBCisomorphictoWalledBrauerCliffordAlgebra} Let $\mathbf{a}$ be a word consisting of $r$ $\uparrow$'s and $s$ $\downarrow$'s.  Then 
\[
\End_{\OBC }(\mathbf{a}) \cong BC_{r,s}
\] as superalgebras.
\end{corollary}

\begin{proof}  We first note  that if $\mathbf{a}$ and $\mathbf{b}$ are two words with $r$ $\up$'s and $s$ $\down$'s, then 
\[
\End_{\OBC}(\mathbf{a}) \cong \End_{\OBC }(\mathbf{b}).
\]  The isomorphism is given by applying suitable symmetric braidings for $\OBC$.
 Consequently, assume $\mathbf{a}=\down^{s}\up^{r}$. By \cite[Theorem 5.1]{JK}, $BC_{r,s}$ is generated by even generators $s_{1}, \dotsc , s_{r-1}, s_{r+1}, \dotsc , s_{r+s-1}, e_{r,r+1}$ and odd generators $c_{1}, \dotsc , c_{r+s}$ subject to an explicit set of relations.  By checking relations we see that there is a superalgebra map $\alpha: BC_{r,s} \to \End_{\OBC}\left(\down^{s}\up^{r} \right)$ given by
\begin{align*}
    \alpha\left(s_{i} \right) &= \begin{cases}
        \begin{tikzpicture}[baseline = 7.5pt, scale=0.5, color=\clr]
            \draw[<-,thick] (0,0) to[out=up, in=down] (0,1.5);
        \end{tikzpicture}^{s}
        \begin{tikzpicture}[baseline = 7.5pt, scale=0.5, color=\clr]
            \draw[->,thick] (0,0) to[out=up, in=down] (0,1.5);
        \end{tikzpicture}^{r-i-1}
        \begin{tikzpicture}[baseline = 7.5pt, scale=0.5, color=\clr]
            \draw[->,thick] (0,0) to[out=up, in=down] (1,1.5);
            \draw[->,thick] (1,0) to[out=up, in=down] (0,1.5);
        \end{tikzpicture} \;\;
        \begin{tikzpicture}[baseline = 7.5pt, scale=0.5, color=\clr]
            \draw[->,thick] (0,0) to[out=up, in=down] (0,1.5);
        \end{tikzpicture}^{i-1}, & i=1, \dots , r-1;\\
        \begin{tikzpicture}[baseline = 7.5pt, scale=0.5, color=\clr]
            \draw[<-,thick] (0,0) to[out=up, in=down] (0,1.5);
        \end{tikzpicture}^{r+s-i-1}
        \begin{tikzpicture}[baseline = 7.5pt, scale=0.5, color=\clr]
            \draw[<-,thick] (0,0) to[out=up, in=down] (1,1.5);
            \draw[<-,thick] (1,0) to[out=up, in=down] (0,1.5);
        \end{tikzpicture} \;\;       
        \begin{tikzpicture}[baseline = 7.5pt, scale=0.5, color=\clr]
            \draw[<-,thick] (0,0) to[out=up, in=down] (0,1.5);
        \end{tikzpicture}^{i-r-1}
        \begin{tikzpicture}[baseline = 7.5pt, scale=0.5, color=\clr]
            \draw[->,thick] (0,0) to[out=up, in=down] (0,1.5);
        \end{tikzpicture}^{r}, &  i=r+1, \dots , r+s-1; 
    \end{cases},\\
    \alpha\left(e_{r,r+1} \right)&=  
        \begin{tikzpicture}[baseline = 7.5pt, scale=0.5, color=\clr]
            \draw[<-,thick] (0,0) to[out=up, in=down] (0,1.5);
        \end{tikzpicture}^{s-1}
        \begin{tikzpicture}[baseline = 7.5pt, scale=0.5, color=\clr]
            \draw[<-,thick] (0,0) to[out=up, in =left] (0.5,0.5) to[out=right,in=up] (1,0);
            \draw[->,thick] (0,1.5) to[out=down, in =left] (0.5,1) to[out=right,in=down] (1,1.5);
        \end{tikzpicture} \;\;
        \begin{tikzpicture}[baseline = 7.5pt, scale=0.5, color=\clr]
            \draw[->,thick] (0,0) to[out=up, in=down] (0,1.5);
        \end{tikzpicture}^{r-1},\\
    \alpha\left(c_{i} \right) &= \begin{cases}
        \sqrt{-1}~
        \begin{tikzpicture}[baseline = 7.5pt, scale=0.5, color=\clr]
            \draw[<-,thick] (0,0) to[out=up, in=down] (0,1.5);
        \end{tikzpicture}^{s}
        \begin{tikzpicture}[baseline = 7.5pt, scale=0.5, color=\clr]
            \draw[->,thick] (0,0) to[out=up, in=down] (0,1.5);
        \end{tikzpicture}^{r-i}
        \begin{tikzpicture}[baseline = 7.5pt, scale=0.5, color=\clr]
            \draw[->,thick] (0,0) to[out=up, in=down] (0,1.5);
            \draw (0,0.7) \wdot;
        \end{tikzpicture} \;\;
        \begin{tikzpicture}[baseline = 7.5pt, scale=0.5, color=\clr]
            \draw[->,thick] (0,0) to[out=up, in=down] (0,1.5);
        \end{tikzpicture}^{i-1}, & i=1, \dotsc , r; \\
        \sqrt{-1}~
        \begin{tikzpicture}[baseline = 7.5pt, scale=0.5, color=\clr]
            \draw[<-,thick] (0,0) to[out=up, in=down] (0,1.5);
        \end{tikzpicture}^{r+s-i}
        \begin{tikzpicture}[baseline = 7.5pt, scale=0.5, color=\clr]
            \draw[<-,thick] (0,0) to[out=up, in=down] (0,1.5);
            \draw (0,0.7) \wdot;
        \end{tikzpicture} \;\;
        \begin{tikzpicture}[baseline = 7.5pt, scale=0.5, color=\clr]
            \draw[<-,thick] (0,0) to[out=up, in=down] (0,1.5);
        \end{tikzpicture}^{i-r-1}
        \begin{tikzpicture}[baseline = 7.5pt, scale=0.5, color=\clr]
            \draw[->,thick] (0,0) to[out=up, in=down] (0,1.5);
        \end{tikzpicture}^{r}, & i=r+1, \dotsc , r+s.
    \end{cases}  
\end{align*}
As for the previous theorem one can use \cref{OBC basis theorem} to verify that $\alpha$ is an isomorphism.  We leave the details to the reader. 
\end{proof}

\subsection{The supercategory \texorpdfstring{$\AOBC$}{AOBC}}\label{subsection: AOBC}
We now introduce an affine version of $\OBC$.

\begin{definition}\label{AOBC defn}
    The \emph{degenerate affine oriented Brauer-Clifford supercategory} $\AOBC$ is the $\k$-linear strict monoidal supercategory generated by two objects $\up, \down$; four even morphisms $\lcup:\unit\to\up\down, \lcap: \down\up\to\unit, \swap: \up\up\to\up\up, \xdot:\up\to\up$; and one odd morphism $\cldot:\up\to\up$ subject to \cref{OB relations 1 (symmetric group)}, \cref{OB relations 2 (zigzags and invertibility)}, \cref{OBC relations}, and the following relations:
    \begin{equation}\label{AOBC relations}
        \begin{tikzpicture}[baseline = 7.5pt, scale=0.5, color=\clr]
            \draw[->,thick] (0,0) to[out=up, in=down] (0,1.5);
            \draw (0,1) \bdot;
            \draw (0,0.5) \wdot;
        \end{tikzpicture}
        ~=-~
        \begin{tikzpicture}[baseline = 7.5pt, scale=0.5, color=\clr]
            \draw[->,thick] (0,0) to[out=up, in=down] (0,1.5);
            \draw (0,1) \wdot;
            \draw (0,0.5) \bdot;
        \end{tikzpicture}
        ,\qquad
        \begin{tikzpicture}[baseline = 7.5pt, scale=0.5, color=\clr]
            \draw[->,thick] (0,0) to[out=up, in=down] (1,1.5);
            \draw[->,thick] (1,0) to[out=up, in=down] (0,1.5);
            \draw (0.2,1) \bdot;
        \end{tikzpicture}
        ~-~
        \begin{tikzpicture}[baseline = 7.5pt, scale=0.5, color=\clr]
            \draw[->,thick] (0,0) to[out=up, in=down] (1,1.5);
            \draw[->,thick] (1,0) to[out=up, in=down] (0,1.5);
            \draw (0.8,0.5) \bdot;
        \end{tikzpicture}
        ~=~
        \begin{tikzpicture}[baseline = 7.5pt, scale=0.5, color=\clr]
            \draw[->,thick] (0,0) to[out=up, in=down] (0,1.5);
            \draw[->,thick] (1,0) to[out=up, in=down] (1,1.5);
        \end{tikzpicture}
        ~-~
        \begin{tikzpicture}[baseline = 7.5pt, scale=0.5, color=\clr]
            \draw[->,thick] (0,0) to[out=up, in=down] (0,1.5);
            \draw[->,thick] (1,0) to[out=up, in=down] (1,1.5);
            \draw (0,0.7) \wdot;
            \draw (1,0.7) \wdot;
        \end{tikzpicture}.
    \end{equation} 
\end{definition} 

We define the downward analogue of $\xdot$ as follows:
\begin{equation}\label{down black dot}
    \begin{tikzpicture}[baseline = 12pt, scale=0.5, color=\clr]
        \draw[->,thick] (0,2) to (0,0);
        \draw (0,1) \bdot;
    \end{tikzpicture}
    ~:=~
    \begin{tikzpicture}[baseline = 12pt, scale=0.5, color=\clr]
        \draw[->,thick] 
            (1,2) to (1,1)
            to[out=down,in=right] (0.75,0.25)
            to[out=left,in=down] (0.5,1)
            to[out=up,in=right] (0.25,1.75)
            to[out=left,in=up] (0,1) to (0,0);
        \draw (0.5,1) \bdot;
    \end{tikzpicture}
\end{equation}

We define a \emph{dotted oriented Brauer-Clifford diagram} (resp.\ \emph{with bubbles}) to be an oriented Brauer diagram (resp.\ with bubbles) with finitely many $\emptydot$'s and $\fulldot$'s on its segments. For example, \cref{two AOBC diagrams} shows two dotted oriented Brauer-Clifford diagrams with bubbles of type $\down^2\up^2\to\down^2\up^3\down$. 
Given a nonnegative integer $k$, we will draw a $\fulldot$ labeled by $\color{\clr}k$ to denote $k$ $\fulldot$'s on a strand; that is, the vertical composition of $k$ $\xdot$'s or $k$ $\downxdot$'s. For example, the diagram on the right of \cref{two AOBC diagrams} could have been drawn as
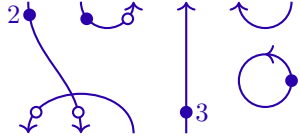
\begin{equation*}
    \begin{tikzpicture}[baseline = 25pt, scale=0.35, color=\clr]
        \draw[<-,thick] (2,0) to[out=up,in=down] (0,5);
        \draw[->,thick] (6,0) to[out=up,in=down] (6,5);
        \draw[<-,thick] (0,0) to[out=up,in=left] (2,1.5) to[out=right,in=up] (4,0);
        \draw[->,thick] (2,5) to[out=down,in=left] (3,4) to[out=right,in=down] (4,5);
        \draw[->,thick] (9.1,3) to (9,3) to[out=left,in=up] (8,2) to[out=down,in=left] (9,1)
                        to[out=right,in=down] (10,2) to[out=up,in=right] (9,3);
        \draw[->,thick] (10,5) to[out=down,in=right] (9,4) to[out=left,in=down] (8,5);
        \draw (0.05,4.5) \bdot;
        \draw (-0.55,4.5) node{2};
        \draw (2.22,4.35) \bdot;
        \draw (3.78,4.35) \wdot;
        \draw (10,2) \bdot;
        \draw (6,0.8) \bdot;
        \draw (6.6,0.8) node{3};
        \draw (0.3,0.8) \wdot;
        \draw (1.9,0.8) \wdot;
    \end{tikzpicture}
\end{equation*}
We can interpret any dotted oriented Brauer-Clifford diagram with bubbles as a morphism in $\AOBC$, and the Hom-spaces in $\AOBC$ are certainly spanned by all dotted oriented Brauer-Clifford diagrams with bubbles. For all $\ob a, \ob b\in\wrd$ and each $k\in\Z$ we let $\Hom_{\AOBC}(\ob a, \ob b)_{\leq k,\bar{0}}$ (resp.~$\Hom_{\AOBC}(\ob a, \ob b)_{\leq k,\bar{1}}$) denote the $\k$-span of all dotted oriented Brauer-Clifford diagrams with bubbles of type $\ob a\to\ob b$ having at most $k$ $\fulldot$'s and an even (resp.~odd) number of $\emptydot$'s. This gives $\AOBC$ the structure of a filtered monoidal supercategory (see \cref{subsection: graded and filtered super-cats}). Given a dotted oriented Brauer-Clifford diagram with bubbles, $d$, we will write $\deg(d)$ for the number of $\fulldot$'s appearing in $d$. For example, $\deg(d)=7$ when $d$ is either of the diagrams in \cref{two AOBC diagrams}. 

\begin{definition}\label{D:N.O. dotted OBC diagram}
A dotted oriented Brauer-Clifford diagram with bubbles is \emph{normally ordered} if 
\begin{itemize} 
    \item[$\diamond$] removing all bubbles and all $\fulldot$'s results in a normally ordered oriented Brauer-Clifford diagram;
    \item[$\diamond$] each bubble has zero $\emptydot$'s and an odd number of $\fulldot$'s, are crossing-free, counterclockwise, and there are no other strands shielding it from the rightmost edge of the picture;
    \item[$\diamond$] each $\fulldot$ is either on a bubble or on an inward-pointing boundary segment; 
    \item[$\diamond$] whenever a $\fulldot$ and a $\emptydot$ appear on a segment that is both inward and outward-pointing, the $\emptydot$ appears ahead of the $\fulldot$ in the direction of the orientation. 
\end{itemize}
\end{definition}

For example, in \cref{two AOBC diagrams} the diagram on the right is normally ordered and the one on the left is not. We say that two normally ordered dotted oriented Brauer-Clifford diagrams with bubbles are \emph{equivalent} if their underlying oriented Brauer diagrams with bubbles are equivalent and their corresponding strands have the same number of $\emptydot$'s and $\fulldot$'s. We can now state our main result:

\begin{theorem}\label{AOBC basis theorem}
    For any $\ob a, \ob b\in\wrd$ the superspace $\Hom_\AOBC(\ob a,\ob b)$ has basis given by equivalence classes of normally ordered dotted oriented Brauer-Clifford diagrams with bubbles of type $\ob a\to\ob b$.
\end{theorem}

\subsection{Consequences of Theorem \ref{AOBC basis theorem} for \texorpdfstring{$\OBC$}{OBC}}\label{Consequences of main result} 
For this subsection we assume  \cref{AOBC basis theorem}. 
Let $\OB(0)$ denote the quotient of $\OB$ obtained by requiring \cref{dim equals 0}. Then the Hom-spaces in $\OB(0)$ have a basis given by equivalence classes of oriented Brauer diagrams \emph{without bubbles}. Moreover, there are obvious monoidal superfunctors 
\begin{equation}\label{OB to OBC to AOBC}
    \OB(0) \to \OBC \to \AOBC.
\end{equation}
We saw in \cref{subsection: OBC} that the Hom-spaces in $\OBC$ are spanned by normally ordered oriented Brauer-Clifford diagrams. It follows from \cref{AOBC basis theorem} that the image of the equivalence classes of normally ordered oriented Brauer-Clifford diagrams under \cref{OB to OBC to AOBC} are linearly independent in $\AOBC$, whence they are linearly independent in $\OBC$. \cref{OBC basis theorem} follows. Moreover, it follows that the monoidal superfunctors in \cref{OB to OBC to AOBC} are both faithful.

\subsection{Normally ordered diagrams span}\label{N.O. diagrams span}
In \cref{subsection: OBC} we showed that normally ordered oriented Brauer-Clifford diagrams span the Hom-spaces in $\OBC$. The goal of this subsection is to show that normally ordered dotted oriented Brauer-Clifford diagrams with bubbles span the Hom-spaces in $\AOBC$. We start with an analogue of \cref{prop: OBC slides} for $\AOBC$. In the proof, we will make use of \cref{prop: OBC slides} without reference. 

\begin{proposition}\label{prop: AOBC slides}
    The following relations hold in $\AOBC$: 
    \begin{align}
        \label{left cap/cup x-slide}

            ~.
    \end{align}
\end{proposition}

\begin{proof}
    The proof of \cref{left cap/cup x-slide} is similar to that of \cref{left cap/cup c-slide}. 
    In \cref{up cross x-slide}, \cref{LR cross x-slide 1}, \cref{LR cross x-slide 2}, and \cref{down cross x-slide} the right equality is obtained from the left by composing on top and bottom with $\swap, \rswap, \rswap$, and $\dswap$ respectively. The left of \cref{up cross x-slide} is one of the defining relations of $\AOBC$. 
    The left of \cref{LR cross x-slide 1} is verified below:
    \begin{align*}
        %
        ~.
    \end{align*}
    Proofs for the left equalities in \cref{LR cross x-slide 2,down cross x-slide} are similar.
    Finally, we verify the right of \cref{right cap/cup x-slide} below. 
    \begin{align*}
        %
        ~.
    \end{align*}
    The left of \cref{right cap/cup x-slide} is similar. 
\end{proof}

Note that \cref{left cap/cup c-slide}-\cref{right cap/cup c-slide} and \cref{left cap/cup x-slide}-\cref{right cap/cup x-slide} allow us to draw $\emptydot$'s and $\fulldot$'s on local extrema of strands in dotted oriented Brauer-Clifford diagrams with bubbles without ambiguity. We will do so whenever convenient. 

We have a downward analogue of the left relation in \cref{AOBC relations}:
\begin{equation}\label{down black white dot}
    %
    ~.
\end{equation}

Before proving that normally ordered dotted oriented Brauer-Clifford diagrams with bubbles span arbitrary Hom-spaces in $\AOBC$, we first consider diagrams of type $\unit\to\unit$. In this case we must show the superalgebra $\End_\AOBC(\unit)$ is generated by $\Delta_1,\Delta_3,\Delta_5,\ldots$ where 
\begin{equation}\label{Delta sub k}
    \Delta_k=
    %
    .
\end{equation} 
First, note that  \cref{prop: OBC slides,,prop: AOBC slides} can be used to express any dotted oriented Brauer-Clifford diagram with bubbles of type $\unit\to\unit$ in terms of crossing-free unnested bubbles. For example,
\begin{align*}
    %
    ~.
\end{align*}
Moreover, we can express any clockwise bubble in terms of counterclockwise ones using computations similar to the following:
\begin{equation*}
    %
    \end{equation*}
    for all integers $k\geq 0$. 
    The computation above shows $\Delta_{-1}\Delta'_1+\Delta_1\Delta'_{-1}=0$.
    A similar computation can be used to show  
    \begin{equation*}
        \sum_{0\leq i\leq \frac{k+1}{2}}\Delta_{2i-1}\Delta'_{k-2i}=0
        % \Delta_k-\Delta'_k=\sum_{0<i<k/2}\Delta_{2i-1}\Delta'_{k-2i}
    \end{equation*}
    whenever $k$ is a nonnegative odd integer.
\end{remark}

Now, using \cref{OBC relations,reduce two down Cliffords,AOBC relations,down black white dot}  we can reduce the number of $\emptydot$'s on any bubble to zero or one.  Finally, the following proposition shows that the only nonzero counterclockwise bubbles with at most one $\emptydot$ are $\Delta_1,\Delta_3,\Delta_5,\ldots$, whence they generate $\End_\AOBC(\unit)$. 

\begin{proposition}\label{prop: bubble reducing}
    The following relations hold in $\AOBC$ for any nonnegative integer $k$:
    \begin{equation}\label{vanishing bubbles}
        \begin{tikzpicture}[baseline = 5pt, scale=0.5, color=\clr]
            \draw[->,thick] (0.6,1) to (0.5,1) to[out=left,in=up] (0,0.5)
                            to[out=down,in=left] (0.5,0) 
                            to[out=right,in=down] (1,0.5)
                            to[out=up,in=right] (0.5,1);
            \draw (1,0.5) \bdot;
            \draw (0,0.5) \wdot;
            \draw (1.5,0.5) node{\footnotesize{$k$}};
        \end{tikzpicture}
        =0,\qquad
        \begin{tikzpicture}[baseline = 5pt, scale=0.5, color=\clr]
            \draw[->,thick] (0.6,1) to (0.5,1) to[out=left,in=up] (0,0.5)
                            to[out=down,in=left] (0.5,0) 
                            to[out=right,in=down] (1,0.5)
                            to[out=up,in=right] (0.5,1);
            \draw (1,0.5) \bdot;
            \draw (1.75,0.5) node{\footnotesize{$2k$}};
        \end{tikzpicture}
        =0.
    \end{equation}
\end{proposition}

\begin{proof}
    The left relation when $k=0$ follows from the right of \cref{OBC relations}. If $k>0$ the left relation follows from the calculation below since $\chr\k\not=2$: 
    \begin{equation*}
       \begin{tikzpicture}[baseline = 5pt, scale=0.5, color=\clr]
            \draw[->,thick] (0.6,1) to (0.5,1) to[out=left,in=up] (0,0.5)
                            to[out=down,in=left] (0.5,0) 
                            to[out=right,in=down] (1,0.5)
                            to[out=up,in=right] (0.5,1);
            \draw (1,0.5) \bdot;
            \draw (0,0.5) \wdot;
            \draw (1.5,0.5) node{\footnotesize{$k$}};
        \end{tikzpicture}
        \stackrel{\cref{left cap/cup x-slide}}{=}
        \begin{tikzpicture}[baseline = 5pt, scale=0.5, color=\clr]
            \draw[->,thick] (0.6,1) to (0.5,1) to[out=left,in=up] (0,0.5)
                            to[out=down,in=left] (0.5,0) 
                            to[out=right,in=down] (1,0.5)
                            to[out=up,in=right] (0.5,1);
            \draw (1,0.5) \bdot;
            \draw (0.1,0.75) \bdot;
            \draw (0.1,0.25) \wdot;
            \draw (2.1,0.5) node{\footnotesize{$k-1$}};
        \end{tikzpicture}
        \stackrel{\cref{down black white dot}}{=}-~
        \begin{tikzpicture}[baseline = 5pt, scale=0.5, color=\clr]
            \draw[->,thick] (0.6,1) to (0.5,1) to[out=left,in=up] (0,0.5)
                            to[out=down,in=left] (0.5,0) 
                            to[out=right,in=down] (1,0.5)
                            to[out=up,in=right] (0.5,1);
            \draw (1,0.5) \bdot;
            \draw (0.1,0.75) \wdot;
            \draw (0.1,0.25) \bdot;
            \draw (2.1,0.5) node{\footnotesize{$k-1$}};
        \end{tikzpicture}
        \stackrel{\cref{right cap/cup x-slide}}{=}-~
        \begin{tikzpicture}[baseline = 5pt, scale=0.5, color=\clr]
            \draw[->,thick] (0.6,1) to (0.5,1) to[out=left,in=up] (0,0.5)
                            to[out=down,in=left] (0.5,0) 
                            to[out=right,in=down] (1,0.5)
                            to[out=up,in=right] (0.5,1);
            \draw (1,0.5) \bdot;
            \draw (0,0.5) \wdot;
            \draw (1.5,0.5) node{\footnotesize{$k$}};
        \end{tikzpicture}
        .
    \end{equation*}
    The right relation follows from the calculation below:
    \begin{align*}
        \begin{tikzpicture}[baseline = 13pt, scale=0.5, color=\clr]
            \draw[->,thick] (0.6,2) to (0.5,2) to[out=left,in=up] (0,1.5) to (0,0.5)
                            to[out=down,in=left] (0.5,0) 
                            to[out=right,in=down] (1,0.5) to (1,1.5)
                            to[out=up,in=right] (0.5,2);
            % \draw (1,1.5) \wdot;
            \draw (1,1)   \bdot;
            % \draw (1,0.5) \bdot;
            \draw (1.7,1) node{\footnotesize{$2k$}};
        \end{tikzpicture}
        &
        ~\hspace{-8pt}\stackrel{\cref{OBC relations}}{=}
        \begin{tikzpicture}[baseline = 13pt, scale=0.5, color=\clr]
            \draw[->,thick] (0.6,2) to (0.5,2) to[out=left,in=up] (0,1.5) to (0,0.5)
                            to[out=down,in=left] (0.5,0) 
                            to[out=right,in=down] (1,0.5) to (1,1.5)
                            to[out=up,in=right] (0.5,2);
            \draw (1,1.5) \wdot;
            \draw (1,1)   \wdot;
            \draw (1,0.5) \bdot;
            \draw (1.7,0.5) node{\footnotesize{$2k$}};
        \end{tikzpicture}
        ~\hspace{-8pt}\stackrel{\cref{left cap/cup c-slide}}{=}
        \begin{tikzpicture}[baseline = 13pt, scale=0.5, color=\clr]
            \draw[->,thick] (0.6,2) to (0.5,2) to[out=left,in=up] (0,1.5) to (0,0.5)
                            to[out=down,in=left] (0.5,0) 
                            to[out=right,in=down] (1,0.5) to (1,1.5)
                            to[out=up,in=right] (0.5,2);
            \draw (0,1.5) \wdot;
            \draw (1,1)   \wdot;
            \draw (1,0.5) \bdot;
            \draw (1.7,0.5) node{\footnotesize{$2k$}};
        \end{tikzpicture}
        ~\hspace{-8pt}\stackrel{\cref{super-interchange}}{=}-~
        \begin{tikzpicture}[baseline = 13pt, scale=0.5, color=\clr]
            \draw[->,thick] (0.6,2) to (0.5,2) to[out=left,in=up] (0,1.5) to (0,0.5)
                            to[out=down,in=left] (0.5,0) 
                            to[out=right,in=down] (1,0.5) to (1,1.5)
                            to[out=up,in=right] (0.5,2);
            \draw (1,1.5) \wdot;
            \draw (1,1)   \bdot;
            \draw (0,0.5) \wdot;
            \draw (1.7,1) node{\footnotesize{$2k$}};
        \end{tikzpicture}
        \\&
        ~\hspace{-8pt}\stackrel{\cref{right cap/cup c-slide}}{=}-~
        \begin{tikzpicture}[baseline = 13pt, scale=0.5, color=\clr]
            \draw[->,thick] (0.6,2) to (0.5,2) to[out=left,in=up] (0,1.5) to (0,0.5)
                            to[out=down,in=left] (0.5,0) 
                            to[out=right,in=down] (1,0.5) to (1,1.5)
                            to[out=up,in=right] (0.5,2);
            \draw (1,1.5) \wdot;
            \draw (1,1)   \bdot;
            \draw (1,0.5) \wdot;
            \draw (1.7,1) node{\footnotesize{$2k$}};
        \end{tikzpicture}
        ~\hspace{-8pt}\stackrel{\cref{AOBC relations}}{=}-(-1)^{2k}~
        \begin{tikzpicture}[baseline = 13pt, scale=0.5, color=\clr]
            \draw[->,thick] (0.6,2) to (0.5,2) to[out=left,in=up] (0,1.5) to (0,0.5)
                            to[out=down,in=left] (0.5,0) 
                            to[out=right,in=down] (1,0.5) to (1,1.5)
                            to[out=up,in=right] (0.5,2);
            \draw (1,1.5) \wdot;
            \draw (1,1)   \wdot;
            \draw (1,0.5) \bdot;
            \draw (1.7,0.5) node{\footnotesize{$2k$}};
        \end{tikzpicture}
        ~\hspace{-8pt}\stackrel{\cref{OBC relations}}{=}-~
        \begin{tikzpicture}[baseline = 13pt, scale=0.5, color=\clr]
            \draw[->,thick] (0.6,2) to (0.5,2) to[out=left,in=up] (0,1.5) to (0,0.5)
                            to[out=down,in=left] (0.5,0) 
                            to[out=right,in=down] (1,0.5) to (1,1.5)
                            to[out=up,in=right] (0.5,2);
            \draw (1,1) \bdot;
            \draw (1.7,1) node{\footnotesize{$2k$}};
        \end{tikzpicture}
        .
    \end{align*}
\end{proof}

The following lemma implies that normally ordered oriented Brauer-Clifford diagrams with bubbles span the Hom-spaces of $\AOBC$.

\begin{lemma}\label{lemma: normally ordered diagrams span}
    For any $\ob a, \ob b\in\wrd$ the superspace $\Hom_\AOBC(\ob a,\ob b)_{\leq k}$ is equal to the $\k$-span of all equivalence classes of normally ordered dotted oriented Brauer-Clifford diagrams with bubbles of type $\ob a\to\ob b$ with at most $k$ $\fulldot$'s.
\end{lemma}

\begin{proof}
    Let $d$ denote a dotted oriented Brauer-Clifford diagram with bubbles with $\deg(d)\leq k$.
    Let $d'$ denote the diagram obtained from $d$ as follows. First, remove pairs of $\emptydot$'s appearing on the same strand until each strand has at most one $\emptydot$. Next, freely slide each bubble to the right of the picture and redraw them so that they are crossing-free and counterclockwise, without changing the number of $\emptydot$'s and $\fulldot$'s appearing on each bubble. Finally, on every non-bubble strand freely slide all $\emptydot$'s (resp.~$\fulldot$'s) until they lie on an outward-pointing (resp.~inward-pointing) segment. For example, if $d$ is the diagram on the left of \cref{two AOBC diagrams}, then $d'$ is the diagram on the right. It follows from \cref{prop: bubble reducing} that $d'$ is either zero or normally ordered with $\deg(d')=\deg(d)$. Moreover, by  \cref{prop: OBC slides,prop: AOBC slides} along with \cref{OBC relations,reduce two down Cliffords,down black white dot} we have $d=\pm d'+d''$ where $d''$ is a linear combination of dotted oriented Brauer-Clifford diagrams with bubbles each of which having fewer than $\deg(d)$ $\fulldot$'s. The result now follows by inducting on $\deg(d)$. 
\end{proof}

\section{Connection to representations of  Lie superalgebras of type Q}\label{S:Liesuperalgebras}

We next explain how the supercategories $\OBC$ and $\AOBC$ are connected with the representations of the Lie superalgebras of type $Q$.  In what follows we assume that $\sqrt{-1}$ is an element of $\k$.  This is only for convenience.  In particular, the arguments in \cref{SS:PositiveCharacteristic} imply the basis theorems hold in general.  

\subsection{The Lie superalgebra \texorpdfstring{$\fq$}{q}}\label{SS:Liesuperalgebras}

Fix a $\k$-superspace $V= V_{\0} \oplus V_{\1}$ with $\dim_{\k}(V_{\0}) = \dim_{\k}(V_{\1})=n$.   Fix a homogeneous basis $v_{1}, \dotsc , v_{n}, v_{\bar{1}}, \dotsc , v_{\bar{n}}$ with $\p{v_{i}}=\0$ and $\p{v_{\bar{i}}}=\1$ for $i=1, \dotsc , n$.  We write $I$ for the index set $\left\{1, \dotsc , n, \bar{1}, \dotsc , \bar{n} \right\}$ and $I_{0}$ for the index set $ \left\{1, \dotsc ,n \right\}$. 
% There is an involution on $I$ given by $i \mapsto \bar{i}$ and $\bar{i}\mapsto i$ for any $i=1, \dotsc , n$. \todo{Is this sentence necessary?}
We adopt the convention that $\bar{\bar{i}}=i$ for all $i \in I$.  Let $c: V\to V$ be the odd linear map given by $c(v_{i})=(-1)^{\p{v_{i}}}\sqrt{-1}v_{\bar{i}}$ for all $i \in I$. 

The vector space of all linear endomorphisms of $V$, $\gl (V)$, is naturally $\Z_{2}$-graded as in \cref{SS:superspaces}.  Furthermore, $\gl (V)$ is a Lie superalgebra under the graded commutator bracket; this, by definition, is given by $[x,y]=xy-(-1)^{\p{x}\p{y}}yx$ for all homogeneous $x,y \in \gl (V)$.  For $i,j \in I$ we write $e_{i,j} \in \gl (V)$ for the linear map $e_{i,j}(v_{k})= \delta_{j,k}v_{i}$. These are the matrix units and they form a homogeneous basis for $\gl(V)$ with $\p{e_{i,j}}=\p{v_{i}}+\p{v_{j}}$.

By definition $\fq (V)$ is the Lie subsuperalgebra of $\gl (V)$ given by 
\[
\fq (V) = \left\{x \in \gl (V) \mid [x,c]=0 \right\}.
\]  Then $\fq (V)$ has a homogenous basis given by $e_{i,j}^{\0}:=e_{i,j}+e_{\ibar,\jbar}$ and $e_{i,j}^{\1 }:=e_{\ibar,j}+e_{i,\jbar}$ for $1\leq i,j\leq n$.  Set $\tilde{e}_{i,j}^{\0}:=e_{i,j}-e_{\ibar,\jbar}$ and $\tilde{e}_{i,j}^{\1}:=e_{\ibar,j}-e_{i,\jbar}$ for $1\leq i,j\leq n$.  These are homogeneous elements of $\gl (V)$ and, together with our basis for $\fq(V)$, provide a homogeneous basis for $\gl (V)$.   Note that $\p{\tilde{e}_{i,j}^{\ep}}=\p{e_{i,j}^{\ep}} = \ep$ for all $1 \leq i,j \leq n$ and $\ep \in \Z_{2}$.

Let $U(\fq )$ denote the universal enveloping superalgebra of the Lie superalgebra $\fq = \fq (V)$.  The superalgebra $U(\fq )$ has a homogeneous PBW basis given by all ordered monomials in the elements $(e_{i,j}^{\ep})^{r}$ (with $1 \leq i, j \leq n$, and $r \geq 1$ if $\ep=\0$ or $r = 0,1$ if $\ep=\1$). Set $U(\fh)$, $U(\fn )$, and $U(\fn^{-})$, respectively,  to be the subsuperalgebras generated by
$\left\{ e_{i,i}^{\ep}  \mid i=1, \dotsc, n, \ep \in \Z_{2} \right\},
\left\{e_{i,j}^{\ep} \mid 1 \leq i < j \leq n, \ep \in \Z_{2} \right\}$, and
% \text{ and }
$\left\{ e_{i,j}^{\ep} \mid 1 \leq j < i \leq n, \ep \in \Z_{2} \right\}$.
% \begin{gather*}
% \left\{ e_{i,i}^{\ep}  \mid i=1, \dotsc, n, \ep \in \Z_{2} \right\}, \\
% \left\{e_{i,j}^{\ep} \mid 1 \leq i < j \leq n, \ep \in \Z_{2} \right\}, 
% \end{gather*}
% and
% \begin{gather*}
% \left\{ e_{i,j}^{\ep} \mid 1 \leq j < i \leq n, \ep \in \Z_{2} \right\}.
% \end{gather*}
The PBW basis implies that there is a triangular decomposition $U(\fq )= U(\fn^{-}) \otimes U(\fh) \otimes U(\fn )$.

A $U(\fq )$-supermodule is a $\k$-superspace $M=M_{\0}\oplus M_{\1}$ with an action by $U(\fq)$ which respects the $\Z_{2}$-grading in that $U(\fq)_{\ep}M_{\ep'} \subseteq M_{\ep+\ep'}$ for all $\ep, \ep' \in \Z_{2}$.  In particular, the superspace $V$ defined above is naturally a $U(\fq )$-supermodule.  A supermodule homomorphism is a $\k$-linear map $f:M \to N$ which satisfies $f(am) = (-1)^{\p{f}\p{a}}af(m)$ for all homogeneous $a \in U(\fq )$ and $m \in M$.  Note that homomorphisms are not assumed to preserve parity. However, $\Hom$-spaces are naturally $\Z_{2}$-graded as in \cref{SS:superspaces}.  Let $\Uksmod$ denote the supercategory of all $U(\fq )$-supermodules. We will write $\Hom_{U(\fq )}(M,M')=\Hom_\Uksmod(M,M')$. 

The $\k$-superalgebra $U(\fq)$ is a Hopf superalgebra.  In particular, given $U(\fq )$-supermodules $M$ and $M'$, the action of a homogeneous $x\in \fq \subseteq  U(\fq )$ on $M \otimes M'$ is given by $x.(m\otimes m') = (x.m)\otimes m' + (-1)^{\p{x}\p{m}} m \otimes (x.m')$ for homogeneous $m \in M$ and $m'\in M'$.  The unit object is given by viewing $\k$ as a superspace concentrated in parity $\0$ and with trivial $U(\fq)$ action.  The symmetric braiding is given by the graded flip map.  Thus $\Uksmod$ is a symmetric braided monoidal supercategory.  

The antipode $\sigma: U(\fq ) \to U(\fq )$ is given by $\sigma(x) = -x$ for $x \in \fq$.  Using the antipode each finite-dimensional $U(\fq)$-supermodule $M$ has a dual given by $M^{*}=\Hom_{\k}(M, \k)$ with the evaluation and coevaluation maps given by the same formulas as for superspaces (see \cref{SS:superspaces}).  In particular, $V$ is finite-dimensional and so admits a dual, $V^{*}$.  

\subsection{Mixed Schur-Weyl-Sergeev duality}\label{SS: mixed SWS duality}

As \cref{OBC basis theorem,C:OBC End is isomorphic to Sergeev Algebra} will be used to prove the following results, let us point out that this section is not used in the proof of \cref{AOBC basis theorem,OBC basis theorem,C:OBC End is isomorphic to Sergeev Algebra}.  

There is a monoidal superfunctor 
\begin{equation*}
    \Phi:\OBC\to\Uksmod 
\end{equation*}
mapping the objects $\up,\down$ to the superspaces $V, V^*$ respectively, and defined on morphisms by 
\begin{align*}
    \Phi\left(\lcup\right)&:\unit \rightarrow V \otimes V^*, \quad & 1 &\mapsto \sum_{i\in I} v_i\otimes v^*_i,\\
    \Phi\left(\lcap\right)&:V^* \otimes V \rightarrow \unit,
    \quad&
    f\otimes v  &\mapsto f(v),\\
    \Phi\left(\swap\right)&:V \otimes V  \rightarrow V \otimes V,\quad&
    u \otimes v  &\mapsto (-1)^{\p{u}\p{v}} v\otimes u,\\
    \Phi\left(\cldot\right)&:V \rightarrow V,
    \quad&
    v  &\mapsto c(v).
\end{align*} 
Indeed, a direct check confirms that $\Phi$ respects the defining relations of $\OBC$. Given $\ob a\in\wrd$, we write $V^{\ob a}:=\Phi(\ob a)$. For example, $V^{\up\down\up}=V\otimes V^*\otimes V$. 

\begin{theorem} \label{T:MixedSchurWeylDuality} If the characteristic of the ground field $\k$ is zero, then
    $\Phi$ is full. 
\end{theorem}

\begin{proof}
    We are required to show 
    \begin{equation}\label{OBC to q Homs}
        \Phi:\Hom_{\OBC}(\ob a,\ob b)\to\Hom_{U(\fq)}(V^{\ob a},V^{\ob b})
    \end{equation}
    is surjective for all $\ob a,\ob b\in\wrd$. 
    Suppose $\ob a$ (resp.~$\ob b$) consists of $r_1$ (resp.~$r'_1$) $\up$'s and $r_2$ (resp.~$r'_2$) $\down$'s. Acting by the central element $\sum_{i=1}^{n} e_{i,i}^{\0}$  one sees that $\Hom_{U(\fq)}(V^{\ob a},V^{\ob b})=0$ unless $r_1+r_2'=r_1'+r_2$.   
    
    In the nontrivial case, set $r=r_1+r_2'=r_1'+r_2$ and consider the following:
    \begin{equation}\label{reduce Hom to End}
        \begin{CD}
            \Hom_\OBC(\ob a, \ob b) &@>>>&
            \Hom_\OBC(\down^{r_2}\up^{r_1}, \up^{r'_1}\down^{r'_2}) &@>>>&
            \End_\OBC(\up^r)\\
            @V \Phi VV&&
            @V \Phi VV&&
            @V \Phi VV\\
            \Hom_{U(\fq)}(V^{\ob a},V^{\ob b}) &@>>>&
            \Hom_{U(\fq)}(V^{\down^{r_2}\up^{r_1}}, V^{\up^{r'_1}\down^{r'_2}}) &@>>>&
            \End_{U(\fq)}(V^{\otimes r}).
        \end{CD}
    \end{equation}
    The horizontal maps are all isomorphisms of superspaces. 
    Indeed, the left horizontal maps are given by the symmetric braidings on $\OBC$ and $\Uksmod$. The right horizontal maps are the $\k$-linear isomorphisms that hold in any monoidal supercategory with duals. In particular, the top right horizontal map is the $\k$-linear isomorphism given on diagrams by 
    
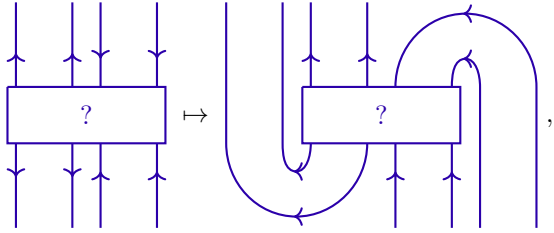
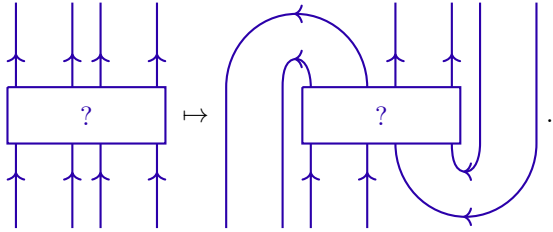
\[
        \begin{tikzpicture}[baseline = 40pt, scale=0.75, color=\clr]
            \draw (2.75,2) node{?};
            \draw[-,thick] (1.35,2.5) to (4.15,2.5) to (4.15,1.5) to (1.35,1.5) to (1.35, 2.5);
            \draw[->,thick] (1.5,2.5) to (1.5,3.1);
            \draw[->,thick] (2.5,2.5) to (2.5,3.1);
            \draw[-,thick] (3,2.5) to (3,3.1);
            \draw[-,thick] (4,2.5) to (4,3.1);
            \draw[-,thick] (1.5,3) to (1.5,4);
            \draw[-,thick] (2.5,3) to (2.5,4);
            \draw[<-,thick] (3,3) to (3,4);
            \draw[<-,thick] (4,3) to (4,4);
            \draw[->,thick] (1.5,1.5) to (1.5,0.9);
            \draw[->,thick] (2.5,1.5) to (2.5,0.9);
            \draw[-,thick] (3,1.5) to (3,0.9);
            \draw[-,thick] (4,1.5) to (4,0.9);
            \draw[-,thick] (1.5,1) to (1.5,0);
            \draw[-,thick] (2.5,1) to (2.5,0);
            \draw[<-,thick] (3,1) to (3,0);
            \draw[<-,thick] (4,1) to (4,0);
        \end{tikzpicture}
        ~\mapsto~
        \begin{tikzpicture}[baseline = 40pt, scale=0.75, color=\clr]
            \draw (2.75,2) node{?};
            \draw[-,thick] (1.35,2.5) to (4.15,2.5) to (4.15,1.5) to (1.35,1.5) to (1.35, 2.5);
            \draw[->,thick] (1.5,2.5) to (1.5,3.1);
            \draw[->,thick] (2.5,2.5) to (2.5,3.1);
            \draw[->,thick] (4.5,0) to (4.5,2.5) to[out=up,in=right] (4.2,3);
            \draw[-,thick] (4.25,3) to[out=left,in=up] (4,2.5);
            \draw[->,thick] (5.5,0) to (5.5,2.5) to[out=up,in=right] (4.2,3.8);
            \draw[-,thick] (4.25,3.8) to[out=left,in=up] (3,2.5);
            \draw[-,thick] (1.5,3) to (1.5,4);
            \draw[-,thick] (2.5,3) to (2.5,4);
            \draw[-,thick] (3,1.5) to (3,0.9);
            \draw[-,thick] (4,1.5) to (4,0.9);
            \draw[->,thick] (1.5,1.5) to[out=down,in=right] (1.2,1);
            \draw[-,thick] (1.25,1) to[out=left,in=down] (1,1.5) to (1,4);
            \draw[->,thick] (2.5,1.5) to[out=down,in=right] (1.2,0.2);
            \draw[-,thick] (1.25,0.2) to[out=left,in=down] (0,1.5) to (0,4);
            \draw[<-,thick] (3,1) to (3,0);
            \draw[<-,thick] (4,1) to (4,0);
        \end{tikzpicture}
        ~,
    \]
    with inverse mapping 
    \[
        \begin{tikzpicture}[baseline = 40pt, scale=0.75, color=\clr]
            \draw (2.75,2) node{?};
            \draw[-,thick] (1.35,2.5) to (4.15,2.5) to (4.15,1.5) to (1.35,1.5) to (1.35, 2.5);
            \draw[->,thick] (1.5,2.5) to (1.5,3.1);
            \draw[->,thick] (2.5,2.5) to (2.5,3.1);
            \draw[->,thick] (3,2.5) to (3,3.1);
            \draw[->,thick] (4,2.5) to (4,3.1);
            \draw[-,thick] (1.5,3) to (1.5,4);
            \draw[-,thick] (2.5,3) to (2.5,4);
            \draw[-,thick] (3,3) to (3,4);
            \draw[-,thick] (4,3) to (4,4);
            \draw[-,thick] (1.5,1.5) to (1.5,0.9);
            \draw[-,thick] (2.5,1.5) to (2.5,0.9);
            \draw[-,thick] (3,1.5) to (3,0.9);
            \draw[-,thick] (4,1.5) to (4,0.9);
            \draw[<-,thick] (1.5,1) to (1.5,0);
            \draw[<-,thick] (2.5,1) to (2.5,0);
            \draw[<-,thick] (3,1) to (3,0);
            \draw[<-,thick] (4,1) to (4,0);
        \end{tikzpicture}
        ~\mapsto~
        \begin{tikzpicture}[baseline = 40pt, scale=0.75, color=\clr]
            \draw (2.75,2) node{?};
            \draw[-,thick] (1.35,2.5) to (4.15,2.5) to (4.15,1.5) to (1.35,1.5) to (1.35, 2.5);
            \draw[->,thick] (1.5,2.5) to[out=up,in=right] (1.2,3);
            \draw[-,thick] (1.25,3) to[out=left,in=up] (1,2.5) to (1,0);
            \draw[->,thick] (2.5,2.5) to[out=up,in=right] (1.2,3.8);
            \draw[-,thick] (1.25,3.8) to[out=left,in=up] (0,2.5) to (0,0);
            \draw[->,thick] (3,2.5) to (3,3.1);
            \draw[->,thick] (4,2.5) to (4,3.1);
            \draw[-,thick] (3,3) to (3,4);
            \draw[-,thick] (4,3) to (4,4);
            \draw[-,thick] (1.5,1.5) to (1.5,0.9);
            \draw[-,thick] (2.5,1.5) to (2.5,0.9);
            \draw[<-,thick] (1.5,1) to (1.5,0);
            \draw[<-,thick] (2.5,1) to (2.5,0);
            \draw[->,thick] (4.5,4) to (4.5,1.5) to[out=down,in=right] (4.2,1);
            \draw[-,thick] (4.25,1) to[out=left,in=down] (4,1.5);
            \draw[->,thick] (5.5,4) to (5.5,1.5) to[out=down,in=right] (4.2,0.2);
            \draw[-,thick] (4.25,0.2) to[out=left,in=down] (3,1.5);
        \end{tikzpicture}
        ~.
    \] 
    Since the monoidal superfunctor $\Phi$ respects the symmetric braidings and duality, the diagram given in \cref{reduce Hom to End} commutes. Thus, surjectivity of \cref{OBC to q Homs} follows from the surjectivity of the right vertical map in \cref{reduce Hom to End}.  However, composing the right vertical map in \cref{reduce Hom to End} with the isomorphism $\varphi$ from \cref{C:OBCandSergeev} gives the superalgebra map $\operatorname{Ser}_{r} \to \End_{U(\fq)}\left(V^{\otimes r} \right)$ from Schur-Weyl-Sergeev duality. When $\k$ has characteristic zero this is known to be surjective by \cite{Ser} (see also \cite[Section 3.4]{ChengWang}).
\end{proof}

\begin{remark}\label{R:JungKangResult}  When $\k$ has characteristic zero, Schur-Weyl-Sergeev duality also implies the right vertical map in \cref{reduce Hom to End} is injective whenever $r\leq n$. It follows that \cref{OBC to q Homs} is an isomorphism whenever the average length of the words $\ob a$ and $\ob b$ is less than or equal to $n$. In particular, $\Phi$ prescribes an isomorphism of superalgebras $\End_{\OBC}(\ob{a}) \cong \End_{U(\fq )}\left(V^{\ob{a}} \right)$ whenever the length of $\ob{a}$ is less than or equal to $n$.  Coupled with \cref{C:EndOBCisomorphictoWalledBrauerCliffordAlgebra} this recovers \cite[Theorem 3.5]{JK}.
\end{remark}

\begin{remark}\label{R:OBCfullness}  If $\k$ has positive characteristic, then one can replace $U(\fq )$ with the superalgebra of distributions for the supergroup $Q(n)$ and again have the superfunctor $\Phi$.  Moreover, the above argument for the fullness of $\Phi$ goes through modulo the statement that the map  $\operatorname{Ser}_{r} \to \End_{U(\fq)}\left(V^{\otimes r} \right)$ from Schur-Weyl-Sergeev duality is surjective.  For given $r$ it can be deduced from \cite{BK2} that this map is surjective whenever $n \geq r$ or the characteristic of $\k$ is greater than $r$.  As far as the authors are aware, surjectivity is not known in general. It is reasonable to expect it to hold (and, hence, the fullness of $\Phi$) under mild conditions on $\k$ (c.f.\  \cite{BD}).  Similar remarks apply to injectivity.
\end{remark}

\subsection{The monoidal superfunctor \texorpdfstring{$\Psi:\AOBC\to\END(\Uksmod)$}{Psi:AOBC to Uk-smod}}
Let $U(\gl )$ (resp.\ $U(\fq )$) denote the enveloping superalgebra of $\gl (V)$ (resp.\ $\fq (V)$).  Using the bases given in \cref{SS:Liesuperalgebras} we can naturally view $U(\fq)$ as Hopf subsuperalgebra of $U(\gl)$.   

Let $\Omega \in U( \gl) \otimes U(\fq)$ be the Casimir element given by
\begin{equation}\label{Casimir}
    \Omega = \sum_{1\leq i,j\leq n} \tilde{e}_{i,j}^{\0 }\otimes e_{j,i}^{\0 } + \tilde{e}_{i,j}^{\1 }\otimes e_{j,i}^{\1 }.
\end{equation}  Given a $U(\gl)$-supermodule $W$ and a $U(\fq)$-supermodule $M$, we have an even linear map $W \otimes M \to W \otimes M$ given by 
\[
\Omega.(w \otimes m) = \sum_{1\leq i,j\leq n} \tilde{e}_{i,j}^{\0 }.w\otimes e_{j,i}^{\0 }.m + (-1)^{\p{w}} \tilde{e}_{i,j}^{\1 }.w\otimes e_{j,i}^{\1 }.m,
\] for all homogeneous $w\in W$ and $m \in M$.  By restriction $W$ is a $U(\fq)$-supermodule and so $W \otimes M$ is a $U(\fq)$-supermodule via its coproduct.  The action of $\Omega$ defines an even $U(\fq )$-supermodule homomorphism by \cite[Theorem 7.4.1]{HKS} (there it is assumed that $\k = \mathbb{C}$ but the calculations do not depend on this fact).  Alternatively, one can use the odd invariant bilinear form given by the supertrace on $\fq$ to define the so-called odd Casimir element of $U(\fq ) \otimes U(\fq )$ which by standard arguments commutes with the image of the coproduct. In turn, since $\Omega$ equals the product of the odd Casimir with $\sqrt{-1}c\otimes 1$ it necessarily defines a supermodule homomorphism.  See \cite[proof of Lemma 3.1]{Brundan-Davidson} for details. Yet another way to see the action of $\Omega$ defines an even $U(\fq )$-supermodule homomorphism is to use the fact that $e_{i,j}^{\ep}$ and $\tilde{e}_{i,j}^{\ep}$ are dual with respect to a supertrace form, and follow \cite[proof of Lemma 4.1.4]{Serganova-Stroppel-EtAl}. 

\begin{theorem}
There is a monoidal superfunctor $\Psi:\AOBC\to\END(\Uksmod)$  by mapping the objects $\up, \down$ to the endofunctors $V\otimes-, V^*\otimes-$, respectively, and on morphisms by 
\begin{align*}
    \Psi\left(\lcup\right)&:\operatorname{Id} \rightarrow V \otimes V^*\otimes-, 
    \quad&
    m &\mapsto \sum_{i\in I} v_i\otimes v^*_i \otimes m,\\
    \Psi\left(\lcap\right)&:V^* \otimes V\otimes- \rightarrow \operatorname{Id},
    \quad&
    f\otimes v\otimes m  &\mapsto f(v) m,\\
    \Psi\left(\swap\right)&:V \otimes V \otimes- \rightarrow V \otimes V \otimes-,
    \quad&
    u \otimes v \otimes m &\mapsto (-1)^{\p{u}\p{v}} v\otimes u \otimes m,\\
    \Psi\left(\xdot\right)&:V \otimes - \rightarrow V \otimes -,
    \quad&
    v \otimes m &\mapsto \Omega(v \otimes m),\\
    \Psi\left(\cldot\right)&:V \otimes - \rightarrow V \otimes -,
    \quad&
    v \otimes m &\mapsto c(v)\otimes m.
\end{align*} 
\end{theorem}

\begin{proof} To show the existence of the superfunctor requires that we verify the defining relations of $\AOBC$.  The first three supernatural transformations are given by maps which are the coevaluation, evaluation, and braiding, respectively, in the supercategory $\svec$.  From this it follows that \cref{OB relations 1 (symmetric group),OB relations 2 (zigzags and invertibility)} are satisfied.  A direct calculation verifies that \cref{OBC relations} is also satisfied.  The relations in \cref{AOBC relations} follow from the verification of \cite[(3.1.4) and (3.1.5)]{HKS}, keeping in mind the authors chose to assume that Clifford elements square to minus one and to tensor by $V$ on the right.  Our different choices impact the signs which appear in formulas but otherwise have no effect.  
\end{proof}

\subsection{Bubbles and central elements of \texorpdfstring{$U(\fq )$}{U(q)}} \label{SS:bubblesandcenter}

As is well known (e.g.\ \cite[Proposition 46]{CM}), the supernatural transformations from the identity superfunctor to itself identify as a superalgebra with $Z(\fq )$, the supercenter of $U(\fq )$.  In particular, using the notation for supernatural transformations set in \cref{SS: monoidal supercats}, $\Psi(\Delta_k)_{U(\fq)}:U(\fq)\to U(\fq)$ is a supermodule homomorphism and $z_k:=\Psi(\Delta_k)_{U(\fq)}(1)$ lies in $Z(\fq)$. In this section we compute these central elements.

To do so requires further notation.  For any $k \geq 1$ and $\bm{\varepsilon} =(\ep_{k}, \dotsc , \ep_{1}) \in \Z_{2}^{k}$, define $\p{\bm{\varepsilon}}= \ep_{k}+\dotsb +\ep_{1}$.  Furthermore, define $\sgn (\bm{\varepsilon})=\pm 1$ recursively by
\[
    \sgn (\ep_{k}, \dotsc , \ep_{1}) =
    \begin{cases} 
        (-1)^{(\ep_{k}+\bar{1})\p{(\ep_{k-1}, \dotsc , \ep_{1})}}\sgn (\ep_{k-1}, \dotsc , \ep_{1}), 
        & \text{ if $k>1$}; \\
        1, & \text{ if $k=1$}.
    \end{cases} 
\]

\begin{theorem}\label{T:centralelements} 
    Let $k$ be a positive odd integer, then the central element of $U(\fq )$ determined by the even supernatural transformation $\Psi (\Delta_{k})$ is
    \[
    z_{k}=2\sum_{\substack{(i_{k}, \dotsc , i_{1}) \in I_{0}^{k}, \\ \bm{\varepsilon}=(\ep_{k}, \dotsc, \ep_{1}) \in \Z_{2}^{k},\\ \text{ with $\p{\bm{\varepsilon}}=\0$}}} 
    \sgn (\bm{\ep})e_{i_{k-1}, i_{k}}^{\ep_{k}}e_{i_{k-2},i_{k-1}}^{\ep_{k-1}}\dotsb e_{i_{2},i_{3}}^{\ep_{3}}e_{i_{1},i_{2}}^{\ep_{2}}e_{i_{k},i_{1}}^{\ep_{1}}.
    \]
\end{theorem}

\begin{proof} 
    We compute $\Psi (\Delta_{k})_{U(\fq )}(1) = (\ev_{V} \otimes 1) \circ (1 \otimes \Omega^{k}) \circ (\gamma_{V, V^{*}} \otimes 1) \circ (\coev_{V} \otimes 1)(1)$.  First observe that 
    \begin{equation}\label{E:centercalc1}
         (\gamma_{V, V^{*}}\otimes 1) \circ (\coev_{V}\otimes 1)(1)=
         \sum_{i \in I} (-1)^{\p{v_{i}}} v_{i}^{*} \otimes v_{i} \otimes 1 = 
         \left(1-c_{2,1} \right)\left(\sum_{i_0 \in I_{0}} v_{i_0}^{*} \otimes v_{i_0} \otimes 1 \right),
    \end{equation}
    where $c_{2,1}:=\Psi \left(\cldotdown\cldot \right)_{U(\fq )}\in\End_{U(\fq)}(V^*\otimes V\otimes U(\fq))$.

    To continue it is helpful introduce some notation to simplify formulas. Given $\bd{i}=(i_{k}, \dotsc , i_{0}) \in I_{0}^{k+1}$ and $\bm{\varepsilon} = (\ep_{k}, \dotsc , \ep_{1}) \in \Z_{2}^{k}$, for short let $e_{\bd{i}}^{\bm{\ep}} \in U(\fq )$ be given by 
    \[
        e_{\bd{i}}^{\bm{\varepsilon}}= e_{i_{k-1}, i_{k}}^{\ep_{k}}e_{i_{k-2}, i_{k-1}}^{\ep_{k-1}}\dotsb e_{i_{1}, i_{2}}^{\ep_{2}} e_{i_{0}, i_{1}}^{\ep_{1}}.
    \]  
    Given $\bm{\varepsilon} \in \Z_{2}^{k}$ and $i \in I_{0}$ we write 
    \[
        v_{i}^{\p{\bm{\varepsilon}}} = 
        \begin{cases} 
            v_{i}, & \text{ if $\p{\bm{\varepsilon}}=\0$};\\
            v_{\bar{i}}, & \text{ if $\p{\bm{\varepsilon}}=\1$}.
        \end{cases}
    \]

    A straightforward induction on $k \geq 1$ proves that for each fixed $i_0\in I_{0}$
    \[
    (1 \otimes \Omega)^{k}. v_{i_0}^{*}\otimes v_{i_0} \otimes 1 = \sum \sgn (\bm{\varepsilon})v_{i_0}^{*} \otimes v_{i_{k}}^{\p{\bm{\varepsilon}}} \otimes e_{\bd{i}}^{\bm{\varepsilon}},
    \] 
    where the sum is over all $\bm{\varepsilon} = (\ep_{k}, \dotsc , \ep_{1}) \in \Z_{2}^{k}$ and $\bd{i}=(i_{k}, \dotsc , i_{1}, i_{0}) \in I_{0}^{k+1}$.  
    Combining this formula with \cref{E:centercalc1} and using $(1 \otimes \Omega)\circ c_{2,1}=-c_{2,1}\circ (1\otimes \Omega)$ and $(\ev_{V} \otimes 1)\circ c_{2,1}=\ev_{V} \otimes 1$ yields
    \begin{align*}
        \Psi(\Delta_{k})_{U(\fq )}(1) &=(\ev_{V} \otimes 1) \circ (1-(-1)^kc_{2,1}) \left( \sum\sgn(\bm{\varepsilon}) v_{i_0}^{*} \otimes v_{i_{k}}^{\p{\bm{\varepsilon}}} \otimes e_{\bd{i}}^{\bm{\varepsilon}} \right) \\
        &=\begin{cases} 
            0, &\text{ if $k$ is even}; \\
            2\sum \sgn (\bm{\varepsilon}) e_{\bd{i}}^{\bm{\varepsilon}}, & \text{ if $k$ is odd}.
        \end{cases}
    \end{align*} The upper sum is over all $\bd{i}=(i_{k}, \dotsc , i_{1}, i_{0}) \in I_{0}^{k+1}$, and $\bm{\varepsilon} = (\ep_{k}, \dotsc , \ep_{1}) \in \Z_{2}^{k}$. The lower sum is over all $\bd{i}\in I_{0}^{k+1}$ with $i_0=i_k$ and all $\bm{\varepsilon} \in \Z_{2}^{k}$ with $\p{\bm{\varepsilon}}=\0$.  This proves the stated result.
\end{proof}

For each integer $t \geq 0$, Sergeev defined an explicit element $S_{t} \in Z(\fq )$.  As we do not need Sergeev's elements we do not reproduce their definition here. The interested reader can find it in  \cite{SergeevCenter} or \cite[Section 8]{BK1}.   For completeness' sake we explain how Sergeev's elements relate to those described above.   Recall that $\sigma: U(\fq ) \to U(\fq )$ is the antipode of $U(\fq )$  (see \cref{SS:Liesuperalgebras}). 

\begin{proposition}\label{P:SergeevCenter}  
    For each integer $t \geq 1$, 
    \[
        z_{2t-1} =-2\sigma \left( S_{t}\right).
    \]
\end{proposition}

\begin{proof}  
    By expanding the recursive formula for $S_{t}$ we see that it is a sum with coefficients of $\pm 1$ over precisely the same set of monomials as given by the formula for $z_{2t-1}$ in \cref{T:centralelements}, except that they are in reverse order.  That is, $\sigma (S_{t})$  and $z_{2t-1}$ are sums over precisely the same set of monomials.  All that remains is to verify that the sign in front of each monomial agrees.  This is a straightforward check, keeping in mind that since $\sigma$ is a superalgebra anti-involution, if $x_{1}, \dotsc , x_{k} \in \fq$, then $\sigma(x_{1}\dotsb x_{k}) = (-1)^{\delta}x_{k}\dotsb x_{1}$, where $\delta = k + \sum_{1 \leq r < s \leq  k}\p{x_{r}}\p{x_{s}}$.  It is also helpful in comparing signs to verify that a closed formula for $\sgn (\ep_{k}, \dotsc , \ep_{1})$ is given by 
    \[
    \sgn (\ep_{k}, \dotsc , \ep_{1}) = (-1)^{\ep_{k-1}+\ep_{k-3}+\dotsb +\ep_{p} + \sum_{1 \leq r < s \leq k} \ep_{r}\ep_{s}},
    \] where $p =1$ if $k$ is even and $p=2$ if $k$ is odd.
\end{proof}

\begin{remark}\label{R:center}  
    Assume $\k$ has characteristic zero. By our basis theorem for $\AOBC$ the set $\left\{\Delta_1,\Delta_3,\Delta_5,\ldots \right\}$ is algebraically independent (see also \cref{R:algebraicallyindependent}). Moreover, Sergeev's elements are known to generate $Z(\fq )$ for every $n\geq 1$ by \cite{SergeevCenter}.  Therefore $\Psi$ defines a surjective homomorphism $\End_{\AOBC}(\unit ) \to \End_{\EndUksmod}(\Id )$.
\end{remark}

\section{The generic Verma supermodule} \label{S:GenericVerma}
We now introduce the generic Verma supermodule.  This supermodule will play a key role in proving the basis theorem for $\AOBC$.

\subsection{The polynomial ring}\label{SS:polynomialring} Let $U(\fh_{\0})$ denote the subsuperalgebra of $U(\fh)$ generated by  $\left\{ h_{i}  \mid i=1, \dotsc, n \right\}$, where for brevity we set $h_{i}=e_{i,i}^{\0}$ for $i=1,\dotsc ,n$.   We put the usual $\Z$-grading on $U(\fh_{\0})$  by putting each $h_{i}$ in degree $1$.  Since $\fh_{\0}$ is an abelian Lie superalgebra concentrated in parity zero, $U(\fh_{\0})$ is nothing more than a polynomial ring on generators $h_{n}, \dotsc , h_{1}$.  We refine the $\Z$-grading on $U(\fh_{\0})$ by also putting a graded lexicographic order on the monomials with the convention that $h_n>\cdots>h_1$; that is, $h_{n}^{r_{n}}\dotsb h_{1}^{r_{1}} < h_{n}^{s_{n}}\dotsb h_{1}^{s_{1}}$ if and only if either $\sum_{i} r_{i} < \sum_{i} s_{i}$ or $\sum_{i} r_{i} = \sum_{i} s_{i}$ and if $t=1, \dotsc , n$ is maximal with $r_{t}\neq s_{t}$, then $r_{t}<s_{t}$.

\subsection{The generic Verma supermodule} Let $U(\fb )$ denote the subsuperalgebra of $U(\fq)$ generated by $U(\fh)$ and $U(\fn )$.  By the PBW theorem we have $U(\fb ) = U(\fh) \otimes U(\fn )$ as superspaces.   
Since the span of the monomials of positive degree, $U(\fn )_{+}$, is an ideal of $U(\fb )$ with $U(\fb )/U(\fn )_{+} \cong U(\fh)$, we can and will view $U(\fh )$ as a $U(\fb )$-supermodule by inflation.
Define the \emph{generic Verma supermodule} to be the $U(\fq)$-supermodule 
\[
M=U(\fq )\otimes_{U(\fb)} U(\fh )  \cong U(\fq ) \otimes_{U(\fb )} \left(U(\fh)\otimes_{U(\fh_{\0})} U(\fh_{\0}) \right).
\] We will write $\hw:=1\otimes 1\otimes 1\in M$ for the ``highest weight vector''. While we do not need this fact, note that if $W$ is a $U(\fq )$-supermodule and $w \in W$ is a homogeneous weight vector for $U(\fh_{\0})$ such that $e_{i,j}^{\ep}w =0$ for all $1 \leq i < j \leq n$ and all $\ep \in \Z_{2}$, then there is a unique $U(\fq )$-supermodule homomorphism $M \to W$ such that $\hw \mapsto w$.
  
For brevity, set $N=n(n-1)/2$ and let $\left\{f_{1}, \dotsc , f_{N} \right\}$ (resp.~$\left\{\bar{f}_{1}, \dotsc , \bar{f}_{N} \right\}$) be any basis for $\fn^{-}_{\0}$ (resp.~$\fn^{-}_{\1}$) such that 
$\left\{e_{i,j}^{\0} \mid 1 \leq j < i \leq n\right\} = \left\{f_{1}, \dotsc , f_{N} \right\}$ and  $\left\{e_{i,j}^{\1} \mid 1 \leq j < i \leq n\right\} = \left\{\bar{f}_{1}, \dotsc , \bar{f}_{N} \right\}$.  
In particular, a PBW basis for $U(\fn^{-})$ is given by the monomials $f_{1}^{a_{1}}\dotsb f_{N}^{a_{N}}\bar{f}_{1}^{b_{1}}\dotsb \bar{f}_{N}^{b_{N}}$, where $a_{1}, \dotsc , a_{N}\in\Z_{\geq 0}$ and $b_{1}, \dotsc , b_{N} \in \{0,1 \}$.
By the PBW theorem $M$ is a free right $U(\fh_{\0})$-supermodule on basis 
\begin{equation}\label{M basis}
f_1^{a_1}\cdots f_N^{a_N}\bar{f}_1^{b_1}\cdots \bar{f}_N^{b_N}\otimes \bar{h}_1^{c_1}\cdots \bar{h}_n^{c_n}\otimes 1
\end{equation}
 where each $a_k\in\Z_{\geq0}$ and $b_k,c_k\in\{0,1\}$. In particular, the superspace $M$ inherits a $\Z$-grading (hence a $\Z$-filtration) from the right action of $U(\fh_{\0})$ where for $t\in \Z$, $M_{t}$ is the span of $f_1^{a_1}\cdots f_N^{a_N}\bar{f}_1^{b_1}\cdots \bar{f}_N^{b_N}\otimes \bar{h}_1^{c_1}\cdots \bar{h}_n^{c_n}\otimes U (\fh_{\0})_{t}$ for all $a_k\in\Z_{\geq0}$ and $b_k,c_k\in\{0,1\}$. More generally, for any $\ob a\in\wrd$ the superspace $V^{\ob a}\otimes M$ is a free right $U(\fh_{\0}) $-supermodule and is similarly a graded (hence filtered) superspace. 
 % Given a $\Z$-graded superspace $W = \oplus_{i \in \Z} W_{i}$, set the notation $W_{\leq t}=\oplus_{i\leq t} W_{i}$ for the $t$-th filtered piece.  -- This notation is setup in section 2.4

\begin{lemma}\label{eM filtered degree 1}
    Suppose $1\leq i,j\leq n$ and $\ep\in \Z_{2}$.  Left multiplication by $e_{i,j}^\ep$ prescribes a map $M\to M$ that is 
    \begin{enumerate}
        \item[(1)] homogeneous degree 0 when $i>j$ and, 
        \item[(2)] filtered degree 1 when $i\leq j$.
    \end{enumerate} 
\end{lemma}

\begin{proof}
    Since $M$ is a $(U(\fq ), U(\fh_{\0}))$-bisupermodule with the grading coming from the right action by $U(\fh_{\\0})$, it suffices to show that the result of acting on \cref{M basis} on the left by $e_{i,j}^\ep$ is (1) in $M_0$ when $i>j$; and (2) in $M_{\leq 1}$ when $i\leq j$.

Part (1) follows from the fact that $e_{i,j}^\ep$ is an element of the subsuperalgebra $U(\fn^{-})$ whenever $i>j$. That is, since monomials of the form $f_1^{a_1}\cdots f_N^{a_N}\bar{f}_1^{b_1}\cdots \bar{f}_N^{b_N}$ are a basis for $U(\fn^{-})$ by the PBW theorem, after acting by $e_{i,j}^\ep$ in $U(\fn^{-})$ one can rewrite the result as a sum of elements which lie in $M_{0}$. 

For part (2) one can argue by induction on 
    \begin{equation}\label{induction sum}
        a_1+b_1+\cdots+a_N+b_N.
    \end{equation} Namely, act by $e_{i,j}^{\ep}$ on \cref{M basis} and then use the commutator formulas $U(\fq)$ given in \cite[Sections 2 and 3]{DuWan} to rewrite the expression into a sum of terms with smaller \cref{induction sum}.  The base case when this sum equals zero also follows by a calculation using the commutator formulas of \cite{DuWan}.

\end{proof}  

\begin{lemma}\label{Psi-M is filtered}
    For any $\ob a\in\wrd$, the map $V\otimes(V^{\ob a}\otimes M)\to V\otimes(V^{\ob a}\otimes M)$ given by the action of $\Omega$ is filtered degree 1. 
\end{lemma}

\begin{proof}
    This follows from the definition of the Casimir \cref{Casimir} and \cref{eM filtered degree 1}. 
\end{proof}

Using the previous lemma we can show that the composition of $\Psi$ followed by evaluation at $M$ defines a filtered monoidal superfunctor
\begin{equation*}
    \Psi_M:\AOBC\to\fsvec.
\end{equation*}
Indeed, it is straightforward to see that $\Psi_M(d)$ is filtered degree $0$ whenever $d$ is an oriented Brauer-Clifford diagram. The fact that $\Psi_M(d)$ is a filtered map with filtered degree $\deg(d)$ whenever $d$ is a dotted oriented Brauer-Clifford diagram with bubbles follows from \cref{Psi-M is filtered}. 

\subsection{Some \texorpdfstring{$\Psi_M$}{PsiM} calculations}\label{SS:PsiM calculations}
In this subsection we prove several lemmas concerning the superfunctor $\Psi_M$, which will be used in \cref{section: proof of AOBC basis theorem} to prove our basis theorem for $\AOBC$. 

\begin{lemma}\label{x1 calculation}
    Whenever $1\leq i\leq n$, the degree 1 components of $\Psi_M(\xdot)(v_i\otimes\hw)$ and $\Psi_M(\xdot)(v_\ibar\otimes\hw)$ are $v_i\otimes\hw h_i$ and $-v_\ibar\otimes\hw h_i$ respectively. 
\end{lemma}

\begin{proof}
    Note that $e_{i,j}^\ep\hw$ has degree 0 unless $\ep=0$ and $i=j$, in which case we have $e_{i,j}^\ep\hw=h_i\hw=\hw h_i$. Thus, using \cref{Casimir} we see the degree 1 components of $\Psi_M(\xdot)(v_i\otimes\hw)$ and $\Psi_M(\xdot)(v_\ibar\otimes\hw)$ are as given by the lemma.
\end{proof}

\begin{lemma}\label{bubbles under Psi-M}
    The degree $k$ component of $\Psi_M(\Delta_k)(\hw)$ is $2\hw(h_n^k+\cdots+h_1^k)$ whenever $k$ is odd.
\end{lemma}

\begin{proof}
    In the following computation of $\Psi_M(\Delta_k)(\hw)$ we list only the top degree terms:
    \begin{align*}
        \hw & \mapsto \sum_{1\leq i\leq n} v_i\otimes v_i^*\otimes\hw + v_\ibar\otimes v_\ibar^*\otimes\hw 
                & \text{(apply $\Psi_M\left(\lcup\right)$)}\\
            & \mapsto \sum_{1\leq i\leq n} v_i^*\otimes v_i\otimes\hw - v_\ibar^*\otimes v_\ibar\otimes\hw 
                & \text{(apply $\Psi_M\left(\rswap\right)$)}\\
            & \mapsto \sum_{1\leq i\leq n} v_i^*\otimes v_i\otimes\hw h_i^k - (-1)^kv_\ibar^*\otimes v_\ibar\otimes\hw h_i^k +\cdots
                & \text{(apply $\Psi_M\left(\downxdotk\!\!\right)$ and}\\
            &   & \text{ use  \cref{x1 calculation})}\\
            & \mapsto \sum_{1\leq i\leq n} \hw h_i^k - (-1)^k\hw h_i^k +\cdots
                & \text{(apply $\Psi_M\left(\lcap\right)$)}.
    \end{align*}
    Now, assume $k$ is odd to get $\Psi_M(\Delta_k)(\hw)=2\hw(h_n^k+\cdots+h_1^k)+\dotsb $.
\end{proof}

\begin{remark} \label{R:algebraicallyindependent}
    Assume $\k$ has characteristic zero.  By taking $n$ sufficently large, the previous lemma along with the fact the first $n$ power sums in $h_{n}, \dotsc , h_{1}$ are algebraically independent can be used to show $\Delta_1,\Delta_3,\Delta_5,\ldots$ are algebraically independent.
\end{remark}

The next two lemmas concern $\Psi_M(d)$ for certain diagrams of the form $d:\up^r\to\up^r$. It will be convenient to let $x_k:\up^r\to\up^r$ denote the diagram obtained from the identity diagram by placing a single $\fulldot$ on the $k$th strand from the right. We will also let $(i,j):\up^r\to\up^r$ denote the crossing of the $i$th and $j$th strands whenever $1\leq i,j\leq r$.  For example, if $r=8$ then 
\[
    x_3=~
    \begin{tikzpicture}[baseline = 7.5pt, scale=0.75, color=\clr]
        \draw[->,thick] (0,0) to (0,1);
        \draw[->,thick] (0.5,0) to (0.5,1);
        \draw[->,thick] (1,0) to (1,1);
        \draw[->,thick] (1.5,0) to (1.5,1);
        \draw[->,thick] (2,0) to (2,1);
        \draw[->,thick] (2.5,0) to (2.5,1);
        \draw[->,thick] (3,0) to (3,1);
        \draw[->,thick] (3.5,0) to (3.5,1);
        \draw (2.5,0.4) \bdot;
    \end{tikzpicture}
    \quad
    \text{and}
    \quad
    (4,1)=~
    \begin{tikzpicture}[baseline = 7.5pt, scale=0.75, color=\clr]
        \draw[->,thick] (0,0) to (0,1);
        \draw[->,thick] (0.5,0) to (0.5,1);
        \draw[->,thick] (1,0) to (1,1);
        \draw[->,thick] (1.5,0) to (1.5,1);
        \draw[->,thick] (2,0) to[out=up,in=down] (3.5,1);
        \draw[->,thick] (2.5,0) to (2.5,1);
        \draw[->,thick] (3,0) to (3,1);
        \draw[->,thick] (3.5,0) to[out=up,in=down] (2,1);
    \end{tikzpicture}
    ~.
\]

\begin{lemma}\label{xk calculation}
    The degree 1 component of $\Psi_M(x_k)(v_r\otimes\cdots v_1\otimes\hw)$ is $v_r\otimes\cdots v_1\otimes\hw h_k$.
\end{lemma}

\begin{proof}
    First, $x_k-(k,1)x_1(k,1)$ is a linear combination of oriented Brauer-Clifford diagrams whenever $k>1$. Hence, since $\Psi_M$ is filtered and oriented Brauer-Clifford diagrams are degree 0, the degree 1 components of $\Psi_M(x_k)(v_r\otimes\cdots\otimes v_1\otimes\hw)$ and $\Psi_M((k,1)x_1(k,1))(v_r\otimes\cdots\otimes v_1\otimes\hw)$ are equal. Thus, it suffices to prove the case $k=1$, which follows from \cref{x1 calculation}.
\end{proof}

Let us fix the following notation for the remainder of the paper. Given a dotted oriented Brauer-Clifford diagram $d$, let $\undot{d}$ denote the oriented Brauer-Clifford diagram obtained from $d$ by removing all $\fulldot$'s. For example, if 
\begin{equation}\label{undot example}
    d=~
    \begin{tikzpicture}[baseline = 10pt, scale=0.5, color=\clr]
        \draw[->,thick] (0,0) to[out=up, in=down] (2,2);
        \draw[->,thick] (1,0) to[out=up, in=down] (4,2);
        \draw[->,thick] (2,0) to[out=up, in=down] (0,2);
        \draw[->,thick] (3,0) to[out=up, in=down] (1,2);
        \draw[->,thick] (4,0) to[out=up, in=down] (5,2);
        \draw[->,thick] (5,0) to[out=up, in=down] (3,2);
        \draw (0.15,1.5) \wdot;
        \draw (1.85,1.5) \wdot;
        \draw (4.9,1.5) \wdot;
        \draw (0.15,0.45) \bdot;
        \draw (2.85,0.45) \bdot;
        \draw (2.85,1.05) \bdot;
        \draw (3.85,1.05) \bdot;
        \draw (4.85,0.45) \bdot;
    \end{tikzpicture},
    \quad\text{then}~ 
    \undot{d}=~
    \begin{tikzpicture}[baseline = 10pt, scale=0.5, color=\clr]
        \draw[->,thick] (0,0) to[out=up, in=down] (2,2);
        \draw[->,thick] (1,0) to[out=up, in=down] (4,2);
        \draw[->,thick] (2,0) to[out=up, in=down] (0,2);
        \draw[->,thick] (3,0) to[out=up, in=down] (1,2);
        \draw[->,thick] (4,0) to[out=up, in=down] (5,2);
        \draw[->,thick] (5,0) to[out=up, in=down] (3,2);
        \draw (0.15,1.5) \wdot;
        \draw (1.85,1.5) \wdot;
        \draw (4.9,1.5) \wdot;
    \end{tikzpicture}
    ~.
\end{equation}
Now, suppose $d: \up^r\to\up^r$ is a normally ordered dotted oriented Brauer-Clifford diagram (as in \cref{D:N.O. dotted OBC diagram}) without bubbles. We let $\beta_k(d)$ denote the number of $\fulldot$'s on the $k$th strand of $d$, where we count strands right-to-left according to their position on the bottom boundary of $d$. In particular, we have 
\begin{equation*}
    d=\undot{d}\circ x_r^{\beta_r(d)}\circ\cdots\circ x_1^{\beta_1(d)}.
\end{equation*}
Finally, given $g\in\End_{\OBC}(\up^r)$ we let $v(g)\in V^{\otimes r}$ denote the image of $v_r\otimes\cdots\otimes v_1$ under $\Phi(g)$. For example, when $d$ as in \cref{undot example} we have 
\[v(\undot{d})=-\sqrt{-1} v_{\bar{4}}\otimes v_{3} \otimes v_{\bar{6}} \otimes v_{1} \otimes v_{5} \otimes v_{\bar{2}}.\]
With this notation in mind, the following result follows immediately from \cref{xk calculation}.

\begin{lemma}\label{top component lemma}
    For any normally ordered dotted oriented Brauer-Clifford diagram $d:\up^r\to\up^r$ without bubbles, the top degree component of $\Psi_M(d)(v_r\otimes\cdots v_1\otimes\hw)$ is $v(\undot{d})\otimes\hw h_r^{\beta_r(d)}\cdots h_1^{\beta_1(d)}$. 
\end{lemma}

\section{Proof of the main result}\label{section: proof of AOBC basis theorem}

We can now prove the key special case of the main result.  Namely, the normally ordered dotted oriented Brauer-Clifford diagrams with bubbles provide a basis for $\End_{\AOBC}(\up^{r})$.

\begin{theorem}\label{Basis theorem for affine Hecke-Clifford with bubbles}  Assume $\k$ has characteristic zero. Then the set of equivalence classes of normally ordered dotted oriented Brauer-Clifford diagrams with bubbles of type $\up^r\to \up^r$ form a basis for $\End_\AOBC(\up^r)$. 
\end{theorem}

\begin{proof}
     By  \cref{lemma: normally ordered diagrams span} the proposed basis spans $\End_\AOBC(\up^r)$. Toward showing linear independence, note that any linear combination of normally ordered dotted oriented Brauer-Clifford diagrams with bubbles of type $\up^r\to \up^r$ can be written in the form 
    \begin{equation}\label{linear combo}
        \sum_{d}f_{d}(\Delta_1,\Delta_3,\Delta_5,\ldots)d,
    \end{equation}
    where the sum is over all normally ordered dotted oriented Brauer-Clifford diagrams $d:\up^r\to \up^r$ (\emph{without bubbles}); and where the $f_{d}$'s are polynomials in countably many variables, only finitely many of which are nonzero. Set $B=\{d \mid f_{d}\neq 0\}$.
 
    We will show \cref{linear combo} is nonzero whenever $B \neq \emptyset$ (completing the proof of the theorem) by showing its image under $\Psi_M$ is nonzero whenever $n$ is sufficiently large. In turn, this will follow from the fact that, when we choose $n$ sufficiently large to ensure the relevant power sums are algebraically independent,  
    \begin{equation}\label{p(completes)}
        f_{d}(h_n+\cdots+h_1, h_n^3+\cdots+h_1^3,h_n^5+\cdots+h_1^5,\ldots)\neq 0
    \end{equation} for any $d\in B$. Recall from \cref{SS:polynomialring} that we have a graded lexicographic ordering on the monomials of $U (\fh )$. For the rest of the proof we assume $n$ is large enough so that for each $d\in B$ the leading monomial of the symmetric polynomial \cref{p(completes)} with respect to this ordering does not contain any of $h_r,\ldots,h_1$.  
    % \begin{equation}\label{n is large enough} 
    %     n\geq\max\{\deg f_d(t,t^3,t^5,\ldots) + r~|~d\in B\}.   
    % \end{equation}

    Given $d\in B$, it follows from  \cref{bubbles under Psi-M} that the top degree component of $\Psi_M(f_d(\Delta_1,\Delta_3,\ldots))(v_r\otimes\cdots\otimes v_1\otimes\hw)$ is of the form $v_r\otimes\cdots\otimes v_1\otimes\hw g_d(h_n,\ldots,h_1)$ where $g_d(h_n,\ldots,h_1)$ is some homogeneous symmetric polynomial. Fix $d_0\in B$ with 
    \begin{equation}\label{d0 condition}
        \deg d_0 + \deg g_{d_0}\geq \deg d+\deg g_{d}\text{ for all }d\in B.
    \end{equation} 
    Set 
    \[
    B_0=\{d\in B \mid \undot{d}=\undot{d_0}\text{ and }\deg d_0 + \deg g_{d_0}=\deg d+\deg g_{d}\}.
    \]
    It follows from \cref{d0 condition} and  \cref{top component lemma} that the top degree component of the image of $v_r\otimes\cdots\otimes v_1\otimes\hw$ under $\Psi_M(\sum_{d}f_d(\Delta_1,\Delta_3,\Delta_5,\ldots)d)$ is of the form 
    \begin{equation}\label{top degree component under Psi_M}
        \sum_{d\in B_0} v(\undot{d_0})\otimes\hw g_d(h_n,\ldots,h_1) h_r^{\beta_r(d)}\cdots h_1^{\beta_1(d)}+w,
    \end{equation}
    where $w$ lies in the $U(\fh)$-span of the basis elements of the form $v(d)\otimes u$ with $v(d)\neq v(\undot{d_0})$. Recall that, since $n$ was chosen sufficiently large, the elements $h_r,\ldots,h_1$ do not appear in leading monomials of each $g_d(h_n,\ldots,h_1)$. Therefore, since a diagram $d\in B_0$ is completely determined by $\beta_1(d),\ldots,\beta_r(d)$, it follows that the leading monomials of $g_d(h_n,\ldots,h_1) h_r^{\beta_r(d)}\cdots h_1^{\beta_1(d)}$ for $d\in B_0$ are pairwise distinct. Thus \cref{top degree component under Psi_M} is nonzero, which implies \cref{linear combo} is nonzero, as desired. 
\end{proof}

\subsection{Proof of Theorem \ref{AOBC basis theorem}} 
It is straightforward to see that the validity of \cref{AOBC basis theorem} when $\k$ has characteristic zero is equivalent to the following lemma.

\begin{lemma}\label{Basis for finite AOBC}     
    Assume $\k$ has characteristic zero.  Then for any $\ob a,\ob b\in\wrd$ the superspace $\Hom_\AOBC(\ob a,\ob b)_{\leq k}$ has basis given by equivalence classes of normally ordered dotted oriented Brauer-Clifford diagrams with bubbles of type $\ob a\to\ob b$ with at most $k$ $\fulldot$'s. 
\end{lemma}

\begin{proof} 
    It follows from  \cref{lemma: normally ordered diagrams span} that the proposed basis spans $\Hom_\AOBC(\ob a, \ob b)_{\leq k}$. In particular, $\Hom_\AOBC(\ob a, \ob b)_{\leq k}$ is finite-dimensional over $\k$. Hence, it suffices to show the proposed basis has size $\dim_\k\Hom_\AOBC(\ob a, \ob b)_{\leq k}$. 
    Now, suppose $\ob a$ (resp.~$\ob b$) consists of $r_1$ (resp.~$r'_1$) $\up$'s and $r_2$ (resp.~$r'_2$) $\down$'s.  If $r_1+r_2'\not=r_1'+r_2$ then there are no oriented Brauer-Clifford diagrams of type $\ob a\to\ob b$, whence $\Hom_\AOBC(\ob a, \ob b)=0$. Thus, we may assume $r_1+r_2'=r_1'+r_2=:r$. 
    In this case we have $\k$-linear isomorphisms
    \begin{equation*}
        \Hom_\AOBC(\ob a, \ob b)_{\leq k} \to \Hom_\AOBC(\down^{r_2}\up^{r_1}, \up^{r'_1}\down^{r'_2})_{\leq k}\to\End_\AOBC(\up^r)_{\leq k}
    \end{equation*}
    defined on diagrams in the same manner as the top horizontal maps in \cref{reduce Hom to End}. 
    In particular, $\dim_\k\Hom_\AOBC(\ob a, \ob b)_{\leq k}=\dim_\k\End_\AOBC(\up^r)_{\leq k}$. On the other hand, there are precisely $r$ strands in any dotted oriented Brauer-Clifford diagram of type $\ob a\to\ob b$. It follows that there are the same number of normally ordered dotted oriented Brauer-Clifford diagrams with bubbles with at most $k$ $\fulldot$'s of type $\ob a\to\ob b$ as there are of type $\up^r\to \up^r$. Thus, the result follows from  \cref{Basis theorem for affine Hecke-Clifford with bubbles}. 
\end{proof}

\subsection{The positive characteristic case}\label{SS:PositiveCharacteristic}

We now explain how to deduce \cref{AOBC basis theorem} when $\k$ has positive characteristic and, more generally, is an arbitrary graded commutative $\Z_{2}$-graded ring of characteristic not two.  We first observe that the definitions given in \cref{SS: monoidal supercats} work equally well if $\k$ is replaced with an arbitrary graded commutative $\Z_{2}$-graded ring, $R$, and $\k$-superspaces are replaced with $\Z_{2}$-graded $R$-modules.  We refer to these as (monoidal) $R$-supercategories.  We define $\AOBC_{R}$ to be the monoidal $R$-supercategory given by the same generators and relations used in \cref{AOBC defn}.  For example, we have the integral form $\AOBC_{\Z}$ of the degenerate affine oriented Brauer-Clifford supercategory.

With the above in mind we have the following integral version of \cref{AOBC basis theorem}. 

\begin{theorem}\label{T:integralAOBC}  
    For any $\ob a,\ob b\in\wrd$ the $\Z$-supermodule $\Hom_{\AOBC_{\Z}}(\ob a,\ob b)$ is a free $\Z$-supermodule with basis given by equivalence classes of normally ordered dotted oriented Brauer-Clifford diagrams with bubbles of type $\ob a\to\ob b$.
\end{theorem}

\begin{proof} 
    Since the relations for $\AOBC$ involve only integral coefficients, the arguments given in \cref{N.O. diagrams span} apply to $\AOBC_{\Z}$ and so the normally ordered dotted oriented Brauer-Clifford diagrams with bubbles of type $\ob a\to\ob b$ span $\Hom_{\AOBC_{\Z}}(\ob{a}, \ob{b})$ as a $\Z$-supermodule.  On the other hand, consider a finite sum 
     \begin{equation}\label{integral linear combo}
            \sum_{d} f_{d}d,
        \end{equation}
     where the sum is over normally ordered dotted oriented Brauer-Clifford diagrams with bubbles of type $\ob{a} \to \ob{b}$ and where the $f_{d}$'s are integers. There is an obvious superfunctor $\AOBC_{\Z} \to \AOBC_{\C}$ which can be applied to \cref{integral linear combo} and linear independence follows from \cref{AOBC basis theorem}.
\end{proof}

Let $\k$ be a graded commutative $\Z_{2}$-graded ring.  If $\mathcal{C}_{\Z}$ denotes a (monoidal) $\Z$-supercategory, then by base change one can define a (monoidal) $\k$-supercategory $\mathcal{C}_{\Z} \otimes \k$.  Namely, the objects of $\mathcal{C}_{\Z}\otimes \k$ are the objects of $\mathcal{C}_{\Z}$ and the morphisms are
\[
\Hom_{\mathcal{C}_{\Z}\otimes \k}(\ob{a},\ob{b}) = \Hom_{\mathcal{C}_{\Z}}(\ob{a},\ob{b}) \otimes_{\Z} \k.
\]  Composition, the monoidal structure, etc., are extended to $\mathcal{C}_{\Z}\otimes \k$ by linearity.   There are obvious mutually inverse superfunctors which provide an isomorphism of monoidal supercategories between $\AOBC_{\Z} \otimes \k$ and $\AOBC_{\k}$.  The previous theorem and base change immediately implies the following result.

\begin{corollary}\label{C:postivecharacteristicAOBCbasis} 
    Let $\k$ be an arbitrary graded commutative $\Z_{2}$-graded ring of characteristic not two. For any $\ob a,\ob b\in\wrd$ the $\k$-supermodule $\Hom_{\AOBC_{\k}}(\ob a,\ob b)$ is a free $\k$-supermodule with basis given by equivalence classes of normally ordered dotted oriented Brauer-Clifford diagrams with bubbles of type $\ob a\to\ob b$.
\end{corollary}

We mention one other application of our basis theorem.
\begin{corollary}\label{C:degenerateaffineHeckeCliffordsuperalgebraIsom}  
    Let $\k$ be a field of characteristic not two.  The subsuperalgebra of $\End_{\AOBC}(\up^r)$ consisting of linear combinations of dotted oriented Brauer-Clifford diagrams without bubbles is isomorphic to the \emph{degenerate affine Sergeev superalgebra}\footnote{Also known as the degenerate affine Hecke-Clifford superalgebra.} $\AS_{r}$ introduced in \cite[Section 3]{Naz}.  
\end{corollary}

\begin{proof} 
    The superalgebra $\AS_{r}$ has a presentation with even generators $s_{1}, \dotsc , s_{r-1}$, $x_{1}, \dotsc , x_{r}$, and odd generators $c_{1}, \dotsc , c_{r}$ subject to the relations (for all admissible $i,j$):
    \begin{enumerate}
    \item $s_{i}^{2}=1$, $s_{i}s_{j}=s_{j}s_{i}$ when $|i-j| >1$, $s_{i}s_{i+1}s_{i}=s_{i+1}s_{i}s_{i+1}$;
    \item $c_{i}^{2}=1$, $c_{i}c_{j}=-c_{j}c_{i}$ when $i\neq j$;
    \item $x_{i}x_{j}=x_{j}x_{i}$;
    \item $c_{i}x_{i}=-x_{i}c_{i}$, $c_{i}x_{j}=x_{j}c_{i}$ when $i \neq j$;
    \item $s_{i}x_{i}=x_{i+1}s_{i}-1-c_{i}c_{i+1}$.
    \end{enumerate}

    By checking relations we see that the following defines a superalgebra homomorphism 
    $\nu:\AS_{r} \to \End_{\AOBC}\left(\up^{r} \right)$.
    \begin{align*}
        \nu\left(s_{i} \right) &= 
                \begin{tikzpicture}[baseline = 7.5pt, scale=0.5, color=\clr]
                    \draw[->,thick] (0,0) to[out=up, in=down] (0,1.5);
                \end{tikzpicture}^{r-i-1}
                  \begin{tikzpicture}[baseline = 7.5pt, scale=0.5, color=\clr]
                    \draw[->,thick] (0,0) to[out=up, in=down] (1,1.5);
                    \draw[->,thick] (1,0) to[out=up, in=down] (0,1.5);
                \end{tikzpicture} \;\;
                \begin{tikzpicture}[baseline = 7.5pt, scale=0.5, color=\clr]
                    \draw[->,thick] (0,0) to[out=up, in=down] (0,1.5);
                \end{tikzpicture}^{i-1}, &
        \nu\left(c_{i} \right) &=  \begin{tikzpicture}[baseline = 7.5pt, scale=0.5, color=\clr]
                    \draw[->,thick] (0,0) to[out=up, in=down] (0,1.5);
                \end{tikzpicture}^{r-i}
                  \begin{tikzpicture}[baseline = 7.5pt, scale=0.5, color=\clr]
                    \draw[->,thick] (0,0) to[out=up, in=down] (0,1.5);
                    \draw (0,0.7) \wdot;
                \end{tikzpicture} \;\;
                \begin{tikzpicture}[baseline = 7.5pt, scale=0.5, color=\clr]
                    \draw[->,thick] (0,0) to[out=up, in=down] (0,1.5);
                \end{tikzpicture}^{i-1}, &
        \nu\left(x_{i} \right) &=  \begin{tikzpicture}[baseline = 7.5pt, scale=0.5, color=\clr]
                    \draw[->,thick] (0,0) to[out=up, in=down] (0,1.5);
                \end{tikzpicture}^{r-i}
                  \begin{tikzpicture}[baseline = 7.5pt, scale=0.5, color=\clr]
                    \draw[->,thick] (0,0) to[out=up, in=down] (0,1.5);
                    \draw (0,0.7) \bdot;
                \end{tikzpicture} \;\;
                \begin{tikzpicture}[baseline = 7.5pt, scale=0.5, color=\clr]
                    \draw[->,thick] (0,0) to[out=up, in=down] (0,1.5);
                \end{tikzpicture}^{i-1}.
    \end{align*} Take note that this map follows our convention of numbering strands from right-to-left.  The image of this map is the subsuperalgebra of $\End_{\AOBC}\left(\up^{r} \right)$ spanned by the dotted oriented Brauer-Clifford diagrams without bubbles.  From \cref{C:postivecharacteristicAOBCbasis} and the PBW-type basis for $\AS_{r}$ given in \cite[Theorem 14.2.2]{KleshBook} one can verify that this map is an isomorphism onto its image.
\end{proof}

\section{Cyclotomic Quotients} \label{S:CyclotomicQuotients}

Fix $a,b\in\Z_{\geq 0}$ and $m_i\in\k$ for each $1\leq i\leq a$. Let $f(t)=t^b\prod_{1\leq i\leq a}(t^2-m_i)$,  $\ell=2a+b$, and $ \OBC^f$ be as in \cref{SS:Cyclotomic quotients}.

\subsection{Bases for cyclotomic quotients}\label{SS:cyclotomic bases theorems}
Since $\OBC^f$ is a quotient of $\AOBC$, we may interpret any dotted oriented Brauer-Clifford diagram with bubbles as a morphism in $\OBC^f$. 

\begin{theorem}\label{Cyclotomic basis conjecture}
    For any $\ob a,\ob b\in\wrd$ the superspace $\Hom_{\OBC^f}(\ob a,\ob b)$ has basis given by equivalence classes of normally ordered dotted oriented Brauer-Clifford diagrams with bubbles of type $\ob a\to\ob b$ with fewer than $\ell$ $\fulldot$'s on each strand.
\end{theorem}

A proof of this theorem can be found in \cref{S:cyclobasis}. However, 
it is easy to show the proposed basis in  \cref{Cyclotomic basis conjecture} spans the appropriate Hom-space. Indeed, any dotted oriented Brauer-Clifford diagram with bubbles having $\ell$ $\fulldot$'s on one of its strands can be realized as a linear combination of diagrams with fewer total $\fulldot$'s by using  \cref{prop: AOBC slides} to slide those $\ell$ $\fulldot$'s to the right of the picture, and then reducing as prescribed by $f$. For example, the following holds in $\OBC^f$ when $f(t)=t$:
\begin{equation*}
    \begin{tikzpicture}[baseline = 12pt, scale=0.5, color=\clr]
        \draw[<-,thick] (0,0) to[out=up,in=down] (1,2);
        \draw[->,thick] (1,0) to[out=up,in=down] (0,2);
        \draw[->,thick] (2,0) to[out=up,in=down] (2,2);
        \draw (0.85,0.5) \bdot;
        \draw (0.1,0.5) \wdot;
    \end{tikzpicture}
    ~=~
    \begin{tikzpicture}[baseline = 12pt, scale=0.5, color=\clr]
        \draw[<-,thick] (0,0) to[out=up,in=down] (1,2);
        \draw[->,thick] (1,0) to[out=up,in=down] (2,1) to[out=up,in=down] (0,2);
        \draw[->,thick] (2,0) to[out=up,in=down] (1.5,1) to[out=up,in=down] (2,2);
        \draw (1.25,1.5) \bdot;
        \draw (0.1,0.5) \wdot;
    \end{tikzpicture}
    ~=~
    \begin{tikzpicture}[baseline = 12pt, scale=0.5, color=\clr]
        \draw[<-,thick] (0,0) to[out=up,in=down] (1,2);
        \draw[->,thick] (1,0) to[out=up,in=down] (2,1) to[out=up,in=down] (0,2);
        \draw[->,thick] (2,0) to[out=up,in=down] (1.5,1) to[out=up,in=down] (2,2);
        \draw (2,1) \bdot;
        \draw (0.1,0.5) \wdot;
    \end{tikzpicture}
    ~+~
    \begin{tikzpicture}[baseline = 12pt, scale=0.5, color=\clr]
        \draw[<-,thick] (0,0) to[out=up,in=down] (1,2);
        \draw[->,thick] (1,0) to[out=up,in=down] (2,1) to[out=up,in=down] (2,2);
        \draw[->,thick] (2,0) to[out=up,in=down] (1.5,1) to[out=up,in=down] (0,2);
        \draw (0.1,0.5) \wdot;
    \end{tikzpicture}
    ~-~
    \begin{tikzpicture}[baseline = 12pt, scale=0.5, color=\clr]
        \draw[<-,thick] (0,0) to[out=up,in=down] (1,2);
        \draw[->,thick] (1,0) to[out=up,in=down] (2,1) to[out=up,in=down] (2,2);
        \draw[->,thick] (2,0) to[out=up,in=down] (1.5,1) to[out=up,in=down] (0,2);
        \draw (0.1,0.5) \wdot;
        \draw (1.5,1) \wdot;
        \draw (2,1) \wdot;
    \end{tikzpicture}
    ~=~
    \begin{tikzpicture}[baseline = 12pt, scale=0.5, color=\clr]
        \draw[<-,thick] (0,0) to[out=up,in=down] (1,2);
        \draw[->,thick] (1,0) to[out=up,in=down] (2,2);
        \draw[->,thick] (2,0) to[out=up,in=down] (0,2);
        \draw (0.1,0.5) \wdot;
    \end{tikzpicture}
    ~-~
    \begin{tikzpicture}[baseline = 12pt, scale=0.5, color=\clr]
        \draw[<-,thick] (0,0) to[out=up,in=down] (1,2);
        \draw[->,thick] (1,0) to[out=up,in=down] (2,2);
        \draw[->,thick] (2,0) to[out=up,in=down] (0,2);
        \draw (0.1,0.5) \wdot;
        \draw (0.15,1.5) \wdot;
        \draw (1.9,1.5) \wdot;
    \end{tikzpicture}
    .
\end{equation*}

\subsection{Connection to the superalgebras of Gao-Rui-Song-Su}\label{SS:connection to GRSS}

In this subsection we explain how to recover the affine and cyclotomic walled Brauer-Clifford superalgebras from \cite{GRSS} from our supercategories. The discussion here parallels the analogous one in \cite[Section 5.5]{BCNR}. 

First,  $\AOBC$ can be viewed as a $\k[\Delta_1,\Delta_3,\ldots]$-linear supercategory with the action of each $\Delta_k$ given by tensoring on the \emph{right}: $h\Delta_k:=h\otimes\Delta_k$. Given $\delta_1,\delta_2,\ldots\in\k$ we let $\AOBC(\delta_1,\delta_2,\ldots)$ denote the supercategory obtained by specializing each $\Delta_{2k-1}$ at $\delta_k$. In other words, $\AOBC(\delta_1,\delta_2,\ldots):=\AOBC\otimes_{\k[\Delta_1,\Delta_3,\ldots]}\k$ viewing $\k$ as a $\k[\Delta_1,\Delta_3,\ldots]$-module with $\Delta_{2k-1}$ acting as $\delta_k$. It follows from \cref{AOBC basis theorem} that the superspace $\Hom_{\AOBC(\delta_1,\delta_2,\ldots)}(\ob a,\ob b)$ has basis consisting of all equivalence classes of normally ordered dotted oriented Brauer-Clifford diagrams (without bubbles) of type $\ob a\to\ob b$. 
We have a similar specialization for level $\ell$ cyclotomic quotients.  
Namely, let us write $f(t) = \sum_{i=0}^{\ell} a_{i}t^{i}$. Now, fix $\delta_{1}, \dotsc , \delta_{\lfloor \ell/2\rfloor} \in \k$, and define $\delta_{k}$ recursively for $k > \lfloor \ell/2\rfloor$ by 
\begin{equation}\label{formula for higher deltas in cyclotomic quotients}
    \delta_{k} = - \sum_{1\leq j\leq\lfloor \ell/2\rfloor} a_{\ell-2j}\delta_{k-j}. 
\end{equation}
Then the specialization $\OBC^f(\delta_1, \dotsc, \delta_{\lfloor \ell/2\rfloor}):=\OBC^f\otimes_{\k[\Delta_1, \Delta_{3}, \dotsc]}\k$. is well defined. By \cref{Cyclotomic basis conjecture} the superspace $\Hom_{\OBC^f(\delta_1,\dotsc, \delta_{\lfloor \ell/2\rfloor})}(\ob a,\ob b)$ has basis consisting of all equivalence classes of normally ordered dotted oriented Brauer-Clifford diagrams (without bubbles) of type $\ob a\to\ob b$ with fewer than $\ell$ $\fulldot$'s on each strand. 
For the remainder of this section we will write $ABC_{s,r}(\delta_1,\delta_2,\ldots)$ and $BC_{s,r}^f(\delta_1, \dotsc, \delta_{\lfloor \ell/2\rfloor})$ for the endomorphism algebras of the object $\down^s\up^r$ in the supercategories $\AOBC(\delta_1,\delta_2,\ldots)$ and $\OBC^f(\delta_1,\dotsc , \delta_{\lfloor \ell/2\rfloor})$, respectively. 

Let $BC^{\text{aff}}_{r,s}$ denote the affine walled Brauer-Clifford superalgebra defined in \cite[Definition 3.1]{GRSS}. This superalgebra is defined via odd generators $c_i$ $(1\leq i\leq r)$,$\bar{c}_j$ $(1\leq j\leq s)$; even generators $e_1$, $x_1$, $\bar{x}_1$, $s_i$ $(1\leq i<r)$, $\bar{s}_j$ $(1\leq j<s)$; and even central generators $\omega_{2k+1}$, $\bar{\omega}_k$ $(k\in\Z_{>0})$ subject to a long list of relations. It is an exercise in checking those relations to see that there is a well-defined superalgebra map $BC^{\text{aff}}_{r,s}\to ABC_{s,r}(\delta_1,\delta_2,\ldots)$ defined by 
\begin{align*}
    c_i & \mapsto\sqrt{-1}~
        \begin{tikzpicture}[baseline = 7.5pt, scale=0.5, color=\clr]
            \draw[<-,thick] (0,0) to[out=up, in=down] (0,1.5);
        \end{tikzpicture}^{s}
        \begin{tikzpicture}[baseline = 7.5pt, scale=0.5, color=\clr]
            \draw[->,thick] (0,0) to[out=up, in=down] (0,1.5);
        \end{tikzpicture}^{r-i}
        \begin{tikzpicture}[baseline = 7.5pt, scale=0.5, color=\clr]
            \draw[->,thick] (0,0) to[out=up, in=down] (0,1.5);
            \draw (0,0.7) \wdot;
        \end{tikzpicture} \;\;
        \begin{tikzpicture}[baseline = 7.5pt, scale=0.5, color=\clr]
            \draw[->,thick] (0,0) to[out=up, in=down] (0,1.5);
        \end{tikzpicture}^{i-1} &
    \bar{c}_j &\mapsto\sqrt{-1}~ 
        \begin{tikzpicture}[baseline = 7.5pt, scale=0.5, color=\clr]
            \draw[<-,thick] (0,0) to[out=up, in=down] (0,1.5);
        \end{tikzpicture}^{s-j}
        \begin{tikzpicture}[baseline = 7.5pt, scale=0.5, color=\clr]
            \draw[<-,thick] (0,0) to[out=up, in=down] (0,1.5);
            \draw (0,0.7) \wdot;
        \end{tikzpicture} \;\;
        \begin{tikzpicture}[baseline = 7.5pt, scale=0.5, color=\clr]
            \draw[<-,thick] (0,0) to[out=up, in=down] (0,1.5);
        \end{tikzpicture}^{j-1}
        \begin{tikzpicture}[baseline = 7.5pt, scale=0.5, color=\clr]
            \draw[->,thick] (0,0) to[out=up, in=down] (0,1.5);
        \end{tikzpicture}^{r} \\
    x_1 & \mapsto -
        \begin{tikzpicture}[baseline = 7.5pt, scale=0.5, color=\clr]
            \draw[<-,thick] (0,0) to[out=up, in=down] (0,1.5);
        \end{tikzpicture}^{s}
        \begin{tikzpicture}[baseline = 7.5pt, scale=0.5, color=\clr]
            \draw[->,thick] (0,0) to[out=up, in=down] (0,1.5);
        \end{tikzpicture}^{r-1}
        \begin{tikzpicture}[baseline = 7.5pt, scale=0.5, color=\clr]
            \draw[->,thick] (0,0) to[out=up, in=down] (0,1.5);
            \draw (0,0.7) \bdot;
        \end{tikzpicture}
        & 
    \bar{x}_1 &\mapsto    
        \begin{tikzpicture}[baseline = 11pt, scale=0.5, color=\clr]
            \draw[<-,thick] (0,0) to[out=up, in=down] (0,2);
        \end{tikzpicture}^{s-1}
        \begin{tikzpicture}[baseline = 11pt, scale=0.5, color=\clr]
            \draw[->,thick] (0,2) to[out=down, in=up] (1.5,1) to[out=down,in=up] (0,0);
            \draw[->,thick] (0.75,0) to (0.75,2);
            \draw[color=black] (1,1.9) node{$~^{r}$};
            \draw (1.5,1) \bdot;
        \end{tikzpicture}
        \\
    s_i & \mapsto   
        \begin{tikzpicture}[baseline = 7.5pt, scale=0.5, color=\clr]
            \draw[<-,thick] (0,0) to[out=up, in=down] (0,1.5);
        \end{tikzpicture}^{s}
        \begin{tikzpicture}[baseline = 7.5pt, scale=0.5, color=\clr]
            \draw[->,thick] (0,0) to[out=up, in=down] (0,1.5);
        \end{tikzpicture}^{r-i-1}
        \begin{tikzpicture}[baseline = 7.5pt, scale=0.5, color=\clr]
            \draw[->,thick] (0,0) to[out=up, in=down] (1,1.5);
            \draw[->,thick] (1,0) to[out=up, in=down] (0,1.5);
        \end{tikzpicture} \;\;
        \begin{tikzpicture}[baseline = 7.5pt, scale=0.5, color=\clr]
            \draw[->,thick] (0,0) to[out=up, in=down] (0,1.5);
        \end{tikzpicture}^{i-1} &
    \bar{s}_j &\mapsto  
        \begin{tikzpicture}[baseline = 7.5pt, scale=0.5, color=\clr]
            \draw[<-,thick] (0,0) to[out=up, in=down] (0,1.5);
        \end{tikzpicture}^{s-j-1}
        \begin{tikzpicture}[baseline = 7.5pt, scale=0.5, color=\clr]
            \draw[<-,thick] (0,0) to[out=up, in=down] (1,1.5);
            \draw[<-,thick] (1,0) to[out=up, in=down] (0,1.5);
        \end{tikzpicture} \;\;
        \begin{tikzpicture}[baseline = 7.5pt, scale=0.5, color=\clr]
            \draw[<-,thick] (0,0) to[out=up, in=down] (0,1.5);
        \end{tikzpicture}^{j-1}
        \begin{tikzpicture}[baseline = 7.5pt, scale=0.5, color=\clr]
            \draw[->,thick] (0,0) to[out=up, in=down] (0,1.5);
        \end{tikzpicture}^{r} \\
    e_1 & \mapsto 
        \begin{tikzpicture}[baseline = 11pt, scale=0.5, color=\clr]
            \draw[<-,thick] (0,0) to[out=up, in=down] (0,2);
        \end{tikzpicture}^{s-1}
        \begin{tikzpicture}[baseline = 11pt, scale=0.5, color=\clr]
            \draw[<-,thick] (0,0) to[out=up,in=left] (2,0.75) to[out=right,in=up] (4,0);
            \draw[->,thick] (0,2) to[out=down,in=left] (2,1.25) to[out=right,in=down] (4,2);
            \draw[->,thick] (2,0) to[out=up, in=down] (2,2);
            \draw[color=black] (2.75,1.9) node{$~^{r-1}$};
        \end{tikzpicture}&
    \bar{\omega}_{2k} & \mapsto 0\\
    \omega_{2k+1} & \mapsto -\delta_{k+1} &
    \bar{\omega}_{2k-1} & \mapsto \delta'_k\\
\end{align*}
where $\delta'_k$ is defined recursively by $\delta_k-\delta'_k=\sum_{0<i<\lfloor k/2\rfloor}\delta_{i}\delta'_{k-i}$ (compare with \cref{R:clockwise bubbles}). This map factors through the quotient $\widetilde{BC}_{r,s}$ of $BC^{\text{aff}}_{r,s}$ by the additional relations $\omega_{2k-1}=-\delta_{k}$, $\bar{\omega}_{2k}=0$, $\bar{\omega}_{2k-1}=\delta'_{k}$ for all $k\in\Z_{>0}$ which is precisely the specialized superalgebra in \cite[Theorem 5.15]{GRSS}. Using our basis theorem one can easily check that the spanning set for $\widetilde{BC}_{r,s}$ described in \cite[Definition 3.15 and Corollary 3.16]{GRSS} maps to a basis for $ABC_{s,r}(\delta_1,\delta_2,\ldots)$. Hence, $\widetilde{BC}_{r,s}\cong ABC_{s,r}(\delta_1,\delta_2,\ldots)$. Note that this also gives a different proof of the linear independence in \cite[Theorem 5.15]{GRSS}.

A similar discussion applies to the the cyclotomic quotients. Again, we fix $\delta_{1}, \dotsc , \delta_{\lfloor \ell/2\rfloor} \in \k$, and define $\delta_{k}$ for $k > \lfloor \ell/2\rfloor$ by \cref{formula for higher deltas in cyclotomic quotients}. 
In \cite[Definition 3.14]{GRSS} the cyclotomic walled Brauer-Clifford superalgebra $BC_{\ell,r,s}$ is defined as the quotient of $\widetilde{BC}_{r,s}$ by the additional relations $f(x_1)=g(\bar{x}_1)=0$ where $g(t)$ is another monic degree $\ell$ polynomial satisfying certain conditions. One can check that those conditions imply $f(x_1)$ and $g(\bar{x}_1)$ are mapped to zero under the composition of the isomorphism $BC^{\text{aff}}_{r,s}\to ABC_{s,r}(\delta_1,\delta_2,\ldots)$ from the previous paragraph with the quotient map $ABC_{s,r}(\delta_1,\delta_2,\ldots)\to BC^f_{s,r}(\delta_1, \dotsc, \delta_{\lfloor \ell/2\rfloor} )$. Hence, that composition factors through $BC_{\ell,r,s}$ to induce a surjection $BC_{\ell,r,s}\to BC^f_{s,r}(\delta_1, \dotsc , \delta_{\lfloor \ell/2\rfloor})$. Now, using our basis theorem one can check that the spanning set for $BC_{\ell,r,s}$ from \cite[Definition 3.15 and Corollary 3.16]{GRSS} maps to a basis for $BC^f_{s,r}(\delta_1, \dotsc \delta_{\lfloor \ell/2\rfloor})$. Hence, $BC_{\ell,r,s}\cong BC^f_{s,r}(\delta_1,\dotsc , \delta_{\lfloor \ell/2\rfloor})$.

% \begin{remark}
%     The discussion above shows that our basis theorems (\cref{AOBC basis theorem} and \cref{Cyclotomic basis conjecture}) imply the linear independence of the corresponding basis theorems of Gao-Rui-Song-Su (\cite[Theorem 5.15 and Theorem 6.10]{GRSS}). Indeed, the linear independence of so-called regular monomials in the superalgebras of Gao-Rui-Song-Su follows from the observation that these monomials map to bases of the appropriate endomorphism superalgebras under the superalgebra maps defined above. These arguments cannot merely be reversed in order to obtain our basis theorems from those of Gao-Rui-Song-Su. The main obstacle comes from the fact that our endomorphism superalgebras do not come equipped with nice descriptions via generators and relations. Indeed, without already having a basis theorem for the supercategory, it is difficult to extract a full system of generators and relations for the endomorphism superalgebras from the defining generators and relations for the monoidal supercategory. Hence, providing a well-defined superalgebra map from our endomorphism superalgebras to the superalgebras of Gao-Rui-Song-Su is not an easy task. In particular, we are unable to conclude \cref{T:integralAOBC} from \cite[Theorem 5.14]{GRSS} nor \cref{Cyclotomic basis conjecture} from \cite[Theorem 6.10]{GRSS}.  However, see \cite{GRSS2}.
% \end{remark}

%\appendix
\section{The Cyclotomic Basis Theorem}\label{S:cyclobasis}
%\begin{center}
%(By Jonathan Brundan)
%\end{center}

In this section we provide a proof of \cref{Cyclotomic basis conjecture}.
Let $\K$ be some commutative $\k$-algebra. We will consider also the
base-changed monoidal supercategory
$\AOBC_\K := \AOBC\otimes_\k \K$.

Fix $\ell \geq 0$ and monic polynomials
\begin{align}\label{f1}
f(u) &= z_0 u^\ell + z_1 u^{\ell-2} + z_2 u^{\ell-4}+\cdots,\\
f'(u) &= z_0' u^\ell + z'_1 u^{\ell-2} + z'_2 u^{\ell-4}+\cdots\label{f2}
\end{align}
in $\K[u]$.
So $z_0 = z_0' = 1$, and all powers of $u$ in these polynomials 
are even or odd according to the parity of $\ell$.
Define the power series 
\begin{align}\label{ps1}
\delta(u) &= \delta_0 + \delta_1 u^{-1} + \delta_2 u^{-2}+\cdots,\\
\delta'(u) &= \delta'_0 + \delta'_1 u^{-1} + \delta'_2 u^{-2}+\cdots\label{ps2}
\end{align}
in $\K[[u^{-1}]]$ from
\begin{align}\label{ps3}
\delta(u^2) &:= f'(u) / f(u),\\\label{ps4}
\delta'(u^2)&:=-f(u)/f'(u).
\end{align}
Note $\delta_0 = 1$ but $\delta_0' = -1$.
Computing the coefficients of $u^{\ell-2r}$ in $f'(u) = f(u) \delta(u^2)$ gives
\begin{align}\label{e1}
\sum_{s=0}^r z_s \delta_{r-s} &= z_r'
\text{ for $r=0,\dots,\lfloor \ell/2 \rfloor$},\\
\sum_{s=0}^{\lfloor \ell/2\rfloor} z_s \delta_{r-s} &= 0
\text{ for $r > \lfloor \ell / 2\rfloor$}.\label{e2}
\end{align}

Let $\Sym$ be the algebra of symmetric functions over $\K$, viewed as
a purely even superalgebra.
Denote the elementary and complete symmetric functions by $e_r$ and
$h_r$ as usual; in particular, $e_0=h_0=1$. Working in 
$\Sym[[u^{-1}]]$, we set
\begin{align}
e(u) &:= e_0 + u^{-1} e_1 + u^{-2} e_2+\cdots,\\
h(u) &:= h_0 + u^{-1} h_1 + u^{-2} h_2+\cdots,
\end{align}
and recall the fundamental identity
$e(u) h(-u) = 1$.  By \cref{AOBC basis theorem}  and \cref{R:clockwise bubbles}, there is a well-defined superalgebra
isomorphism
\begin{equation}\label{beta}
\beta:\Sym \stackrel{\sim}{\rightarrow} \End_{\AOBC_\K} (\unit),\quad
h_r \mapsto
(-1)^r \mathord{
\begin{tikzpicture}[baseline = 1.25mm]
  \draw[->,thick,darkblue] (0.2,0.2) to[out=90,in=0] (0,.4);
  \draw[-,thick,darkblue] (0,0.4) to[out=180,in=90] (-.2,0.2);
\draw[-,thick,darkblue] (-.2,0.2) to[out=-90,in=180] (0,0);
  \draw[-,thick,darkblue] (0,0) to[out=0,in=-90] (0.2,0.2);
   \node at (0.2,0.2) {$\color{darkblue}\bullet$};
   \node at (0.65,0.2) {$\color{darkblue}\scriptstyle{2r-1}$};
\end{tikzpicture}
},
\quad
e_r \mapsto
-\mathord{
\begin{tikzpicture}[baseline = 1.25mm]
  \draw[<-,thick,darkblue] (0,0.4) to[out=180,in=90] (-.2,0.2);
  \draw[-,thick,darkblue] (0.2,0.2) to[out=90,in=0] (0,.4);
 \draw[-,thick,darkblue] (-.2,0.2) to[out=-90,in=180] (0,0);
  \draw[-,thick,darkblue] (0,0) to[out=0,in=-90] (0.2,0.2);
   \node at (-0.2,0.2) {$\color{darkblue}\bullet$};
   \node at (-0.65,0.2) {$\color{darkblue}\scriptstyle{2r-1}$};
\end{tikzpicture}}\:.
\end{equation}
So that this also makes sense in the case $r=0$, it is natural to
adopt the convention that
$\mathord{
\begin{tikzpicture}[baseline = 1.25mm]
  \draw[->,thick,darkblue] (0.2,0.2) to[out=90,in=0] (0,.4);
  \draw[-,thick,darkblue] (0,0.4) to[out=180,in=90] (-.2,0.2);
\draw[-,thick,darkblue] (-.2,0.2) to[out=-90,in=180] (0,0);
  \draw[-,thick,darkblue] (0,0) to[out=0,in=-90] (0.2,0.2);
   \node at (0.2,0.2) {$\color{darkblue}\bullet$};
   \node at (0.5,0.2) {$\color{darkblue}\scriptstyle{-1}$};
\end{tikzpicture}
} := 1_\unit$ and $\mathord{
\begin{tikzpicture}[baseline = 1.25mm]
  \draw[<-,thick,darkblue] (0,0.4) to[out=180,in=90] (-.2,0.2);
  \draw[-,thick,darkblue] (0.2,0.2) to[out=90,in=0] (0,.4);
 \draw[-,thick,darkblue] (-.2,0.2) to[out=-90,in=180] (0,0);
  \draw[-,thick,darkblue] (0,0) to[out=0,in=-90] (0.2,0.2);
   \node at (-0.2,0.2) {$\color{darkblue}\bullet$};
   \node at (-0.5,0.2) {$\color{darkblue}\scriptstyle{-1}$};
\end{tikzpicture}} := -1_\unit$.

\begin{lemma}\label{l1}
The $\K$-linear left tensor ideal $\mathcal I_{f,f'}$ of $\AOBC_\K$ generated by
\begin{equation}\label{set1}
\left\{f\left(\mathord{
\begin{tikzpicture}[baseline = -0.5mm]
	\draw[->,thick,darkblue] (0.08,-.2) to (0.08,.3);
      \node at (0.08,0.05) {$\color{darkblue}\bullet$};
\end{tikzpicture}
}\right),
\mathord{\begin{tikzpicture}[baseline = -1mm]
  \draw[-,thick,darkblue] (0,0.2) to[out=180,in=90] (-.2,0);
  \draw[->,thick,darkblue] (0.2,0) to[out=90,in=0] (0,.2);
 \draw[-,thick,darkblue] (-.2,0) to[out=-90,in=180] (0,-0.2);
  \draw[-,thick,darkblue] (0,-0.2) to[out=0,in=-90] (0.2,0);
   \node at (0.65,0) {$\color{darkblue}\scriptstyle{2r-1}$};
      \node at (0.2,0) {$\color{darkblue}\bullet$};
\end{tikzpicture}
}-\delta_r 1_\unit\:\Big|\: r = 1,\dots,\lfloor \ell/2\rfloor
\right\}
\end{equation}
is generated equivalently by
\begin{equation}\label{set2}
\left\{f'\left(\mathord{
\begin{tikzpicture}[baseline = -.5mm]
	\draw[<-,thick,darkblue] (0.08,-.2) to (0.08,.3);
      \node at (0.08,0.1) {$\color{darkblue}\bullet$};
\end{tikzpicture}
}\right),
\mathord{\begin{tikzpicture}[baseline = -1mm]
  \draw[<-,thick,darkblue] (0,0.2) to[out=180,in=90] (-.2,0);
  \draw[-,thick,darkblue] (0.2,0) to[out=90,in=0] (0,.2);
 \draw[-,thick,darkblue] (-.2,0) to[out=-90,in=180] (0,-0.2);
  \draw[-,thick,darkblue] (0,-0.2) to[out=0,in=-90] (0.2,0);
   \node at (-0.65,0) {$\color{darkblue}\scriptstyle{2r-1}$};
      \node at (-0.2,0) {$\color{darkblue}\bullet$};
\end{tikzpicture}
}
-\delta_r' 1_\unit\:\Big|\:r=1,\dots,\lfloor \ell/2 \rfloor\right\}.
\end{equation}
Moreover, $\mathcal I_{f,f'}$ contains 
$
\mathord{\begin{tikzpicture}[baseline = -1mm]
  \draw[-,thick,darkblue] (0,0.2) to[out=180,in=90] (-.2,0);
  \draw[->,thick,darkblue] (0.2,0) to[out=90,in=0] (0,.2);
 \draw[-,thick,darkblue] (-.2,0) to[out=-90,in=180] (0,-0.2);
  \draw[-,thick,darkblue] (0,-0.2) to[out=0,in=-90] (0.2,0);
   \node at (0.65,0) {$\color{darkblue}\scriptstyle{2r-1}$};
      \node at (0.2,0) {$\color{darkblue}\bullet$};
\end{tikzpicture}
}-\delta_r 1_\unit $ and $\mathord{\begin{tikzpicture}[baseline = -1mm]
  \draw[<-,thick,darkblue] (0,0.2) to[out=180,in=90] (-.2,0);
  \draw[-,thick,darkblue] (0.2,0) to[out=90,in=0] (0,.2);
 \draw[-,thick,darkblue] (-.2,0) to[out=-90,in=180] (0,-0.2);
  \draw[-,thick,darkblue] (0,-0.2) to[out=0,in=-90] (0.2,0);
   \node at (-0.65,0) {$\color{darkblue}\scriptstyle{2r-1}$};
      \node at (-0.2,0) {$\color{darkblue}\bullet$};
\end{tikzpicture}
}
-\delta_r' 1_\unit$
for all $r \geq 0$.
\end{lemma}

\begin{proof}
This is similar to Lemma 1.8 of \cite{B}.
We first show by induction on $r$ that
$\mathcal I_{f,f'}$ contains
$\mathord{\begin{tikzpicture}[baseline = -1mm]
  \draw[-,thick,darkblue] (0,0.2) to[out=180,in=90] (-.2,0);
  \draw[->,thick,darkblue] (0.2,0) to[out=90,in=0] (0,.2);
 \draw[-,thick,darkblue] (-.2,0) to[out=-90,in=180] (0,-0.2);
  \draw[-,thick,darkblue] (0,-0.2) to[out=0,in=-90] (0.2,0);
   \node at (0.65,0) {$\color{darkblue}\scriptstyle{2r-1}$};
      \node at (0.2,0) {$\color{darkblue}\bullet$};
\end{tikzpicture}
}-\delta_r 1_\unit$ for all $r \geq 0$.
This is immediate from the definitions for $r \leq \lfloor
\ell/2\rfloor$, so assume that $r > \lfloor \ell/2 \rfloor$.
Since $2r-1 \geq \ell$, 
we get from $f\left(\mathord{
\begin{tikzpicture}[baseline = -0.5mm]
	\draw[->,thick,darkblue] (0.08,-.2) to (0.08,.3);
      \node at (0.08,0.05) {$\color{darkblue}\bullet$};
\end{tikzpicture}
}\right) \in \mathcal I_{f,f'}$ that
$
\sum_{s=0}^{\lfloor\ell/2\rfloor}
z_s \mathord{\begin{tikzpicture}[baseline = -1mm]
  \draw[-,thick,darkblue] (0,0.2) to[out=180,in=90] (-.2,0);
  \draw[->,thick,darkblue] (0.2,0) to[out=90,in=0] (0,.2);
 \draw[-,thick,darkblue] (-.2,0) to[out=-90,in=180] (0,-0.2);
  \draw[-,thick,darkblue] (0,-0.2) to[out=0,in=-90] (0.2,0);
   \node at (0.95,0) {$\color{darkblue}\scriptstyle{2(r-s)-1}$};
      \node at (0.2,0) {$\color{darkblue}\bullet$};
\end{tikzpicture}
}\in \mathcal I_{f,f'}
$ too.
Now the following verifies the induction step:
\begin{align*}
\mathord{\begin{tikzpicture}[baseline = -1mm]
  \draw[-,thick,darkblue] (0,0.2) to[out=180,in=90] (-.2,0);
  \draw[->,thick,darkblue] (0.2,0) to[out=90,in=0] (0,.2);
 \draw[-,thick,darkblue] (-.2,0) to[out=-90,in=180] (0,-0.2);
  \draw[-,thick,darkblue] (0,-0.2) to[out=0,in=-90] (0.2,0);
   \node at (0.65,0) {$\color{darkblue}\scriptstyle{2r-1}$};
      \node at (0.2,0) {$\color{darkblue}\bullet$};
\end{tikzpicture}
}
\!- \delta_r 1_\unit
\stackrel{(\ref{e2})}{=}\mathord{\begin{tikzpicture}[baseline = -1mm]
  \draw[-,thick,darkblue] (0,0.2) to[out=180,in=90] (-.2,0);
  \draw[->,thick,darkblue] (0.2,0) to[out=90,in=0] (0,.2);
 \draw[-,thick,darkblue] (-.2,0) to[out=-90,in=180] (0,-0.2);
  \draw[-,thick,darkblue] (0,-0.2) to[out=0,in=-90] (0.2,0);
   \node at (0.65,0) {$\color{darkblue}\scriptstyle{2r-1}$};
      \node at (0.2,0) {$\color{darkblue}\bullet$};
\end{tikzpicture}
}
\!+ \sum_{s=1}^{\lfloor\ell/2\rfloor}
z_s \delta_{r-s} 1_\unit
\equiv
\sum_{s=0}^{\lfloor\ell/2\rfloor}
z_s \mathord{\begin{tikzpicture}[baseline = -1mm]
  \draw[-,thick,darkblue] (0,0.2) to[out=180,in=90] (-.2,0);
  \draw[->,thick,darkblue] (0.2,0) to[out=90,in=0] (0,.2);
 \draw[-,thick,darkblue] (-.2,0) to[out=-90,in=180] (0,-0.2);
  \draw[-,thick,darkblue] (0,-0.2) to[out=0,in=-90] (0.2,0);
   \node at (0.95,0) {$\color{darkblue}\scriptstyle{2(r-s)-1}$};
      \node at (0.2,0) {$\color{darkblue}\bullet$};
\end{tikzpicture}
}
\equiv 0\pmod{\mathcal I_{f,f'}}.
\end{align*}
Hence, recalling (\ref{beta}), we have that 
$\beta(h(-u^2)) \equiv \delta(u^2) 1_\unit\pmod{\mathcal I_{f,f'}}$.
Since $e(u) = h(-u)^{-1}$ 
and $\delta'(u) = - \delta(u)^{-1}$, it follows that $\beta(e(u^2)) \equiv
-\delta'(u^2) 1_\unit\pmod{\mathcal I_{f,f'}}$.
This shows that 
$\mathord{\begin{tikzpicture}[baseline = -1mm]
  \draw[<-,thick,darkblue] (0,0.2) to[out=180,in=90] (-.2,0);
  \draw[-,thick,darkblue] (0.2,0) to[out=90,in=0] (0,.2);
 \draw[-,thick,darkblue] (-.2,0) to[out=-90,in=180] (0,-0.2);
  \draw[-,thick,darkblue] (0,-0.2) to[out=0,in=-90] (0.2,0);
   \node at (-0.65,0) {$\color{darkblue}\scriptstyle{2r-1}$};
      \node at (-0.2,0) {$\color{darkblue}\bullet$};
\end{tikzpicture}
}
-\delta_r' 1_\unit \in 
\mathcal I_{f,f'}$ for all $r \geq 0$.
Now we can show that $f'\left(\mathord{
\begin{tikzpicture}[baseline = -0.5mm]
	\draw[<-,thick,darkblue] (0.08,-.2) to (0.08,.3);
      \node at (0.08,0.1) {$\color{darkblue}\bullet$};
\end{tikzpicture}
}\right) \in \mathcal I_{f,f'}$:
\begin{align*}
f'\left(\mathord{
\begin{tikzpicture}[baseline = -0.5mm]
	\draw[<-,thick,darkblue] (0.08,-.2) to (0.08,.3);
      \node at (0.08,0.1) {$\color{darkblue}\bullet$};
\end{tikzpicture}
}\right) &= \sum_{r=0}^{\lfloor \ell/2\rfloor} z_r' 
\mathord{
\begin{tikzpicture}[baseline = 0]
	\draw[<-,thick,darkblue] (0.08,-.3) to (0.08,.4);
      \node at (0.08,0.07) {$\color{darkblue}\bullet$};
      \node at (0.5,.07) {$\color{darkblue}\scriptstyle{\ell-2r}$};
\end{tikzpicture}
}
\stackrel{(\ref{e1})}{=}
\sum_{r=0}^{\lfloor \ell/2 \rfloor} 
\sum_{s=0}^r
z_s
\delta_{r-s}
\mathord{
\begin{tikzpicture}[baseline = 0]
	\draw[<-,thick,darkblue] (0.08,-.3) to (0.08,.4);
      \node at (0.08,0.07) {$\color{darkblue}\bullet$};
      \node at (0.5,.07) {$\color{darkblue}\scriptstyle{\ell-2r}$};
\end{tikzpicture}
}
\equiv
\sum_{s=0}^{\lfloor\ell/2\rfloor}
z_s
\sum_{r=s}^{\lfloor\ell/2\rfloor} 
\mathord{
\begin{tikzpicture}[baseline = 0]
	\draw[<-,thick,darkblue] (0.08,-.3) to (0.08,.4);
      \node at (0.08,0.07) {$\color{darkblue}\bullet$};
      \node at (0.5,.07) {$\color{darkblue}\scriptstyle{\ell-2r}$};
\end{tikzpicture}
}
\mathord{\begin{tikzpicture}[baseline = -1mm]
  \draw[-,thick,darkblue] (0,0.2) to[out=180,in=90] (-.2,0);
  \draw[->,thick,darkblue] (0.2,0) to[out=90,in=0] (0,.2);
 \draw[-,thick,darkblue] (-.2,0) to[out=-90,in=180] (0,-0.2);
  \draw[-,thick,darkblue] (0,-0.2) to[out=0,in=-90] (0.2,0);
   \node at (0.92,0) {$\color{darkblue}\scriptstyle{2(r-s)-1}$};
      \node at (0.2,0) {$\color{darkblue}\bullet$};
\end{tikzpicture}
}\\
&
=
\sum_{s=0}^{\lfloor\ell/2\rfloor}
z_s\:
\mathord{
\begin{tikzpicture}[baseline = -0.5mm]
	\draw[-,thick,darkblue] (0,0.6) to (0,0.3);
	\draw[-,thick,darkblue] (0,0.3) to [out=-90,in=180] (.3,-0.2);
	\draw[-,thick,darkblue] (0.3,-0.2) to [out=0,in=-90](.5,0);
	\draw[-,thick,darkblue] (0.5,0) to [out=90,in=0](.3,0.2);
	\draw[-,thick,darkblue] (0.3,.2) to [out=180,in=90](0,-0.3);
	\draw[->,thick,darkblue] (0,-0.3) to (0,-0.6);
   \node at (0.9,0.02) {$\color{darkblue}\scriptstyle{\ell-2s}$};
      \node at (0.48,.02) {$\color{darkblue}\bullet$};
\end{tikzpicture}
} \equiv 0\pmod{\mathcal I_{f,f'}},
\end{align*}
where for the last equality we have used \cref{prop: AOBC slides} repeatedly to pull the $\ell-2s$
dots on the right curl down past the crossing, plus \cref{prop: bubble reducing} to see many of the dotted bubbles produced are zero.
Now we have shown that the left tensor ideal generated by (\ref{set1})
contains \cref{set2}. A similar argument shows that the left
tensor ideal generated by \cref{set2} contains \cref{set1},
completing the proof of the lemma.
\end{proof}

\begin{definition}
Define the {\em cyclotomic oriented Brauer-Clifford supercategory}
assocated to the polynomials $f$ and $f'$ fixed above
to be the $\K$-linear supercategory $\OBC^{f,f'}$
that is the quotient of $\AOBC_\K$ by the $\K$-linear left tensor
ideal $\mathcal I_{f,f'}$ from \cref{l1}.
\end{definition}

Our goal is to establish a basis theorem for the morphism spaces in 
$\OBC^{f,f'}$.
As we will explain fully later on, the 
cyclotomic oriented Brauer-Clifford supercategory $\OBC^f$ is a special case, so that \cref{Cyclotomic basis conjecture} will follow from this more
general result.

Continuing to work over $\K$,
recall from \cref{C:degenerateaffineHeckeCliffordsuperalgebraIsom} there is a $\K$-superalgebra
homomorphism
\begin{equation}
\alpha:\AS_n \rightarrow \End_{\AOBC_\K}(\up^{\otimes n})
\end{equation}
sending $x_i$ to the closed dot on the $i$th strand, $c_i$ to the open
dot on the $i$th strand, and $s_i$ to the crossing of the $i$th and
$(i+1)$th strands (numbering strands $1,\dots,n$ from right to left).
By \cref{AOBC basis theorem}, 
the map
\begin{equation}
\alpha \otimes \beta: \AS_n \otimes_{\K} \Sym
\rightarrow \End_{\AOBC_\K}(\up^{\otimes n})
\end{equation}
is a superalgebra isomorphism.
Let $\Serg_n^f$ be the 
{\em cyclotomic Sergeev superalgebra} 
from \cite[Section 3-e]{BK3}, namely,  the
quotient of $\AS_n$ by the two-sided ideal
generated by $f(x_1)$.
Composing $\alpha$ with the canonical quotient 
map $\Pi:\End_{\AOBC_\K}(\up^{\otimes n}) \twoheadrightarrow \End_{\OBC^{f,f'}}(\up^{\otimes n})$
gives a well-defined $\K$-superalgebra homomorphism
\begin{equation}
\gamma:\Serg_n^f \rightarrow \End_{\OBC^{f,f'}}(\up^{\otimes n}).
\end{equation}
The following is the key to all our subsequent arguments.

\begin{lemma}\label{l2}
$\gamma$ is an isomorphism.
\end{lemma}

\begin{proof}
Let $\pi:\AS_n \otimes_{\K} \Sym \twoheadrightarrow \Serg_n^f$ be the $\K$-superalgebra
homomorphism that sends 
$a \otimes 1$
to the canonical image of $a$ in 
$\Serg^f_n$,
and $1 \otimes h_r$ to $(-1)^r \delta_r$.
Note that $\ker \pi$ is $I \otimes \Sym + \AS_n \otimes J$
where $I$ is the two-sided ideal of $\AS_n$ generated by $f(x_1)$
and $J$ is the two-sided ideal of $\Sym$ generated by
$h_r- (-1)^r \delta_r$ for $r \geq 1$.
Also let $\Pi:\End_{\AOBC_\K}(\up^{\otimes n})
\twoheadrightarrow \End_{\OBC^{f,f'}}(\up^{\otimes n})$
be the canonical quotient map as above.
By directly checking it on generators, one sees that the following diagram commutes:
$$
\begin{CD}
\AS_n \otimes_{\K}
\Sym&@>\sim>\alpha\otimes\beta>&\End_{\AOBC_\K}(\up^{\otimes n})\\
@V\pi VV&&@VV\Pi V\\
\Serg_n^f&@>>\gamma>&\End_{\OBC^{f,f'}}(\up^{\otimes n}).
\end{CD}
$$
It follows immediately that $\gamma$ is surjective. Moreover,
the injectivity of $\gamma$ follows if we can show that
$(\alpha\otimes\beta)^{-1}(\ker \Pi) \subseteq \ker \pi$.

By the definition of $\OBC^{f,f'}$,
$\ker \Pi$ is the subspace of
$\End_{\AOBC_\K}(\up^{\otimes n})$ 
defined by the left tensor ideal $\mathcal I_{f,f'}$. This
means that any
element of $\ker \Pi$ is a $\K$-linear combination of morphisms
$\theta:\up^{\otimes n} \rightarrow
\up^{\otimes n}$ in $\AOBC_\K$ 
of the form $$
\theta = 
\sigma \circ (\lambda \otimes \rho) \circ \tau
=
\mathord{
\begin{tikzpicture}[baseline = -.5mm]
\draw[darkblue,thick,yshift=-4pt,xshift=-4pt] (.5,0) rectangle ++(8pt,8pt);
\node at (.5,0) {$\scriptstyle \rho$};
\draw[darkblue,thick,yshift=-4pt,xshift=-4pt] (-.33,0) rectangle ++(16pt,8pt);
\node at (-.2,0.02) {$\scriptstyle \lambda$};
\draw[darkblue,thick,yshift=-4pt,xshift=-4pt] (-.34,-.5) rectangle ++(32pt,8pt);
\node at (0.1,-.5) {$\scriptstyle \tau$};
\draw[darkblue,thick,yshift=-4pt,xshift=-4pt] (-.34,.5) rectangle ++(32pt,8pt);
\node at (0.1,.5) {$\scriptstyle \sigma$};
\draw[->,thick,darkblue] (-.34,0.65) to (-.34,.9);
\node at (0.1,.8) {$\scriptstyle \cdots$};
\draw[->,thick,darkblue] (.49,0.65) to (.49,.9);
\draw[-,thick,darkblue] (-.34,-0.65) to (-.34,-.9);
\node at (0.1,-.8) {$\scriptstyle \cdots$};
\draw[-,thick,darkblue] (.49,-0.65) to (.49,-.9);
\draw[-,darkblue] (.49,-0.35) to (.49,-.15);
\draw[-,darkblue] (.49,0.35) to (.49,.15);
\draw[-,thick,darkblue] (-.34,-0.35) to (-.34,-.15);
\draw[-,thick,darkblue] (-.34,0.35) to (-.34,.15);
\draw[-,thick,darkblue] (-.04,-0.35) to (-.04,-.15);
\draw[-,thick,darkblue] (-.04,0.35) to (-.04,.15);
\end{tikzpicture}
}\:.
$$
where $\rho$ is one of the generating morphisms
$f\left(\mathord{
\begin{tikzpicture}[baseline = -0.5mm]
	\draw[->,thick,darkblue] (0.08,-.2) to (0.08,.3);
      \node at (0.08,0.05) {$\color{darkblue}\bullet$};
\end{tikzpicture}
}\right)$ or $\beta(h_r) - (-1)^r \delta_r$
for $\mathcal I_{f,f'}$, and $\sigma,\tau,\lambda$ are any
other morphisms so that the compositions make sense.
Thus, we must show that the inverse image under $\alpha\otimes\beta$ of 
such a morphism $\theta=\sigma \circ (\lambda \otimes \rho) \circ \tau$ 
lies in $\ker \pi = I \otimes \Sym + \AS_n \otimes J$.

If $\rho = \beta(h_r) - (-1)^r \delta_r$ for some
$r$
then $(\alpha\otimes\beta)^{-1}(\theta)$ 
obviously lies in $\AS_n \otimes J$.
Instead, suppose that $\rho = f\left(\mathord{
\begin{tikzpicture}[baseline =-0.5mm]
	\draw[->,thick,darkblue] (0.08,-.2) to (0.08,.3);
      \node at (0.08,0.05) {$\color{darkblue}\bullet$};
\end{tikzpicture}
}\right)$.
Using the relations established earlier in the paper (especially \cref{prop: AOBC slides} and \cref{prop: bubble reducing}), 
we ``straighten'' the diagram $\theta$
leaving the $\rho$-coupon on the right edge fixed, to rewrite it as a
$\K$-linear combination of morphisms of the following two types:
\begin{itemize}
\item[(I)]
$
\mathord{
\begin{tikzpicture}[baseline = -.5mm]
\draw[darkblue,thick,yshift=-4pt,xshift=-4pt] (.5,0) rectangle ++(8pt,8pt);
\node at (.5,0) {$\scriptstyle \rho$};
\draw[darkblue,thick,yshift=-4pt,xshift=-4pt] (-.34,-.5) rectangle ++(32pt,8pt);
\node at (0.1,-.5) {$\scriptstyle \tau'$};
\draw[darkblue,thick,yshift=-4pt,xshift=-4pt] (-.34,.5) rectangle ++(32pt,8pt);
\node at (0.1,.5) {$\scriptstyle \sigma'$};
\draw[->,thick,darkblue] (-.34,0.65) to (-.34,.9);
\node at (0.1,.8) {$\scriptstyle \cdots$};
\draw[->,thick,darkblue] (.49,0.65) to (.49,.9);
\draw[-,thick,darkblue] (-.34,-0.65) to (-.34,-.9);
\node at (0.1,-.8) {$\scriptstyle \cdots$};
\draw[-,thick,darkblue] (.49,-0.65) to (.49,-.9);
\draw[-,thick,darkblue] (.49,-0.35) to (.49,-.15);
\draw[-,thick,darkblue] (.49,0.35) to (.49,.15);
\draw[-,thick,darkblue] (-.34,-0.35) to (-.34,.35);
\draw[-,thick,darkblue] (.14,0.35) to (.14,-.35);
\node at (-0.1,0) {$\scriptstyle \cdots$};
\end{tikzpicture}
} \otimes \delta$
for $\sigma',\tau'\in \operatorname{Im} \alpha$ and $\delta \in \operatorname{Im} \beta$;
\item[(II)]  $\mathord{
\begin{tikzpicture}[baseline = -.5mm]
\draw[darkblue,thick,yshift=-4pt,xshift=-4pt] (-.34,0) rectangle ++(32pt,8pt);
\node at (0.1,0) {$\scriptstyle \lambda'$};
\draw[->,thick,darkblue] (-.34,0.15) to (-.34,.4);
\node at (0.1,.3) {$\scriptstyle \cdots$};
\draw[->,thick,darkblue] (.49,0.15) to (.49,.4);
\draw[-,thick,darkblue] (-.34,-0.15) to (-.34,-.4);
\node at (0.1,-.3) {$\scriptstyle \cdots$};
\draw[-,thick,darkblue] (.49,-0.15) to (.49,-.4);
\end{tikzpicture}
}\:\:\:
\mathord{\begin{tikzpicture}[baseline = -0.2mm]
  \draw[->,thick,darkblue] (0,0.3) to[out=180,in=90] (-.3,0);
  \draw[-,thick,darkblue] (0.3,0) to[out=90,in=0] (0,.3);
 \draw[-,thick,darkblue] (-.3,0) to[out=-90,in=180] (0,-0.3);
  \draw[-,thick,darkblue] (0,-0.3) to[out=0,in=-90] (0.3,0);
\filldraw[white,thick,yshift=-4pt,xshift=-4pt] (0.28,-0.05) rectangle ++(7.8pt,7.6pt);
\draw[darkblue,thick,yshift=-4pt,xshift=-4pt] (0.28,-0.06) rectangle ++(8pt,8pt);
\node at (0.28,-0.05) {$\scriptstyle\rho$};
   \node at (0.75,.35) {$\color{darkblue}\scriptstyle{2r-\ell-1}$};
      \node at (0.16,.25) {$\color{darkblue}\bullet$};
\end{tikzpicture}
}\otimes \delta$
for $\lambda'\in \operatorname{Im}(\alpha), \delta \in
\operatorname{Im}(\beta)$ and $r > \lfloor \ell/2\rfloor$.
\end{itemize}
These morphisms 
arise when the
$\rho$-coupon ends up on a propagating strand 
(type I)
or on a dotted bubble (type II)
after straightening.
The inverse image under $\alpha\otimes\beta$ of a
type I morphism 
lies in $I \otimes \Sym$.
The inverse image under $\alpha\otimes\beta$ of 
a
type II morphism lies in $\AS_n \otimes J$
because
\begin{align*}
\beta^{-1}\left(\mathord{\begin{tikzpicture}[baseline = -0.2mm]
  \draw[->,thick,darkblue] (0,0.3) to[out=180,in=90] (-.3,0);
  \draw[-,thick,darkblue] (0.3,0) to[out=90,in=0] (0,.3);
 \draw[-,thick,darkblue] (-.3,0) to[out=-90,in=180] (0,-0.3);
  \draw[-,thick,darkblue] (0,-0.3) to[out=0,in=-90] (0.3,0);
\filldraw[white,thick,yshift=-4pt,xshift=-4pt] (0.28,-0.05) rectangle ++(7.8pt,7.6pt);
\draw[darkblue,thick,yshift=-4pt,xshift=-4pt] (0.28,-0.06) rectangle ++(8pt,8pt);
\node at (0.28,-0.05) {$\scriptstyle\rho$};
   \node at (0.78,.35) {$\color{darkblue}\scriptstyle{2r-\ell-1}$};
      \node at (0.16,.25) {$\color{darkblue}\bullet$};
\end{tikzpicture}
}\right)
&=
\sum_{s =0}^{\lfloor \ell/2\rfloor}  (-1)^{r-s} z_s h_{r-s}
\\&
\!\!\!\stackrel{(\ref{e2})}{\displaystyle=}
\sum_{s =0}^{\lfloor\ell/2\rfloor}  (-1)^{r-s} z_s \left(h_{r-s} - (-1)^{r-s} \delta_{r-s}\right)
\in J.
\end{align*}
\end{proof}

Now we can prove the main result of this section. Recall the
definition of normally ordered dotted oriented Brauer-Clifford diagram
from \cref{D:N.O. dotted OBC diagram}.

\begin{theorem}\label{thm}
For any $\mathtt{a}, \mathtt{b} \in \operatorname{ob} \OBC^{f,f'}$,
the morphism space $\Hom_{\OBC^{f,f'}}(\mathtt{a}, \mathtt{b})$
is a free $\K$-module with basis given by equivalence classes of
normally ordered 
dotted oriented Brauer-Clifford diagrams (without bubbles) of type
$\mathtt{a}\rightarrow\mathtt{b}$ 
having fewer than $\ell$
closed dots on each strand.
\end{theorem}

\begin{proof}
Let us first show that the given diagrams {\em span}
$\Hom_{\OBC^{f,f'}}(\mathtt{a}, \mathtt{b})$. In fact, we show by
induction on $N$ that any diagram representing a morphism
$\mathtt{a} \rightarrow \mathtt{b}$ in $\OBC^{f,f'}$ with $N$ closed
dots 
can be written as a linear combination of the ones among the specified
diagrams which 
have 
$N$ or fewer closed dots. The case $N=0$ is straightforward (and also follows
from \cref{OBC basis theorem}).
In general, take a diagram with $N > 0$ closed dots. 
Using the relations, dots can be moved past crossings or other dots
(possibly introducing a sign in the case of open dots) 
modulo diagrams with strictly fewer closed dots. In particular, any dotted bubble can be moved to the right
hand side of the picture modulo diagrams with strictly fewer closed dots and,
once on the
right hand side, it
may be replaced by a scalar using \cref{prop: bubble reducing} and the
last part of \cref{l1}. To complete the proof of the spanning
part of the theorem, it remains to observe that if any strand has
$\ell$ closed dots, it can be rewritten in terms of of diagrams
with strictly fewer closed dots:
for any objects $\mathtt{a}$ and $\mathtt{b}$,
the relations $\mathtt{b}\otimes\mathtt{a} \otimes f\left(\mathord{
\begin{tikzpicture}[baseline = -.5mm]
	\draw[->,thick,darkblue] (0.08,-.2) to (0.08,.3);
      \node at (0.08,0.05) {$\color{darkblue}\bullet$};
\end{tikzpicture}
}\right) =0$ and $\mathtt{b}\otimes\mathtt{a} \otimes f'\left(\mathord{
\begin{tikzpicture}[baseline = -.5mm]
	\draw[<-,thick,darkblue] (0.08,-.2) to (0.08,.3);
      \node at (0.08,0.1) {$\color{darkblue}\bullet$};
\end{tikzpicture}
}\right) =0$
 in $\OBC^{f,f'}$ imply 
that 
$$
\mathord{
\begin{tikzpicture}[baseline = -.5mm]
\node at (0.5,-.7) {$\scriptstyle \mathtt{a}$};
\draw[-,ultra thick,darkblue] (.5,-0.5) to (.5,.5);
\node at (-0.5,-.7) {$\scriptstyle \mathtt{b}$};
\node at (0,0) {$\scriptstyle \color{darkblue}\bullet$};
\node at (0.2,0) {$\scriptstyle \color{darkblue}\ell$};
\draw[-,ultra thick,darkblue] (-.5,-0.5) to (-.5,.5);
\draw[->,thick,darkblue] (0,-0.5) to (0,.5);
\end{tikzpicture}
}
\equiv
\mathord{
\begin{tikzpicture}[baseline = -.5mm]
\node at (0.5,-.7) {$\scriptstyle \mathtt{a}$};
\draw[-,ultra thick,darkblue] (.5,-0.5) to (.5,.5);
\node at (-0.5,-.7) {$\scriptstyle \mathtt{b}$};
\node at (0.7,0) {$\scriptstyle \color{darkblue}\bullet$};
\node at (0.9,0) {$\scriptstyle \color{darkblue}\ell$};
\draw[-,ultra thick,darkblue] (-.5,-0.5) to (-.5,.5);
\draw[-,thick,darkblue] (0,-0.5) to [out=90,in=-90] (.7,0);
\draw[->,thick,darkblue] (0.7,0) to [out=90,in=-90] (0,0.5);
\end{tikzpicture}
}
\equiv 0
\:,\qquad
\mathord{
\begin{tikzpicture}[baseline = -.5mm]
\node at (0.5,-.7) {$\scriptstyle \mathtt{a}$};
\draw[-,ultra thick,darkblue] (.5,-0.5) to (.5,.5);
\node at (-0.5,-.7) {$\scriptstyle \mathtt{b}$};
\node at (0,0) {$\scriptstyle \color{darkblue}\bullet$};
\node at (0.2,0) {$\scriptstyle \color{darkblue}\ell$};
\draw[-,ultra thick,darkblue] (-.5,-0.5) to (-.5,.5);
\draw[<-,thick,darkblue] (0,-0.5) to (0,.5);
\end{tikzpicture}
}
\equiv
\mathord{
\begin{tikzpicture}[baseline = -.5mm]
\node at (0.5,-.7) {$\scriptstyle \mathtt{a}$};
\draw[-,ultra thick,darkblue] (.5,-0.5) to (.5,.5);
\node at (-0.5,-.7) {$\scriptstyle \mathtt{b}$};
\node at (0.7,0) {$\scriptstyle \color{darkblue}\bullet$};
\node at (0.9,0) {$\scriptstyle \color{darkblue}\ell$};
\draw[-,ultra thick,darkblue] (-.5,-0.5) to (-.5,.5);
\draw[<-,thick,darkblue] (0,-0.5) to [out=90,in=-90] (.7,0);
\draw[-,thick,darkblue] (0.7,0) to [out=90,in=-90] (0,0.5);
\end{tikzpicture}
}
\equiv 0
\:,
$$
where $\equiv$ means ``equal modulo a linear combination of diagrams with
$< \ell$ closed dots.''

It remains to establish linear independence.
Note to start with that the result is true in the special case
$\mathtt{a} = \mathtt{b}  = \up^{\otimes n}$, for in this case it follows using
Lemma~\ref{l2} and the basis theorem for $\Serg_n^f$ established 
 in \cite[$\S$3-e]{BK3}.
In general, we first reduce to the case $\mathtt{b} = \unit$
using $\K$-module isomorphism
\begin{align*}
\Hom_{\OBC^{f,f'}}(\mathtt{a}, \mathtt{b})
&\stackrel{\sim}{\rightarrow}
\Hom_{\OBC^{f,f'}}(\mathtt{b}^{*}\otimes\mathtt{a}, \unit),
\qquad
&\mathord{
\begin{tikzpicture}[baseline = -.5mm]
\draw[darkblue,thick,yshift=-4pt,xshift=-4pt] (0,0) rectangle ++(12pt,8pt);
\node at (0.07,0) {$\scriptstyle \lambda$};
\node at (0.07,-.52) {$\scriptstyle \mathtt{a}$};
\node at (0.07,.54) {$\scriptstyle \mathtt{b}$};
\draw[-,ultra thick,darkblue] (0.07,-0.15) to (0.07,-.4);
\draw[-,ultra thick,darkblue] (0.07,0.15) to (0.07,.4);
\end{tikzpicture}
}
&\mapsto
\mathord{
\begin{tikzpicture}[baseline = -.5mm]
\draw[darkblue,thick,yshift=-4pt,xshift=-4pt] (0,0) rectangle ++(12pt,8pt);
\node at (0.07,0) {$\scriptstyle \lambda$};
\node at (0.07,-.52) {$\scriptstyle \mathtt{a}$};
\node at (-0.4,-.49) {$\scriptstyle \mathtt{b}^{*}$};
\draw[-,ultra thick,darkblue] (0.07,-0.15) to (0.07,-.4);
\draw[-,ultra thick,darkblue] (0.07,0.15) to[out=90,in=0] (-0.2,.35);
\draw[-,ultra thick,darkblue] (-0.43,0.15) to[out=90,in=180] (-0.2,.35);
\draw[-,ultra thick,darkblue] (-0.43,0.15) to[out=-90,in=90] (-0.43,-.4);
\end{tikzpicture}
}\:.
\end{align*}
This same reduction proves the theorem in the case $\mathtt{a} =
\down^{\otimes n}\otimes  \up^{\otimes n}, \mathtt{b}  =
\unit$, since it follows from the special case
$\mathtt{a}=\mathtt{b} = \up^{\otimes n}$
treated already.

Now suppose that $\mathtt{a}$ is arbitrary and $\mathtt{b} =
\unit$.
The space $\Hom_{\OBC^{f,f'}}(\mathtt{a}, \unit)$ is zero unless
$\mathtt{a}$ has $n$ letters equal to $\up$ and $n$ letters equal to
$\down$ for some $n \geq 0$. Assuming that is the case, the object
$\mathtt{a}$ is a ``shuffle'' of the tensor $\down^{\otimes n}
\otimes\up^{\otimes n}$ already
treated. Consider the minimal length permutation of tensor factors taking $\mathtt{a}$
to $\down^{\otimes n} \otimes\up^{\otimes n}$. There is corresponding isomorphism
$w:\mathtt{a} \stackrel{\sim}{\rightarrow} \down^{\otimes n}\otimes
\up^{\otimes n}$ in
$\OBC^{f,f'}$ obtained by composing various rightward crossings. Hence, we get another $\K$-module isomorphism
\begin{align*}
\Hom_{\OBC^{f,f'}}(\down^{\otimes n}\otimes \up^{\otimes n}, \unit) 
&\stackrel{\sim}{\rightarrow}
\Hom_{\OBC^{f,f'}}(\mathtt{a}, \unit),
&&
\theta \mapsto w \circ \theta.
\end{align*}
Applying this isomorphism to the basis for 
$\Hom_{\OBC^{f,f'}}(\down^{\otimes n}\otimes \up^{\otimes n}, \unit)$
already obtained at the end of the previous paragraph, we obtain a
basis for
$\Hom_{\OBC^{f,f'}}(\mathtt{a}, \unit)$. It is not quite the
same as the basis of normally ordered diagrams that we are after, but we can
slide dots to the incoming ends of the strands to obtain the
desired basis from it (up to some signs), modulo diagrams with
strictly fewer closed dots, just like in
the opening paragraph of this proof. This means that
the transition matrix between the basis in hand and the basis in mind 
is unitriangular when suitably ordered, which is all
that is needed to complete the proof.
\end{proof}

Finally we explain the connection to the cyclotomic quotients $\OBC^f$.
For this, we specialize to the case that the monic polynomial $f(u)$
from \cref{f1} has its coefficients in the original ground field $\k$.
Also 
let
\begin{equation}\label{four}
f'(u) := u^\ell + z_1' u^{\ell-2} + \cdots \in \K[u]
\end{equation}
where $\K := 
\k[z_1',\dots,z_{\lfloor \ell/2\rfloor}']$ for {\em indeterminates}
$z_1',\dots,z_{\lfloor \ell/2\rfloor}'$.
Define $\delta(u), \delta'(u)$ from \cref{ps3,ps4}.
By the identity \cref{e1},
the coefficients $\delta_1,\dots,\delta_{\lfloor \ell/2 \rfloor}$ of
$\delta(u)$
are related to the indeterminates $z_1',\dots, z_{\lfloor
  \ell/2\rfloor}'$ by a unitriangular transition matrix. Thus,
$\K$ is also freely generated by 
$\delta_1,\dots,\delta_{\lfloor \ell/2 \rfloor}$. In $\OBC^{f,f'}$,
the scalar
% $(-1)^r \delta_r$ 
$\delta_r$ acts in the same way
as tensoring on the right with
 the counterclockwise bubble with $2r-1$ dots.
So we get the following corollary 
immediately from \cref{thm}.

\begin{corollary}\label{kyle}
Assume $f(u) \in \k[u]$ and $f'(u) \in \K[u]$ as in \cref{four}. 
For objects $\mathtt{a}, \mathtt{b}$ in $\OBC^{f,f'}$,
the morphism space $\Hom_{\OBC^{f,f'}}(\mathtt{a},
\mathtt{b})$
has basis as a vector space over $\k$ given by equivalence classes of 
normally ordered 
dotted oriented Brauer-Clifford diagrams with bubbles of type
$\mathtt{a}\rightarrow\mathtt{b}$
having fewer than $\ell$
closed dots on each strand.
\end{corollary}

Continuing with $f(u) \in \k[u]$, 
the $\k$-linear supercategory $\OBC^f$ 
is the quotient of $\AOBC$ by
the $\k$-linear left tensor ideal generated by 
$f\left(\mathord{
\begin{tikzpicture}[baseline = -0.5mm]
	\draw[->,thick,darkblue] (0.08,-.2) to (0.08,.3);
      \node at (0.08,0.05) {$\color{darkblue}\bullet$};
\end{tikzpicture}}\right)$.
The composition of the natural 
$\k$-linear superfunctor $\AOBC \rightarrow \AOBC_\K$
followed by the quotient functor $\AOBC_\K \rightarrow \OBC^{f,f'}$
factors through $\OBC^f$ to induce a $\k$-linear superfunctor
\begin{equation}\label{cyclotomic quotient isomorphism}
% \Psi:
\OBC^f \rightarrow \OBC^{f,f'}.
\end{equation}
This is the identity on objects. On morphisms, it is noted already in \cref{SS:cyclotomic bases theorems}  that $\Hom_{\OBC^f}(\mathtt{a}, \mathtt{b})$ is
spanned as a vector space over $\k$ by the 
 equivalence classes of 
normally ordered 
dotted oriented Brauer-Clifford diagrams with bubbles of type
$\mathtt{a}\rightarrow\mathtt{b}$
having fewer than $\ell$
closed dots on each strand; the proof is the same argument as in the
first paragraph of the proof of Theorem~\ref{thm}.
The images of these spanning morphisms under the superfunctor \cref{cyclotomic quotient isomorphism} give a basis for
$\Hom_{\OBC^{f,f'}}(\mathtt{a}, \mathtt{b})$ thanks to
Corollary~\ref{kyle}.
Hence, they are already linearly independent in
$\Hom_{\OBC^f}(\mathtt{a}, \mathtt{b})$, and \cref{cyclotomic quotient isomorphism} is an isomorphism
on morphism spaces.
We have proved:

\begin{corollary}\label{kyle2}
The $\k$-linear superfunctor \cref{cyclotomic quotient isomorphism} is an isomorphism.
\end{corollary}

\cref{kyle,kyle2} together imply \cref{Cyclotomic basis conjecture}.

\bibliographystyle{alphanum}    
\bibliography{references}   

\end{document}